\UseRawInputEncoding
\documentclass[a4paper,11pt]{article}
\usepackage{amsmath,amsfonts,amssymb,amsthm}
\usepackage{graphicx,subfigure,booktabs,multirow}
\usepackage{mathrsfs,enumerate,color}
\usepackage{algorithm,algpseudocode}
\usepackage[textwidth=16cm]{geometry}
\usepackage{multirow}
\usepackage{epstopdf}
\usepackage{graphicx}
\usepackage{epstopdf}
\usepackage{color}
\usepackage{bm}

\newtheorem{theorem}{Theorem}[section]
\newtheorem{lemma}[theorem]{Lemma}
\newtheorem{proposition}[theorem]{Proposition}

\newtheorem{remark}[theorem]{Remark}
\newtheorem{definition}{Definition}[section]

\newcommand{\ten}[1]{\mathcal{{#1}}}
\newcommand{\tenhat}[1]{\widehat{\mathcal{{#1}}}}
\newcommand{\R}{\mathbb{R}}

\newcommand{\Z}{\mathbb{Z}}
\newcommand{\PP}{\mathbb{P}}
\newcommand{\E}{\mathbb{E}}
\newcommand{\<}{\langle}
\renewcommand{\>}{\rangle}
\newcommand{\vct}[1]{\boldsymbol{#1}}
\newcommand{\mtx}[1]{\boldsymbol{#1}}
\newcommand{\tub}[1]{\mathring{\vct{#1}}}
\newcommand{\tc}[1]{\vec{\vct{#1}}}

\newcommand{\rank}{\operatorname{rank}}
\newcommand{\tr}{\operatorname{Tr}}

\def\Eabc{\mathcal{E}_{abc}}
\def\Eabck{\mathcal{E}_{a_kb_kc_k}}

\def\Zabc{\mathcal{Z}_{abc}}
\def\Zabck{\mathcal{Z}_{a_kb_kc_k}}
\def\be{\bm{e}}

\def\Z{\mathcal{Z}}
\DeclareMathOperator{\Pj}{\mathrm{P}}
\DeclareMathOperator{\I}{\mathrm{I}}
\DeclareMathOperator{\RR}{\mathrm{R}}
\DeclareMathOperator{\TT}{\mathrm{T}}

\newcommand{\PT}{P_{\ten{T}}}

\newcommand{\PO}{R_{\Omega}}

\newcommand{\OpId}{\mathcal{I}}
\newcommand{\dct}{ \mbox{\tt dct} }
\newcommand{\idct}{ \mbox{\tt idct} }

\newcommand{\vv}{ \mbox{\tt vec} }
\newcommand{\bdiag}{\mbox{\tt blockdiag}}

\def\lcb{\left\{}
\def\rcb{\right\}}

\graphicspath{{fig1/}}

\title{Riemannian Conjugate Gradient Descent Method for Third-Order Tensor Completion}
\author{
Guang-Jing Song\thanks{School of Mathematics and Information Sciences,
Weifang University, Weifang 261061, P.R. China (e-mail: sgjshu@163.com).}
\and
Xue-Zhong Wang\thanks{School of Mathematics and Statistics, Hexi University, P.R. China
(e-mail: xuezhongwang77@126.com).}
\and
Michael K. Ng\thanks{The Corresponding Author. Department of Mathematics, The University of
Hong Kong (e-mail: mng@maths.hku.hk).
Research supported in part by HKRGC GRF 12306616, 12200317, 12300218 and 12300519.}}
\date{}

\begin{document}

\maketitle

\noindent
\textbf{Abstract:} The goal of tensor completion is to
fill in missing entries of a partially known tensor under a low-rank
constraint. In this paper, we mainly study low rank third-order tensor
completion problems by using Riemannian optimization methods on the smooth
manifold. Here the tensor rank is defined to be a set of matrix ranks
where the matrices are the slices of the transformed tensor
obtained by applying the Fourier-related transformation onto the tubes of the original
tensor. We show that with suitable incoherence conditions on the underlying low rank tensor,
the proposed Riemannian optimization method is guaranteed to converge and find
such low rank tensor with a high probability. In addition, numbers of sample entries required for
solving low rank tensor completion problem under different initialized methods are studied and derived.
Numerical examples for both synthetic and image data sets
are reported to demonstrate the proposed method is able to recover low rank tensors.

\vspace{2mm} \noindent \textbf{Keywords:} tensor completion, manifold, tangent spaces, conjugate gradient
descent method

\vspace{2mm} \noindent \textbf{AMS Subject Classifications 2010:}
15A69, 15A83, 90C25

\section{Introduction}

This paper addresses the problem of low rank tensor completion
when the rank is a priori known or estimated. Let
$\mathcal{A}\in \R^{n_{1}\times n_{2}\times n_{3}}$ be a third-order tensor that is only known on a
subset $\Omega$ of the complete set of entries. The low rank tensor
completion problem consists of finding the tensor with the lowest rank
that agrees with $\mathcal{A}$ on $\Omega$:
\begin{equation} \label{n1}
\min\limits_{\mathcal{Z}} ~~ \rank(\mathcal{Z}) ~~ \text{s.t.} ~~ \mathcal{Z} \in \R^{n_{1} \times n_{2}\times
n_{3}}, ~~ P_{\Omega}(\mathcal{Z})=P_{\Omega}(\mathcal{A}),
\end{equation}
where
\begin{align*}
P_{\Omega}:\R^{n_{1}\times n_{2}\times n_{3}}\rightarrow
\R^{n_{1}\times n_{2}\times n_{3}},~\mathcal{Z}_{ijk} \mapsto
\left\{\begin{array}{lll}
\mathcal{Z}_{ijk}, & \mbox{if} \ (i,j,k)\in \Omega,\\
0, & \mbox{if} \ (i,j,k)\notin \Omega,\\
\end{array}\right.
\end{align*}
denotes the projection onto $\Omega$. In particular, if $n_{3}=1$
the tensor completion problem (\ref{n1}) reduces to the well
known matrix completion problem, see for instance
\cite{candes2010matrix,candes2009exact,candes2010the,gross2010,recht2010guaranteed}
and references therein.

Tensor has many kinds of rank definitions which lead to different mathematical models for studying the low rank tensor completion problem.
In order to well understand the model \eqref{n1}, we first summarize some most popular definitions of tensor rank.  Let $\ten{X}\in \mathbb{R}^{n_1\times\cdots\times n_d}.$
The CANDECOMP/PARAFAC (CP) rank of a tensor $\ten{X}$ is defined as the minimal number
of summations of rank-one tensors  that generate $\ten{X}$. CP rank is an exact analogue to the matrix rank, however, its properties are quite different.
For example, calculating the tensor (CP) rank has been  demonstrated to be NP-hard \cite{Kolda2009,Hastad1990}.
The  Tucker rank \cite{Tucker66} (or multilinear rank) of $\ten{X}$ is defined as
$\rank_{n}(\ten{X})=(\rank(X_{1}),\rank(X_{2}),...,\rank(X_{d})),$ where $X_{{i}}$  is derived by unfolding  $\ten{X}$ along its $i$-th mode into a matrix with size $n_{i}\times \Pi_{k\neq i}n_{k}$.
Although multilinear rank  is perhaps the most widely adopted rank assumption in the existing tensor completion literature, a crucial drawback is pointed out that it only takes into consideration the ranks of matrices that are constructed based on the unbalanced matrixization scheme \cite{phien}. The Tensor Train (TT)
rank \cite{oseledets2011tensor} of  $\ten{X}$ is defined as a vector  $\textbf{r}=[r_1,...,r_{d-1}]$ with $r_{j}=\rank(X^{<j>}), $ where $X^{<j>}$ is the $j$-th unfolding of $\ten{X}$ as a matrix with size $(n_1\cdots n_j)\times (n_{j+1}\cdots n_d).$ TT-rank is generated by the TT decomposition using the link structure of each core tensor. Since the link structure, the TT rank is only  efficient for higher order tensor completion problem.

Optimization technique on Riemannian manifold has gained increasing popularity in recent years. Meanwhile,
some classical optimization algorithms that worked well in the Euclidean
space have been extended to the smooth manifolds. For instance, gradients
descent method, conjugate gradients method and Riemannian
trust-region method can be considered, see \cite{Absil,edelman,Kressner,bart,wei,wei2}  and the references therein. For the low rank tensor completion problem, the Riemannian manifold theory can provide us an alternative way to consider the rank constraint condition.
Note that the fixed tensor set can form a smooth manifold, then the model given in (\ref{n1}) can be equivalent rewritten as
\begin{equation} \label{n4}
\text{min} ~~
\frac{1}{2}\|P_{\Omega}(\mathcal{Z})-P_{\Omega}(\mathcal{A})\|_{F}^2 ~~ \text{s.t.} ~~
\mathcal{Z} \in \mathcal{M}_{\textbf{r}}:=\{\mathcal{Z}\in \R^{n_{1}\times
n_{2}\times n_{3}}~|~\rank{(\mathcal{Z})}=\textbf{r}\},
\end{equation}
with $P_{\Omega}$ is a projection to the sampling set $\Omega$, where $\Omega$ ($|\Omega|=m$) is a set of indices sampled
independently and uniformly without replacement. When the rank in \eqref{n4} is chosen as the Tucker rank, Kressner et al. \cite{Kressner} showed the
fixed Tucker rank  tensor set form a smooth embedded submanifold of $\R^{n_{1}\times n_{2}\times n_{3}},$ and  proposed a number of basic tools from differential geometry for tensors of low Tucker rank.
Moreover, they  studied the low Tucker rank tensor completion problem by
Riemannian CG method.  With the manifold framework listed in \cite{Kressner},  Heidel and Schulz \cite{heidel} considered the low Tucker rank tensor completion problem by Riemannian trust-region methods. Some other results can be found in \cite{Kasai,nimishakavi}. Besides the fixed Tucker rank tensor manifold, other choices of the smooth manifold are also used such as hierarchical Tucker format \cite{uschmajew} and fixed TT rank manifold \cite{sebastian2010} for handling high dimensional applications.  Unfortunately, their results are not sufficient
to derive the bounds on the number of sample entries required to recover low rank tensors.

Recently, Kilmer et al. \cite{kilmer,Kilmer2011} proposed the tubal rank of a
third-order tensor, which is based on tensor-tensor product
(t-product) and its algebra framework, where the t-product allows
tensor factorizations like matrix cases. Jiang et al. \cite{jiangm} showed that one can recover a low tubal-rank tensor
exactly with overwhelming probability by simply solving a convex program.
Other results can be found in \cite{Zhang2017,zhang}
and the references therein. These approaches have been shown to yield
good recovery results when applied to the tensors from various fields
such as medical imaging, hyperspectral images and seismic data.

\subsection{The Contribution}
In this paper, we mainly consider the low rank tensor completion
problem by Riemannian optimization methods on the manifold of low transformed rank tensor.  Mathematically, it can be expressed as
\begin{equation}\label{main1}
\text{min}~~f(\mathcal{Z})=\frac{1}{2}\|R_{\Omega}(\mathcal{Z})-R_{\Omega}(\mathcal{A})\|_{F}^{2} ~~ \text{s.t.} ~~
 \rank_{t}(\mathcal{Z})=\textbf{r}.
\end{equation}
Here, $\rank_{t}(\mathcal{Z})$ is the transformed multi-rank of $\mathcal{Z}$ given in
\cite{Kilmer2011} and $R_{\Omega}$  denotes  the sampling
operator
\begin{equation}\label{s1}
R_{\Omega}(\mathcal{Z})=\sum_{i,j,k=1}^{|\Omega|}\langle \mathcal{E}_{ijk},\mathcal{Z}\rangle\mathcal{E}_{ijk},
\end{equation}
with $\mathcal{E}_{ijk}$ being a tensor whose
$(i,j,k)$ position is $1$ and zeros everywhere else.   The $(i,j,k)$-th
component of $R_{\Omega}(\mathcal{Z})$ is zero unless $(i,j,k)\in
\Omega.$ Here, we consider the sampling with replacement model instead of without
replacement model. Then, for $(i,j,k)\in \Omega$, $R_{\Omega}(\mathcal{Z})$ is equal
to $\mathcal{Z}_{ijk}$ times the multiplicity of $(i,j,k)\in \Omega$.
 Sampling with replacement model can be viewed as a proxy for the uniform sampling model
and the failure probability under the uniform sampling model is less
than or equal to the failure probability under the sampling with
replacement model \cite{yuan}.

The main contribution of this paper can be expressed as follows. We first establish the set of fixed  transformed multi-rank tensors forms a   Riemannian manifold, which is different from the well known fixed Tucker rank tensor manifold \cite{Kressner} and fixed tensor train rank manifold \cite{sebastian2010}.    And then, we show that with suitable incoherence conditions
the proposed Riemannian optimization method on the fixed  transformed multi-rank tensor is guaranteed to converge to
the underlying low rank tensor with a high probability. Moreover, numbers of sample entries required for
solving low rank tensor completion problem under different initialized methods are studied and derived.

The outline of this paper is given as follows. In Section 2, we summarize the
notations used through out this paper. The preliminaries of
tensor singular value decomposition theory as well as the
differential geometric properties of transformed multi-rank tensor manifold are
also presented. In Section 3, the tensor conjugate gradient descent algorithms
based on the  fixed multi-rank tensor manifold is presented and analysed.
In Section 4, we provide the bounds on the number of sample entries required for tensor completion under
different initialization methods. In Section 5, we present several synthetic data and imaging data sets to demonstrate
the performance of the proposed algorithms. Finally, some concluding remarks are given in Section 6.

\section{Preliminaries}
In this section, some notations and notions relate to tensors and manifolds used throughout this paper are reviewed. We also refer the reader to \cite{kernfel,lee2013smooth} and references therein for more details about tensors and manifolds, respectively.
\subsection{Tensors}

Throughout this paper, tensors
are denoted by boldface Euler letters and matrices by boldface
capital letters. Vectors are represented by boldface lowercase
letters and scalars by lowercase letters. The field of real number
is denoted as $\R$. For a
third-order tensor $\mathcal{A} \in \R^{n_1 \times n_2 \times n_3}$,
we denote its $(i, j, k)$-th entry as $\mathcal{A}_{ijk}$ and use
the \textsc{Matlab} notation $\mathcal{A}(i,:,:)$,
$\mathcal{A}(:,i,:)$ and $\mathcal{A}(:,:,i)$ to denote the $i$-th
horizontal, lateral and frontal slice, respectively. Specifically,
the front slice $\mathcal{A}(:,:,i)$ is denoted compactly as
$\mathcal{A}^{(i)}$ and the  horizontal slice $\mathcal{A}(:,i,:)$ is
denoted as $\mathcal{A}^{[i]}$. $\mathcal{A}(i, j, :)$ denotes a
tubal fiber oriented into the board obtained by fixing the first two
indices and varying the third. Moreover, a tensor tube of size $1
\times 1 \times n_3$ is denoted as $\tub{a}$ and a tensor column of
size $n_1 \times 1 \times n_3$ is denoted as $\tc{a}$.

The inner product of $\mtx{A}$ and $\mtx{B}$ in $\R^{n_1 \times
n_2}$ is given by $\langle\mtx{A}, \mtx{B}\rangle =
\tr(\mtx{A}^{T}\mtx{B})$, where $\mtx{A}^{T}$ denotes the
transpose of $\mtx{A}$ and $\tr(\cdot)$ denotes the matrix trace.
The inner product of $\mathcal{A}$ and $\mathcal{B}$ in $\R^{n_1
\times n_2 \times n_3}$ is defined as
\begin{align}\label{e8}
\langle\mathcal{A}, \mathcal{B}\rangle = \sum_{i=1}^{n_3}\langle
\mathcal{A}^{(i)}, \mathcal{B}^{(i)}\rangle.
\end{align} For a
tensor $\mathcal{A}$, we denote  the infinity norm as
$\|\mathcal{A}\|_{\infty} = \max_{ijk}|\mathcal{A}_{ijk}|$  and the
Frobenius norm as $\|\mathcal{A}\|_F =
\sqrt{\sum_{ijk}|\mathcal{A}_{ijk}|^2}$.

\subsection{t$_{c}$-SVD}

Based on applying the Fast Fourier Transform (FFT) along all the tubes of a tensor, Kilmer et al. \cite{kilmer,Kilmer2011} introduced the tensor-tensor product (tt-product) and tensor singular value decomposition (t-svd) theory, respectively. Later, Kernfeld et al. \cite{kernfel} shown that the tt-product and t-svd can be implemented by using a discrete cosine transform, and the corresponding algebraic framework can also be derived.
In signal processing, many real data sets satisfy reflexive Boundary Conditions rather than  periodic boundary conditions, then better results can be derived by using  Discrete Cosine Transform (DCT) instead of FFT. Moreover, DCT only produces real number for real input in the transform domain which is important in the Riemannian manifold structure analysis. For these reasons, in this paper, we mainly consider the tensor singular value decomposition theory based on the DCT.

For a
third-order tensor $\mathcal{A} \in \R^{n_1 \times n_2 \times n_3}$, $\widehat{\mathcal{A}}$ represents the tensor obtained by
taking the DCT of all the tubes along
the third dimension of $\mathcal{A}$, i.e.,
\begin{equation*}
\vv(\widehat{\mathcal{A}}(i, j, :)) = \dct(\vv(\mathcal{A}(i,
j, :))),
\end{equation*}
where $\vv$ is the vectorization operator that maps the tensor tube
to a vector, and $\dct$ stands for the DCT. For compactness,
we will denote  $\widehat{\mathcal{A}} = \dct(\mathcal{A}, [], 3)$. In
the same fashion, one can also compute $\mathcal{A}$ from
$\widehat{\mathcal{A}}$ via $\idct(\widehat{\mathcal{A}}, [], 3)$
using the inverse DCT operation along the third-dimension. For sake
of brevity, we direct the interested readers to \cite{kilmer,Kilmer2011}.

\begin{definition}[Block diagonal form of third-order tensor \cite{Kilmer2011}]
Let $\overline{\mathcal{A}}$  be the block diagonal matrix of the
tensor $\mathcal{A} \in \R^{n_1 \times n_2 \times n_3}$ in the
transform domain, namely,
\begin{equation*}
\overline{\mathcal{A}}  = \bdiag(\widehat{\mathcal{A}}) = \left[
\begin{array}{llll}
\widehat{\mathcal{A}}^{(1)} & & & \\
 & \widehat{\mathcal{A}}^{(2)} & & \\
 & & \ddots &\\
 & & & \widehat{\mathcal{A}}^{(n_3)}
\end{array}
\right] \in \R^{n_1 n_3 \times n_2 n_3}.
\end{equation*}
\label{def3}
\end{definition}

In addition, the block diagonal matrix can be converted into a tensor by the `fold' operator:
$\mbox{\tt fold} \left(\bdiag(\widehat{\mathcal{A}})\right)= \widehat{\mathcal{A}}$.
The following fact will be used through out the paper. For any
tensor $\mathcal{A} \in \R^{n_1 \times n_2 \times n_3}$ and
$\mathcal{B} \in \R^{n_1 \times n_2 \times n_3}$, the inner product of two tensors satisfies
$\langle\mathcal{A}, \mathcal{B}\rangle =\langle{\overline{\mathcal{A}}}, {\overline{\mathcal{B}}}\rangle
\in \R$. After introducing the tensor notation and terminology, we give the
basic definitions on t$_{c}$-product and outline the associated algebraic
framework which
serves as the foundation for our analysis in next section.

\begin{definition}[\cite{kernfel}]\label{def1}
The $t_{c}$-product $\mathcal{A} \ast \mathcal{B}$ of $\mathcal{A} \in
\R^{n_1 \times n_2 \times n_3}$ and $\mathcal{B} \in \R^{n_2 \times
n_4 \times n_3}$ is a tensor $\mathcal{C} \in \R^{n_1 \times n_4
\times n_3}$ which is given by
\begin{equation*}
\mathcal{C} =\mathcal{A} \ast \mathcal{B}=\idct\left[ \mbox{\tt fold} \left(
 \mbox{\tt blockdiag}(\mathcal{\widehat{A}}) \times
\mbox{\tt blockdiag}(\mathcal{\widehat{B}})\right) \right ],
\end{equation*}
where $``\times"$ denotes the usual matrix product.
\end{definition}

Note that a third-order tensor of size $n_1 \times n_2 \times n_3$
can be regarded as an $n_1 \times n_2$ matrix with each entry as a
tube lies in the third dimension. This new perspective has
endowed multidimensional data arrays with an advantageous
representation in real-world applications. Similar as Definition \ref{def1}, it is convenient to
rewrite Definition 4.4 in \cite{kernfel} as follows.

\begin{definition}[\cite{kernfel}] 
\label{def2} The transpose of $\mathcal{A}\in
\R^{n_{1}\times n_{2}\times n_{3}}$ with respect to ${\bf \dct}$ is
the tensor $\mathcal{A}^{T}\in \R^{n_{2}\times n_{1}\times n_{3}}$
obtained by
\begin{equation*}
\mathcal{A}^{T} = \idct \left [ \mbox{\tt fold} \left
(\mbox{\tt blockdiag}(\mathcal{\widehat{A}})^{T}\right
)\right].
\end{equation*}
\end{definition}

Next we would like to introduce the identity tensor with respect to
DCT, which is also given in \cite{kernfel}. We construct a
tensor ${\cal I} \in \R^{n \times n \times n_3}$ with each frontal
slice $\mathcal{I}^{(i)}$ $(i = 1, \cdots, n_3)$ being an $n\times n$ identity matrix.

\begin{definition} \cite[Proposition 4.1]{kernfel}\label{idd}
The identity tensor $\mathcal{I}_{dct} \in \R^{n \times n
\times n_3}$ (with respect to DCT) is defined to be a
tensor such that ${\cal I}_{dct} = \idct  [ {\cal I}]$.
\end{definition}

Note that
$\mathcal{A}$ is a diagonal tensor if and only if each frontal slice $\mathcal{A}^{(i)}$ is a diagonal matrix.
The aforementioned notions allow us
to propose the following tensor singular value decomposition theory
(t$_{c}$-SVD).
\begin{definition}[\cite{kernfel}] \label{def1.4}
For $\mathcal{A} \in \R^{n_1 \times n_2 \times n_3}$, the t$_{c}$-SVD of
$\mathcal{A}$ is given as
\begin{equation*}
\mathcal{A} = \mathcal{U} \ast \mathcal{S} \ast \mathcal{V}^{T},
\label{eq8}
\end{equation*}
where $\mathcal{U} \in \R^{n_1 \times n_1 \times n_3}$ and
$\mathcal{V} \in \R^{n_2 \times n_2 \times n_3}$ are orthogonal
tensors, and $\mathcal{S} \in \R^{n_1 \times n_2 \times n_3}$ is a
diagonal tensor, respectively.
\end{definition}
Then $\mathcal{U}$, $\mathcal{V}$ and $\mathcal{S}$ in the
t$_{c}$-SVD can be computed by SVDs of $\widehat{\mathcal{A}}^{(i)}$, which is summarized in Algorithm 1.

\begin{algorithm}[h]
\caption{t$_{c}$-SVD for third-order tensors \cite{kernfel}} \label{algj1}
\textbf{Input:} $\mathcal{A}\in \R^{n_1 \times n_2 \times n_3}$. \\
~~1: $\mathcal{\hat{A}}=\dct[\mathcal{A}]$;\\
~~2: \textbf{for} $i=1,...,n_{3}$ \textbf{do}\\
~~3: $[\textbf{U},\textbf{S},\textbf{V}]=SVD(\mathcal{\hat{A}}^{(i)})$;\\
~~4: $\mathcal{\hat{U}}^{(i)}=\textbf{U},~\mathcal{\hat{S}}^{(i)}=\textbf{S},~\mathcal{\hat{V}}^{(i)}=\textbf{V};$\\
~~5:  \textbf{end for}\\
~~6:
 $\mathcal{U}=\idct[\mathcal{\hat{U}}]$, $\mathcal{S}=\idct[\mathcal{\hat{S}}]$,
 $\mathcal{V}=\idct(\mathcal{\hat{V}}]$. \\
\textbf{Output:}  $\mathcal{U}\in \R^{n_1 \times n_1 \times
n_3},~\mathcal{S}\in \R^{n_1 \times n_2 \times n_3},~\mathcal{V}\in
\R^{n_2 \times n_2 \times n_3}.$
\end{algorithm}

\subsection{Tensor Transformed Multi-rank}

Based on t$_{c}$-SVD given in Definition \ref{def1.4}, one can get the definitions of the transformed multi-rank and tubal
rank, respectively.

\begin{definition}[\cite{Zhang2017}]
The transformed multi-rank of a tensor $\mathcal{A} \in
\R^{n_1 \times n_2 \times n_3}$ denoted as
$\rank_{t}(\mathcal{A})$, is a vector $\vct{r} \in \R^{n_3}$
with its $i$-th entry as the rank of the $i$-th frontal slice of
$\widehat{\mathcal{A}}$, i.e.,
\begin{equation*}
\rank_{t}(\mathcal{A})=\textbf{r},~~r_i = \rank(\widehat{\mathcal{A}}^{(i)}),~~i=1,...,n_{3}.
\end{equation*}
The tubal rank of a tensor, denoted as
$\rank_{ct}(\mathcal{A})$, is defined as the number of nonzero
singular tubes of $\mathcal{S}$, where $\mathcal{S}$ comes from the
t$_{c}$-SVD of $\mathcal{A}=\mathcal{U} \ast
\mathcal{S}\ast\mathcal{V}^{T}$, i.e.,
\begin{equation*}
\rank_{ct}(\mathcal{A}) = \#\{i: \mathcal{S}(i, i, :) \neq \vct{0}\}
= \max_{i} r_i.
\end{equation*}
\end{definition}

For computational improvement, we will use the skinny t$_{c}$-SVD throughout the paper unless otherwise stated.

\begin{remark}
For $\mathcal{A}\in \R^{n_{1}\times n_{2}\times n_{3}}$ with $\emph{rank}_{t}(\mathcal{A})=(r_{1},...,r_{n_{3}})=\textbf{r}$ and $\emph{rank}_{ct}(\mathcal{A})=r,$  denote $\mathcal{I}_{\textbf{r}}= \idct \left [ \mathcal{I}_{d}\right],$ where each frontal
slice of $\mathcal{I}_{d}$ obeying $ \mathcal{I}_{d}^{(i)}= \left(
                                                                                       \begin{array}{cc}
                                                                                         I_{r_{i}} & 0 \\
                                                                                         0 & 0 \\
                                                                                       \end{array}
                                                                                     \right)
$ $(i = 1, \cdots, n_3)$.
Then the skinny t$_{c}$-SVD of $\mathcal{A}$ is given as $\mathcal{A}= \mathcal{U} \ast \mathcal{S}
\ast \mathcal{V}^{T}$, where  $\mathcal{U}\in \R^{n_{1}\times r\times n_{3}},$
$\mathcal{V}\in \R^{n_{2}\times r\times n_{3}}$ satisfying
 \begin{align*}
 \mathcal{U}^T\ast \mathcal{U}=\mathcal{I}_{\textbf{r}},~\mathcal{V}^{T}\ast \mathcal{V}=\mathcal{I}_{\textbf{r}},~\rank_{t}(\mathcal{U})=\rank_{t}(\mathcal{V})=\textbf{r}, ~\rank_{ct}(\mathcal{U})=\rank_{ct}(\mathcal{V})=r,
 \end{align*}
 and $\mathcal{S}\in \R^{r\times r\times n_{3}}$ is diagonal tensor.
\end{remark}

The spectral norm  and the condition number of a tensor  are defined as follows.

\begin{definition}\label{def5}
The tensor spectral norm of $\mathcal{A}\in \R^{n_1 \times n_2 \times n_3}$, denoted as $\|\mathcal{A}\|$ is defined as
$\|\mathcal{A}\| = \|\overline{\mathcal{A}}\|$, where $\overline{\mathcal{A}}$ is  the block diagonal matrix of $\mathcal{A}$ in the
transform domain. The condition number of $\mathcal{A}$, denoted by $\kappa(\mathcal{A})$
 is defined as
$\kappa(\mathcal{A})=\frac{\sigma_{max}(\mathcal{A})}{\sigma_{min}(\mathcal{A})}$,
where $\sigma_{max}(\mathcal{A})$ and $\sigma_{min}(\mathcal{A})$ are the largest and the smallest nonzero singular values of $\overline{\mathcal{A}},$ respectively.
\end{definition}

In other words, the
tensor spectral norm of $\mathcal{A}$ equals to the matrix spectral
norm of its block diagonal form $\overline{\mathcal{A}}$.

\begin{definition}[\cite{Zhang2017}]
 Then the tensor operator norm is defined as
\begin{equation*}
\|{\tt L}\|_{\textup{op}} = \sup_{\|\mathcal{A}\|_F \leq 1} \|{\tt
L}(\mathcal{A})\|_F.
\end{equation*}
\end{definition}

If ${\tt L}: \mathbb{R}^{n_1\times n_2 \times n_3}\rightarrow \mathbb{R}^{n_4\times n_2 \times n_3}$ is a tensor operator mapping an $n_1\times n_2 \times n_3$  tensor $\ten{A}$ to an $n_4\times n_2 \times n_3$ tensor $\ten{B}$ via the $t_{c}$-product  as $\ten{B}={\tt L}(\ten{A})=\ten{L}\ast \ten{A},$ where $\ten{L}$ is an $n_4\times n_1 \times n_3$ tensor, we have $\|
{\tt L}\|_{\textup{op}} = \|\mathcal {L}\|$.
Now we need to introduce a new kind of  tensor basis which
is different from Definition 2.2 in \cite{Zhang2017}.  It is
worth noting that the new tensor basis plays an  important role in
tensor coordinate decomposition and defining  the tensor
incoherence conditions in the sequel.

\begin{definition} \label{defn}
The transformed column basis with respect to $\dct$, denoted as
$\tc{e}_i$, is a tensor of size $n_1 \times 1 \times n_3$ with the
$i$-th tube of $\dct [\tc{e}_i]$ is equal to $\frac{1}{\sqrt{n_3}}\tc{1}$
(each entry in the $i$-th tube is $1/\sqrt{n_3}$) and the rest equaling to 0.
Its associated conjugate transpose $\tc{e}_i^T$ is called
transformed row basis with respect to $\dct$.
\end{definition}

Moreover, some incoherence conditions on $\mathcal{L}_0$ are needed to ensure that it is not sparse.

\begin{definition}[Tensor Incoherence Conditions] Suppose
that $\mathcal{L}_0\in \R^{n_{1}\times n_{2}\times n_{3}},$ and  $\rank_{t}(\mathcal{L}_0)
= \textbf{r} ~\text{with}~ \rank_{ct}(\mathcal{L}_0) = r$. Its skinny t$_{c}$-SVD  is
 $\mathcal{L}_0 = \mathcal{U} \ast \mathcal{S} \ast \mathcal{V}^{T}$. Then $\mathcal{L}_0$ is
 said to satisfy the tensor incoherence conditions with parameter $\mu_{0} > 0$ if
\begin{align}
\max_{i=1, \dots, n_1} \|\mathcal{U}^{T} &\ast \tc{e}_i\|_F \leq
\sqrt{\frac{\mu_{0} r}{n_1}},~~ \max_{j=1, \dots, n_2} \|\mathcal{V}^{T}
\ast \tc{e}_j\|_F \leq \sqrt{\frac{\mu_{0} r}{n_2}}. \label{eq17}
\end{align}
\label{def13}
\end{definition}

\begin{definition} [Tensor Joint Incoherence Condition] Suppose that $\mathcal{L}_0\in \R^{n_{1}\times n_{2}\times n_{3}}$ with $\rank_{t}(\mathcal{L}_0) = \textbf{r} ~\text{and}~ \rank_{ct}(\mathcal{L}_0) = r.$  Assume there exist a positive numerical constant $\mu_{1}$ such that
\begin{align*}
\|\mathcal{L}_{0}\|_{\infty}\leq \mu_{1}\sqrt{\frac{ r}{n_1 n_2
n_3}}\|\mathcal{L}_{0}\|.
\end{align*}
\end{definition}

\subsection{Manifolds}

A smooth manifold with its tangent space, a proper definition of Riemannian metric for gradient projection and the retraction map are three essential settings for Riemannian optimization.
Moreover, the key idea of Riemannian gradient iteration method  contains two step in each iteration: (1) perform a gradient step in the search space; (2) map the result back to the  manifold by a proper retraction. In this subsection,
we first show  the  fixed transformed multi-rank tensor set forms an embedded manifold of $\mathbb{R}^{n_1\times n_2\times n_3}$. And then, we also consider the tangents spaces, Riemannian metric as well as the retraction mapping relate to the fixed transformed multi-rank tensor manifold.

The following proposition shows that the fixed transform multi-rank tensor set is indeed a smooth manifold.

\begin{proposition}\label{prop1}
Let
\begin{equation}\label{m1}
 \ten{M}_{\textbf{r}}=\{\mathcal{X}\in \R^{n_{1}\times n_{2}\times n_{3}}:\rank_{t}(\mathcal{X})=\textbf{r}\}
 \end{equation}
 denote the set of fixed  transformed multi-rank tensors with $\textbf{r}=(r_{1},r_{2},...,r_{n_{3}}).$ Then $\ten{M}_{\textbf{r}}$ is an embedded manifold of $\R^{n_{1}\times n_{2}\times n_{3}}$ and its dimension is  $\sum_{i=1}^{n_{3}}((n_{1}+n_{2})r_{i}-r_{i}^2).$
\end{proposition}

While the existence of such a smooth manifold structure, the tangent space of a point on the manifold can be given as follows.

\begin{proposition}\label{prop2}
Let $\ten{M}_{\textbf{r}}$ be given as \eqref{m1} and $\mathcal{X}_{l}\in \ten{M}_{\textbf{r}}$ be arbitrary. Suppose that $\mathcal{X}_{l}=\mathcal{U}_{l}\ast \mathcal{S}_{l}\ast
\mathcal{V}_{l}^{T},$ then the tangent space of $\ten{M}_{r}$ at
$\mathcal{X}_{l}$ can be given as
\begin{equation}\label{xin1}
\ten{T}_{\mathcal{X}_l}\ten{M}_{\textbf{r}}=\{\mathcal{U}_{l}\ast
\mathcal{Z}_{1}^{T}+\mathcal{Z}_{2}\ast
\mathcal{V}^{T}_{l}\},
\end{equation}
where $\mathcal{Z}_{1}\in \R^{n_2 \times r \times
n_3}~\text{and}~\mathcal{Z}_{2}\in \R^{n_1 \times r \times n_3}$ are free parameters.
\end{proposition}

The tangent bundle of $\ten{M}_{\textbf{r}}$, denoted by $T\ten{M}_{\textbf{r}}$, is defined as the  disjoint union of the tangent spaces at all points of $\ten{M}_{\textbf{r}}: T\ten{M}_{\textbf{r}}=\bigcup_{p\in \ten{M}_{\textbf{r}}}T_{p}\ten{M}_{\textbf{r}}.$  In some sense, the tangent bundle can be seen as a collection of vector spaces.
 Note that when a smooth manifold $\mathcal{M}$ is endowed with a specific Riemannian metric $g$, then the pair $(\ten{M},g)$ is called a Riemannian manifold.  Here, for the smooth manifold $\ten{M}_{\textbf{r}},$ the Riemannian metric $g_{\ten{X}}$ is defined as
\begin{equation}
g_{\ten{X}}(\zeta,\eta):=<\zeta,\eta> ~\text{with}~ \ten{X}\in \ten{M}_{\textbf{r}}~\text{and}~\zeta,\eta \in T_{\ten{X}} \ten{M}_{\textbf{r}}.
\end{equation}
Based on this metric, $(\ten{M}_{r},g_{\ten{X}})$ becomes a Riemannian manifold. In the sequel, we write $\ten{M}_{r}$ as a Riemannian manifold   for simplicity.   After that, the  gradient of an objective function in Riemannian manifold can be introduced. For the Riemannian manifold $\ten{M}_{\textbf{r}},$  the Riemannian gradient of a smooth function $f: \ten{M}_{\textbf{r}}\rightarrow \mathbb{R}$ at $\mathcal{X}\in \ten{M}_{\textbf{r}}$ is defined as the unique tangent vector $\textbf{grad} f(\ten{X}) $ in $T_{\ten{X}}\ten{M}_{\textbf{r}}$ such that	$\<\textbf{grad}  f(\ten{X}),\xi \>= \textbf{D}f(\xi)$ for all $\xi \in T_{\ten{X}}\ten{M}_{\textbf{r}},$ where $\textbf{D}f$ denotes the directional derivative. Note that $\ten{M}_{\textbf{r}}$ is an  embedded smooth manifold of $\mathbb{R}^{n_1\times n_2 \times n_3},$ then the Riemannian gradient can be seen as the orthogonal projection onto the tangent space of the gradient of $f$ on $\mathbb{R}^{n_1\times n_2 \times n_3}$.  Note that the tangent space of $\ten{M}_{\textbf{r}}$ at $\ten{X}_{l}$ can be expressed as \eqref{xin1}, then a tensor $\mathcal{A}\in \R^{n_{1}\times n_{2} \times n_{3}}$  can be projected onto $\ten{T}_{\mathcal{X}_{l}}\ten{M}_{\textbf{r}}$ by the orthogonal projection $P_{\ten{T}_{\mathcal{X}_{l}}\ten{M}_{\textbf{r}}}  :  \R^{n_{1}\times n_{2} \times n_{3}}\rightarrow \ten{T}_{\mathcal{X}_{l}}\ten{M}_{\textbf{r}}$ with
\begin{eqnarray}
P_{\ten{T}_{\mathcal{X}_{l}}\ten{M}_{\textbf{r}}} (\mathcal{A}) =  \mathcal{U}_{l}\ast
\mathcal{U}_{l}^{T}\ast \mathcal{A} +\mathcal{A}\ast
\mathcal{V}_{l}\ast \mathcal{V}_{l}^{T} - \mathcal{U}_{l}\ast
\mathcal{U}_{l}^{T}\ast \mathcal{A}\ast \mathcal{V}_{l}\ast
\mathcal{V}_{l}^{T}.\label{prj1}
\end{eqnarray}
It follows that the Riemannian gradient of the
objective function
$f(\mathcal{Z})=\frac{1}{2}\|R_{\Omega}\mathcal{Z}-R_{\Omega}\mathcal{A}\|_{F}^{2}$
given in model \eqref{main1}, can be expressed as
$\mbox{\tt grad}~ f(\mathcal{Z}):=P_{\ten{T}_{\mathcal{Z}}\ten{M}_{\textbf{r}}}(R_{\Omega}\mathcal{Z}-R_{\Omega}\mathcal{A})$.

\subsubsection{Retraction}

A Riemannian manifold is  not a linear space in general, the calculations
required for a continuous optimization method need to be performed in its
tangent space. Therefore, in each step, a so-called retraction mapping is needed to project points from a
tangent bundle  $T\ten{M}$  to the manifold $\ten{M}$  to generate the new iteration.
Known from Definition 1 in \cite{absil2012} that retractions are essentially first-order approximations of the exponential map of the manifold which is usually expensive to compute.
If $\ten{M}$ is an embedded submanifold,  then the orthogonal projection
$$P_{\ten{M}}(x+\xi)=\mbox{\tt arg}_{y\in\ten{M}}\|x+\xi-y\|$$
include the so called projective retraction
$$R:\mathcal{U}\rightarrow \ten{M}, ~~(x,\xi)\rightarrow P_{\ten{M}}(x+\xi).$$
For the fixed rank matrix  manifold case, mapping $R$ can be computed in closed-form by the SVD truncation \cite{bart}. In our setting,  the Riemannian manifold $\ten{M}_{\textbf{r}}$ is a embedded manifold of $\R^{n_{1}\times n_{2}\times n_{3}}$, we can choose metric projection as a retraction:
\begin{align}\label{new1}
R_{\mathcal{X}}:\mathcal{U}_{\mathcal{X}}\rightarrow
\ten{M}_{\textbf{r}},~(\ten{X},\mathcal{\xi})\rightarrow
P_{\ten{M}_{\textbf{r}}}(\mathcal{X}+\mathcal{\xi}):=\mbox{\tt arg} \min_{\mathcal{Z}\in
\ten{M}_{\textbf{r}}}\|\mathcal{X}+\mathcal{\xi}-\mathcal{Z}\|_{F},
\end{align}
where $\mathcal{U}_{\mathcal{X}}\subseteq
\ten{T}_{\mathcal{X}}\ten{M}_{\textbf{r}}$ is a suitable neighborhood around zero
and $P_{\ten{M}_{\textbf{r}}}$ is the orthogonal projection onto
$\ten{M}_{\textbf{r}}$.
Under the tensor t$_{c}$-SVD framework, and recall the retraction operator defined in \eqref{new1}, we can get the follows.

\begin{proposition}(t$_{c}$-SVD truncation as   Retraction) \label{prop3}
Let $\ten{X}\in \ten{M}_{\textbf{r}},$ with $\textbf{r}=(r_{1},...,r_{n_{3}}).$ The map
\begin{align}\label{truncation}
R_{\mathcal{X}}:\ten{T}_{\mathcal{X}}\ten{M}_{\textbf{r}}\rightarrow
\ten{M}_{\textbf{r}},~(\ten{X},\mathcal{\xi})\rightarrow
P_{\ten{M}_{\textbf{r}}}(\mathcal{X}+\mathcal{\xi}):=\mathcal{U} \ast \mathcal{S}_{r} \ast \mathcal{V}^{T}
\end{align}
 where $\mathcal{U} \in \R^{n_1 \times r \times n_3}, \mathcal{V} \in \R^{n_2 \times r \times n_3},$  $\rank_{t}(\mathcal{U})=\rank_{t}(\mathcal{V})=\textbf{r}$, $\mathcal{U}^T\ast \mathcal{U}=\mathcal{I}_{\textbf{r}},$ $ \mathcal{V}^{T}\ast \mathcal{V}=\mathcal{I}_{\textbf{r}},$ and
  $\mathcal{S}_{\textbf{r}}$ is a diagonal tensor with
 \begin{align*}
\hat{\mathcal{S}}_{r}^{(i)}:= \left\{\begin{array}{lll}
\hat{\mathcal{S}}_{r}^{(i)}, & \mbox{if} \ i\leq r_{i},\\
0, &\mbox{if} \ i> r_{i},\\
\end{array}\right.
\end{align*}
is a retraction on $\ten{M}_{\textbf{r}}$ around $\ten{X}$.
\end{proposition}

In addition, vector transport is defined as a method to transport tangent vectors
from one tangent space to another, which was introduced in
\cite{Absil,Kressner}. In the Riemannian conjugate gradient descent method discussed here, the
search direction is a linear combination of the projected gradient direction and the past search direction
 projected onto the tangent space of the current estimate.
Then we can define
\begin{align*}
\ten{T}_{\mathcal{X}\rightarrow\mathcal{Y}}:~\ten{T}_{\mathcal{X}}\ten{M}_{\textbf{r}}\rightarrow \ten{T}_{\mathcal{Y}}\ten{M}_{\textbf{r}},~
\mathcal{\xi}\rightarrow P_{\ten{T}_{\mathcal{Y}}\ten{M}_{\textbf{r}}}(\mathcal{\xi}),
\end{align*}
 where $P_{\ten{T}_{\mathcal{Y}}\ten{M}_{\textbf{r}}}$ is defined as in \eqref{prj1}.

\section{The Convergence Analysis}
\label{sec3}

\subsection{Sampling with Replacement}

For matrix or tensor completion problems, most of the existing work \cite{candes2009exact,jiangm,Kressner,bart} studied a Bernoulli sampling model as a proxy uniform sampling. It has been proved that the probability of failure when the set of observed entries is sampled uniform from the collection of set is bounded by 2 times of the probability of failure
under the Bernoulli model \cite{candes2006robust}. However, it is bigger or equal to the probability that fails when the entries are sampled independently with replacement \cite{Recht2011}. It is surprising that after changing the sampling model, most of the theorems from \cite{candes2009exact} came to be simple consequences of a noncommutative variant of Bernstein's inequality \cite{Recht2011}. In this paper, we consider the sampling with replacement model for
$\Omega$ in which each index is sampled independently from the
uniform distribution on
$\{1,...,n_{1}\}\times\{1,...,n_{2}\}\times\{1,...,n_{3}\}.$  At first glance,
sampling with replacement is not suitable for analyzing matrix or tensor completion problems,
as there may be some duplicate entries. However, the maximum duplication of any entry can be derived. Then,
this model can be viewed as a proxy for the uniform sampling model.

Different from the existing conclusions which based on
sampling without replacement model, here $R_{\Omega}$ defined as \eqref{s1} is not a unitary
projection if there are duplicates in $\Omega$.  Then,
the tensor completion model (\ref{n1}) can be rewritten as
\begin{equation*}
\min\limits_{\mathcal{Z}}~~f(\mathcal{Z})=\frac{1}{2}\|R_{\Omega}(\mathcal{Z}-\mathcal{A})\|^{2}_{F}
~~\text{s.t.}~~rank_{t}(\mathcal{Z})=\textbf{r}.
\end{equation*}
Suppose that $m=|\Omega|$ and $n=\max\{n_{1},n_{2},n_{3}\}$, then
for $\beta>1$, we have the following lemma.

\begin{lemma}\label{le2}
With high probability at least $1-n^{3-3\beta}$, the maximum number
of repetitions of any entry in $\Omega$ is less than
$\frac{10}{3}\log n$ for $n>9$ and $\beta>1$.
\end{lemma}

\begin{proof}
The proof follows the lines of \cite[Proposition 5]{Recht2011} in the matrix case. The tool used here is the standard Chernoff bound for the Bernoulli
distribution. Note that for a fixed entry, the probability it is
sampled more than $t$ times is equal to the probability of more than
$t$ heads occurring in a sequence of $m$ tosses where the
probability of a head is $\frac{1}{n_{1}n_{2}n_{3}}$. It follows from
\cite{hagerup1990} that
\begin{equation*}
P[\text{more that t head in m trials}] \leq
\left(\frac{m}{n_{1}n_{2}n_{3}}\right)^{t}
\exp\left(t-\frac{m}{n_{1}n_{2}n_{3}}\right).
\end{equation*}
 Then if $n\geq9$, applying the union bound over all of the $n_{1}n_{2}n_{3}$ entries we can
 get
\begin{eqnarray*}
& & P[\text{any entry is selected more than} ~\frac{10}{3}\beta\log
n~\text{times}]\\
& \leq &n_{1}n_{2}n_{3}(\frac{10}{3}\beta (\log
n))^{-\frac{10}{3}\beta \log n} \exp(\frac{10}{3}\beta \log n)\\
& \leq & n^{3}\exp(-\frac{10}{3}\beta \log n \log( \frac{10}{3}\beta
\log n ))\exp(\frac{10}{3}\beta \log n)\leq
n^{3}(n^{-\frac{10\beta}{3}})^{\log(\frac{10}{3} \beta \log
n)-1} \leq n^{3-3\beta}.
\end{eqnarray*}
\end{proof}
It follows from Lemma \ref{le2} that $\|R_{\Omega}\|\leq \frac{10}{3}\beta\log n$ with
high probability.

\subsection{The Algorithm}

We first identify a small neighborhood around the measured
low rank tensor such that for any given initial guess in this
neighborhood, the
tensor Riemannian conjugate gradient descent (Algorithm 2) will converge linearly to the objective
tensor. In Algorithm 2 (line 3), the search direction is a linear
combination of the projected gradient descent direction and the past
search direction projected onto the tangent space of the current
estimate. The orthogonalization weight $\beta_{l}$ in line 2 is
selected in a way such that $P_{\ten{T}_{l}}(P_{l})$ is conjugate
orthogonal to $P_{\ten{T}_{l}}(P_{l-1})$. In line 6, $H_{r}(\cdot)$ denotes the $tc$-SVD truncation operator given in Proposition \ref{prop3}.

\begin{algorithm}[h]
\caption{Tensor Riemannian Conjugate Gradient Descent} \label{alg1}
\textbf{Initilization:} $\mathcal{X}_{0}=\mathcal{U}_{0}\ast\mathcal{S}_{0}\ast \mathcal{V}_{0}^{T}\in \R^{n_1 \times n_2 \times n_3}$, $\beta_{0}=0,$ and $\mathcal{Q}_{-1}=0$. \\
\textbf{for} $l=0,1,...$ do\\
~~1: $\mathcal{G}_{l}=R_{\Omega}(\mathcal{X}-\mathcal{X}_{l})$;\\
~~2: $\beta_{l}=-\frac{\langle P_{\ten{T}_{l}}(\mathcal{G}_{l}),R_{\Omega}P_{\ten{T}_{l}}(\mathcal{Q}_{l-1})\rangle}{\langle P_{\ten{T}_{l}}(\mathcal{Q}_{l-1}),R_{\Omega}P_{\ten{T}_{l}}(\mathcal{Q}_{l-1})\rangle}$;\\
~~3: $\mathcal{Q}_{l}=P_{\ten{T}_{l}}(\mathcal{G}_{l})+\beta_{l}P_{\ten{T}_{l}}(\mathcal{Q}_{l-1})$;\\
~~4: $\alpha_{l}=\frac{\langle P_{\ten{T}_{l}}(\mathcal{G}_{l}),P_{\ten{T}_{l}}(\mathcal{Q}_{l})\rangle}{\langle P_{\ten{T}_{l}}(\mathcal{Q}_{l}),R_{\Omega}P_{\ten{T}_{l}}(\mathcal{Q}_{l})\rangle};$\\
~~5: $\mathcal{W}_{l}=\mathcal{X}_{1}+\alpha_{l}\mathcal{Q}_{l}$\\
~~6: $\mathcal{X}_{l}=H_{r}(\mathcal{W}_{l})$\\
  \textbf{end for}
\end{algorithm}

In order to improve the robustness of the non-linear conjugate
gradient descent methods, we introduce a restarted variant in
Algorithm 4: that is, $\beta_{l}$ is set $0$ and restarting occurs as long
as either of the following conditions is violated:
\begin{align}\label{eq6}
\frac{| \langle
P_{\ten{T}_{l}}(\mathcal{G}_{l}),P_{\ten{T}_{l}}(\mathcal{Q}_{l-1})\rangle
|}{\|P_{\ten{T}_{l}}(\mathcal{G}_{l})\|_{F}\|P_{\ten{T}_{l}}(\mathcal{Q}_{l-1})\|_{F}}\leq
k_{1},~~\|P_{\ten{T}_{1}}(\mathcal{G}_{l})\|_{F}\leq
k_{2}\|P_{\ten{T}_{l}}(\mathcal{Q}_{l-1})\|_{F}.
\end{align}
The first restarting condition guarantees that the residual will be
substantially orthogonal to the past search direction when projected
onto the tangent space of current estimate so that the new search
direction can be sufficiently gradient related. In the classical CG
algorithm for linear systems, the residual is exactly orthogonal to
all the past search directions. Roughly speaking, the second
restarting condition implies that the projection of current residual
cannot be too large when compared to the projection of the past
residual since the search direction is gradient related by the first
restarting condition. In our implementations, we take $k_1 = 0.1$
and $k_2 = 1$.

We first list some lemmas which will be used many times in the sequel. Their proofs can be found in Section \ref{seca}.

\begin{lemma}\label{lem5}
Suppose that $\mathcal{X}, \mathcal{X}_{l}\in \R^{n_{1}\times n_{2}\times n_{3}}$ with $\rank_{t}(\mathcal{X})=\rank_{t}(\mathcal{X}_{l}) = \textbf{r} ~\text{and}~ \rank_{ct}(\mathcal{X}_{l})=\rank_{ct}(\mathcal{X}_{l}) = r.$
Their skinny $t_{c}$-SVDs are given as $\mathcal{X}=\mathcal{U} \ast \Sigma \ast
\mathcal{V}^{T}$ and $\mathcal{X}_{l}=\mathcal{U}_{l} \ast \Sigma_{l}  \ast
\mathcal{V}^{T}_{l}.$  Let $\ten{T}$ and  $\ten{T}_{l}$  be the corresponding
tangent spaces of the fixed transformed  multi-rank  manifold at
$\mathcal{X}$ and $\mathcal{X}_{l},$ respectively.
Then
\begin{align*}
& (i)~\|\mathcal{U}_{l}\ast \mathcal{U}_{l}^{T}- \mathcal{U}\ast
\mathcal{U}^{T}\|\leq
\frac{\|\mathcal{X}_{l}-\mathcal{X}\|}{\sigma_{min}(\mathcal{X})},~
~~~~~~~~~~(ii)~\|\mathcal{V}_{l}\ast \mathcal{V}_{l}^{T}- \mathcal{V}\ast
\mathcal{V}^{T}\|\leq
\frac{\|\mathcal{X}_{l}-\mathcal{X}\|}{\sigma_{min}(\mathcal{X})},\\
& (iii)~\|\mathcal{U}_{l}\ast \mathcal{U}_{l}^{T}- \mathcal{U}\ast
\mathcal{U}^{T}\|_{F}\leq
\frac{\sqrt{2}\|\mathcal{X}_{l}-\mathcal{X}\|_{F}}{\sigma_{min}(\mathcal{X})},~~(iv)~\|\mathcal{V}_{l}\ast
\mathcal{V}_{l}^{T}- \mathcal{V}\ast \mathcal{V}^{T}\|_{F}\leq
\frac{\sqrt{2}\|\mathcal{X}_{l}-\mathcal{X}\|_{F}}{\sigma_{min}(\mathcal{X})},\\
& (v)~\|(\mathcal{I}-\mathcal{P}_{\ten{T}_{l}})\mathcal{X}\|_{F}\leq
\frac{\|\mathcal{X}_{l}-\mathcal{X}\|^{2}_{F}}{\sigma_{min}(\mathcal{X})},~~~~~~~~~~~~~~(vi)~\|\mathcal{P}_{\ten{T}_{l}}-\mathcal{P}_{\ten{T}}\|\leq
\frac{2\|\mathcal{X}_{l}-\mathcal{X}\|_{F}}{\sigma_{min}(\mathcal{X})}.
\end{align*}
\end{lemma}

\begin{lemma} \label{lem1}
Let $R_{\Omega}$ be defined as \eqref{main1}, $\mathcal{X},\mathcal{X}_{l},\ten{T}, \ten{T}_{l}$ be defined as in Lemma \ref{lem5}, and  $n=\max\{n_{1},n_{2},n_{3}\}$.
Assume that
\begin{align}
&  \|R_{\Omega}\|\leq \frac{10}{3}\beta \log n,\label{e1}\\
 &
\|P_{\ten{T}}-p^{-1}P_{\ten{T}}R_{\Omega}P_{\ten{T}}\|\leq \epsilon_{0},\label{e2}\\
&
   \frac{\|\mathcal{X}_{l}-\mathcal{X}\|_{F}}{\sigma_{min}(\mathcal{X})}\leq \frac{3p^{\frac{1}{2}}\varepsilon_{0}}{20\beta (1+\epsilon_{0})\log
n },\label{e3}
\end{align}
for some $0< \varepsilon_{0}< 1$ and $\beta>1.$ Then
\begin{align}
\|R_{\Omega}P_{\ten{T}_{l}}\|\leq \frac{10}{3}\beta
\log(n)(1+\epsilon_{0})p^{\frac{1}{2}},\label{e5}
\\\|P_{\ten{T}_{l}}-p^{-1}P_{\ten{T}_{l}}R_{\Omega}P_{\ten{T}_{l}}\|\leq
4\epsilon_{0}.\label{e6}
\end{align}
\end{lemma}

\begin{lemma}[Lemma 4.4 in \cite{wei2}]\label{wei4} Let $\rho_{1},\rho_{2}$ and $\gamma$ be
positive constants satisfying $\rho_{2}>\rho_{1}$. Define
\begin{equation*}
\gamma_{1}=\rho_{1}+\gamma, \gamma_{2}=(\rho_{2}-\rho_{1})\gamma,
~~\text{and}~~v=\frac{1}{2}(\gamma_{1}+\sqrt{\gamma_{1}^2+4\gamma_{2}}).
\end{equation*}
Let $\{c_{l}\}_{l\geq0}$ be a non-negative sequence satisfying
$c_{1}\leq vc_{0}$ and
\begin{align*}
c_{l+1}\leq \rho_{1}c_{l}+\rho_{2}\sum_{j=0}^{l-1}\gamma^{l-j}c_{j},
~\forall~ l\geq 1.
\end{align*}
Then if $\gamma_{1}+\gamma_{2}<1,$ we have $v<1$ and $c_{l+1}\leq
v^{l+1}c_{0}$.
\end{lemma}

In addition, we need to estimate $\alpha_{l}$ and $\beta_{l}$ in
Algorithm 4 with the restarting conditions in $(\ref{eq6})$. Their proof can
be found in Section \ref{seca}.

\begin{lemma}\label{lem2}
Assume that \eqref{e6} is satisfied.
When restarting occurs, $\beta_{l}=0,$ then the
stepsize $\alpha_{l}$ in Algorithm 2 can be bounded as
$$
\frac{1}{(1+4\epsilon_{0})}\leq \alpha_{l}=\frac{\|P_{\ten{T}_{l}}(\mathcal{G}_{l})\|_{F}^{2}}{\langle P_{\ten{T}_{l}}(\mathcal{G}_{l}),R_{\Omega}P_{\ten{T}_{l}}(\mathcal{G}_{l})\rangle }\leq \frac{1}{(1-4\epsilon_{0})p}.
$$
Moreover, the spectral norm of
$\|P_{\ten{T}_{l}}-\alpha_{l}P_{\ten{T}_{l}}R_{\Omega}P_{\ten{T}_{l}}\|$ can be
bounded as
$$
\|P_{\ten{T}_{l}}-\alpha_{l}P_{\ten{T}_{l}}R_{\Omega}P_{\ten{T}_{l}}\|\leq \frac{8\epsilon_{0}}{(1-4\epsilon_{0})}.
$$
\end{lemma}

\begin{lemma} \label{lemn1}
Assume that \eqref{e6} is satisfied. When restarting dose not occur, $\beta_{l}\neq 0,$ then we
have
\begin{equation*}
|\beta_{l}|\leq \epsilon_{\beta} ~~~{and}~~~ |\alpha_{l}\cdot
p-1|\leq \epsilon_{\alpha},
\end{equation*}
where
$$\epsilon_{\beta}=\frac{4k_{2}\epsilon_{0}}{(1-4\epsilon_{0})}+\frac{k_{1}k_{2}}{(1-4\epsilon_{0})},~~
\epsilon_{\alpha}=\frac{4\epsilon_{0}}{(1-4\epsilon_{0})-k_{1}(1+4\epsilon_{0})}.$$
Moreover, the spectral norm of $\|P_{\ten{T}_{l}}-\alpha_{l}
P_{\ten{T}_{l}}R_{\Omega}P_{\ten{T}_{l}}\|$ can be bounded as
$$\|P_{\ten{T}_{l}}-\alpha_{l}
P_{\ten{T}_{l}}R_{\Omega}P_{\ten{T}_{l}}\|\leq
4\epsilon_{0}+\epsilon_{\alpha}(1+4\epsilon_{0}).$$
\end{lemma}

With the above tools in hand, we can get one of the main results
of this paper.

\begin{theorem} \label{themm}
Suppose that \eqref{e1}-\eqref{e3} are satisfied,
where $\beta>1,$ $k_1,k_2$ are defined as in \eqref{eq6}, and $\epsilon_{0}$ is a positive numerical
constant satisfying $\gamma_{1}+\gamma_{2}<1$,
where
\begin{equation*}
\gamma_{1}=\frac{18\epsilon_{0}-10k_{1}\epsilon_{0}(1+4\epsilon_{0})}
{(1-4\epsilon_{0})-k_{1}(1+4\epsilon_{0})}+\frac{4k_{2}\epsilon_{0}+k_{1}k_{2}}{1-4\epsilon_{0}},~
\gamma_{2}=\frac{8k_{2}\epsilon_{0}+2k_{1}k_{2}}{1-4\epsilon_{0}}.
\end{equation*}
Then we have
$v_{cg}=\frac{1}{2}(\gamma_{1}+\sqrt{\gamma_{1}^{2}+4\gamma_{2}})<1$
and the iterates $\mathcal{X}_{l}$ generated by Algorithm 2
satisfy
$$\|\mathcal{X}_{l}-\mathcal{X}\|_{F}\leq v_{cg}^{l}\|\mathcal{X}_{0}-\mathcal{X}\|_{F}.$$
When $k_{1}=k_{2}=0,$ $v_{cg}$ is reduced to $\frac{18\epsilon_{0}}{1-4\epsilon_{0}}<1$.
On the other hand, we have $lim_{\epsilon_{0}\rightarrow
0}(\gamma_{1}+\gamma_{2})=3k_{1}k_{2}.$ So if
$k_{1}k_{2}<\frac{1}{3},~\gamma_{1}+\gamma_{2}$ can be less than one
when $\epsilon_{0}$ is small. In particular, when $k_{1}=0.1$ and
$k_{2}=1,$ a sufficient condition for $\gamma_{1}+\gamma_{2}<1$ is
$\epsilon_{0}\leq 0.01$.
\end{theorem}

\begin{proof}
We  prove the results by induction.  Firstly, assume that for all
$j\leq l$, we have
\begin{align*}
\frac{\|\mathcal{X}_{j}-\mathcal{X}\|_{F}}{\sigma_{min}(\mathcal{X})}\leq
\frac{3p^{\frac{1}{2}}\varepsilon_{0}}{20\beta \log
n(1+\epsilon_{0}) }.
\end{align*}
Since  \eqref{e1}-\eqref{e3} are satisfied, Lemma
\ref{lem1} implies,
\begin{align*}
\|R_{\Omega}P_{\ten{T}_{j}}\|\leq \frac{10}{3}\beta
\log(n)(1+\epsilon_{0})p^{\frac{1}{2}}~~\text{and}~~
\|P_{\ten{T}_{j}}-p^{-1}P_{\ten{T}_{j}}R_{\Omega}P_{\ten{T}_{j}}\|\leq 4\epsilon_{0}.
\end{align*}
Thus the assumptions in Lemma \ref{lemn1} are satisfied for all $j\leq l$.

For the case $k=l+1$.
Noting that
$\mathcal{W}_{l}= \mathcal{X}_{l}+\alpha_{l}\mathcal{Q}_l$ in Algorithm 4, and  $\mathcal{X}_{l+1}=\ten{H}_{r}(\mathcal{W}_{l})$,
there holds
$$\|\mathcal{X}_{l}-\mathcal{X}\|_{F}\leq \|\mathcal{X}_{l}-\mathcal{W}_{l}\|_{F}+\|\mathcal{W}_{l}-\mathcal{X}\|_{F}\leq 2\|\mathcal{W}_{l}-\mathcal{X}\|_{F}.$$
 It follows that
\begin{align}
&\|\mathcal{X}_{l+1}-\mathcal{X}\|_{F} \leq
2\|\mathcal{X}_{l}+\alpha_{l}\mathcal{Q}_{l}-\mathcal{X}\|_{F}\nonumber\\
=&~2\|\mathcal{X}_{l}-\mathcal{X}-\alpha_{l}P_{\ten{T}_{l}}
R_{\Omega}(\mathcal{X}_{l}-\mathcal{X})+\alpha_{l}\beta_{l}P_{\ten{T}_{l}}(\mathcal{Q}_{l-1})\|_{F} \nonumber\\
= &~
 2\|(P_{\ten{T}_{l}}-\alpha_{l}P_{\ten{T}_{l}}R_{\Omega}P_{\ten{T}_{l}})(\mathcal{X}_{l}-\mathcal{X})+
 (\mathcal{I}-P_{\ten{T}_{l}})(\mathcal{X}_{l}-\mathcal{X})
 -\alpha_{l}P_{\ten{T}_{l}}R_{\Omega}(\mathcal{I}-P_{\ten{T}_{l}})(\mathcal{X}_{l}-\mathcal{X})\|_{F}\nonumber\\
 \leq &~ 2\|(P_{\ten{T}_{l}}-\alpha_{l}P_{\ten{T}_{l}}R_{\Omega}P_{\ten{T}_{l}})(\mathcal{X}_{l}-\mathcal{X})\|_{F}
 +2\|(\mathcal{I}-P_{\ten{T}_{l}})(\mathcal{X}_{l}-\mathcal{X})\|_{F}\nonumber\\
 &~+2\|\alpha_{l}P_{\ten{T}_{l}}R_{\Omega}(\mathcal{I}-P_{\ten{T}_{l}})(\mathcal{X}_{l}-\mathcal{X})\|_{F}+2|\alpha_{l}||\beta_{l}|
 \|P_{\ten{T}_{l}}(P_{l-1})\|_{F}\nonumber\\
 :=&~I_{1}+I_{2}+I_{3}+I_{4}.\label{ne1}
\end{align}
When one of the conditions in \eqref{eq6} is violated, and the restarting occurs, then $\beta_{l}$ is set $0$. In this case,
Lemmas \ref{lem5} and \ref{lem2} imply that
\begin{align*}
I_{1}&=2\|(P_{\ten{T}_{l}}-\alpha_{l}P_{\ten{T}_{l}}R_{\Omega}P_{\ten{T}_{l}})(\mathcal{X}_{l}-\mathcal{X})\|_{F}\leq
\frac{16\epsilon_{0}}{(1-4\epsilon_{0})}\|\mathcal{X}_{l}-\mathcal{X}\|_{F},\\
I_{2}&=2\|(\mathcal{I}-P_{\ten{T}_{l}})(\mathcal{X}_{l}-\mathcal{X})\|_{F}\leq\frac{\epsilon_{0}}{1-4\epsilon_{0}},\\
I_{3}&=2\|\alpha_{l}P_{\ten{T}_{l}}R_{\Omega}(\mathcal{I}-P_{\ten{T}_{l}})(\mathcal{X}_{l}-\mathcal{X})\|_{F}\leq \frac{\epsilon_{0}}{1-4\epsilon_{0}},\\
I_{4}&=0.
\end{align*}
Therefore, there holds
\begin{align*}
\|\mathcal{X}_{l+1}-\mathcal{X}\|_{F}\leq
\frac{18\epsilon_{0}}{1-4\epsilon_{0}}\|\mathcal{X}_{l}-\mathcal{X}\|_{F}\leq\big(\frac{18\epsilon_{0}}{1-4\epsilon_{0}}\big)^{l}\|\mathcal{X}_{0}-\mathcal{X}\|_{F}.
\end{align*}
When the conditions in \eqref{eq6} are satisfied, and the restarting does not occur and  the following results can be derived. For $I_{1}$,  by  Lemma \ref{lem2}, we have
\begin{equation*}
I_{1}=2\|(P_{\ten{T}_{l}}-\alpha_{l}P_{\ten{T}_{l}}R_{\Omega}P_{\ten{T}_{l}})(\mathcal{X}_{l}-\mathcal{X})\|_{F}\leq
(8\epsilon_{0}+2\epsilon_{\alpha}(1+2\epsilon_{0})\|\mathcal{X}_{l}-\mathcal{X}\|_{F}.
\end{equation*}
For $I_{2},$ noting that $P_{\ten{T}_{l}(\mathcal{X}_{l})}=\mathcal{X}_{l},$
 by the fifth inequality in Lemma \ref{lem5}, we have
\begin{align*}
I_{2}
=&2\|(\mathcal{I}-P_{\ten{T}_{l}})(\mathcal{X}_{l}-\mathcal{X})\|_{F}=2\|(\mathcal{I}-P_{\ten{T}_{l}})(\mathcal{X})\|_{F}
\leq\frac{\|\mathcal{X}_{l}-\mathcal{X}\|_{F}^{2}}{\sigma_{min}(\mathcal{X})}\\
\leq &\frac{3p^{\frac{1}{2}}\varepsilon_{0}}{10\beta \log
n(1+\epsilon_{0}) }\|\mathcal{X}_{l}-\mathcal{X}\|_{F}
 \leq
\frac{\epsilon_{0}}{1-4\epsilon_{0}}\|\mathcal{X}_{l}-\mathcal{X}\|_{F}
\leq(1+\epsilon_{\alpha})\epsilon_{0}\|\mathcal{X}_{l}-\mathcal{X}\|_{F}.
\end{align*}
For $I_{3}$,   noting the definition of $\varepsilon_{\alpha}$ and the
inequality (\ref{e5}) in Lemma \ref{lem1}, we have
\begin{equation*}
\mathcal{I}_{3}=2\|\alpha_{l}P_{\ten{T}_{l}}R_{\Omega}(\mathcal{I}-P_{\ten{T}_{l}})(\mathcal{X}_{l}-\mathcal{X})\|_{F}\leq
(1+\epsilon_{\alpha})\epsilon_{0}\|\mathcal{X}_{l}-\mathcal{X}\|_{F}.
\end{equation*}
For $I_{4}$, noting that
\begin{align*}
\beta_{l}P_{\ten{T}_{l}}(P_{l-1})=\sum_{j=1}^{l-1}(\Pi_{i=j+1}^{l}\beta_{i})(\Pi_{k=j}^{l-1}P_{\ten{T}_{k}})(\mathcal{G}_{j}),
\end{align*}
we have
\begin{align*}
\mathcal{I}_{4}&\leq2|\alpha_{l}|\sum_{j=0}^{l-1}(\Pi_{i=j+1}^{l}\beta_{i})\|(\Pi_{k=j}^{l-1})P_{\ten{T}_{k}}(\mathcal{G}_{j})\|_{F}\\
&\leq
2|\alpha_{l}|\sum_{j=0}^{l-1}\epsilon_{\beta}^{l-j}\|P_{\ten{T}_{j}}(\mathcal{G}_{j})\|_{F}=2|\alpha_{l}|\sum_{j=0}^{l-1}\epsilon_{\beta}^{l-j}\|P_{\ten{T}_{j}}R_{\Omega}(\mathcal{X}_{j}-\mathcal{X})\|_{F}\\
&\leq 2|\alpha_{l}|\sum_{j=0}^{l-1}\epsilon_{\beta}^{l-j}(\|P_{\ten{T}_{j}}R_{\Omega}P_{\ten{T}_{j}}(\mathcal{X}_{j}-\mathcal{X})\|_{F}
+\|P_{\ten{T}_{j}}R_{\Omega}(\mathcal{I}-P_{\ten{T}_{j}})(\mathcal{X}_{j}-\mathcal{X})\|_{F})\\
&\leq
2|\alpha_{l}|\sum_{j=0}^{l-1}\epsilon_{\beta}^{l-j}\Big(\|P_{\ten{T}_{j}}R_{\Omega}(\mathcal{X}_{j}-\mathcal{X})P_{\ten{T}_{j}}\|_{F}
+\|R_{\Omega}P_{\ten{T}_{j}}\|\frac{\|(\mathcal{X}_{j}-\mathcal{X})\|^2_{F}}{\sigma_{min}(\mathcal{X})}\Big)\\
&\leq \frac{(1+\epsilon_{\alpha})}{p}\sum_{j=0}^{l-1}\epsilon_{\beta}^{l-j}(2(1+4\epsilon_{0})+p\epsilon_{0})\|\mathcal{X}_{j}-\mathcal{X}\|_{F}\\
&=
(1+\epsilon_{\alpha})\sum_{j=0}^{l-1}\epsilon_{\beta}^{l-j}(2(1+4\epsilon_{0})+\epsilon_{0})\|\mathcal{X}_{j}-\mathcal{X}\|_{F},
\end{align*}
where the second inequality follows from
\begin{equation*}
\|(\Pi_{k}^{l-1}P_{\ten{T}_{k}})(\mathcal{G}_{j})\|_{F}\leq
\|P_{\ten{T}_{j}}(\mathcal{G}_{j})\|_{F}.
\end{equation*}
Taking $I_{1}-I_{4}$ into (\ref{ne1}) yields
\begin{align*}
&\|\mathcal{X}_{l+1}-\mathcal{X}\|_{F}\\
\leq~&
(10\epsilon_{0}+2\epsilon_{\alpha}(1+5\epsilon_{0}))\|\mathcal{X}_{l}-\mathcal{X}\|_{F}
+(1+\epsilon_{\alpha})(2(1+4\epsilon_{0})+\epsilon_{0})\sum_{j=0}^{l-1}\epsilon_{\beta}^{l-j}\|\mathcal{X}_{j}-\mathcal{X}\|_{F}\\
\leq~&(10\epsilon_{0}+2\epsilon_{\alpha}(1+5\epsilon_{0}))\|\mathcal{X}_{l}-\mathcal{X}\|_{F}
+(2+10\epsilon_{0}+2\epsilon_{\alpha})(1+5\epsilon_{0})\sum_{j=0}^{l-1}\epsilon_{\beta}^{l-j}\|\mathcal{X}_{j}-\mathcal{X}\|_{F}\\
:=~&\rho_{1}\|\mathcal{X}_{l}-\mathcal{X}\|_{F}+\rho_{2}\sum_{j=0}^{l-1}\epsilon_{\beta}^{l-j}\|\mathcal{X}_{j}-\mathcal{X}\|_{F}.
\end{align*}
Define
\begin{align*}
&\gamma_{1}=\rho_{1}+\epsilon_{\beta}=\frac{18\epsilon_{0}-10\epsilon_{0}k_{1}(1+4\epsilon_{0})}{(1-4\epsilon_{0})-k_{1}(1+4\epsilon_{0})}+\frac{4k_{2}\epsilon_{0}+k_{1}k_{2}}{1-4\epsilon_{0}},\\
&\gamma_{2}=(\rho_{2}-\rho_{1})\epsilon_{\beta}=\frac{8k_{2}\epsilon_{0}+2k_{1}k_{2}}{1-4\epsilon_{0}},\\
&v_{cg}=\frac{1}{2}(\gamma_{1}+\sqrt{\gamma_{1}^2+\gamma_{2}}).
\end{align*}
When $l=0$, it is easy to get
\begin{align*}
\|\mathcal{X}_{1}-\mathcal{X}\|_{F}\leq
\frac{18\varepsilon_{0}}{1-4\varepsilon_{0}}\|\mathcal{X}_{0}-\mathcal{X}\|_{F}\leq
v_{cg}\|\mathcal{X}_{0}-\mathcal{X}\|_{F}.
\end{align*}
It follows Lemma \ref{wei4} that if $\gamma_{1}+\gamma_{2}<1$,
we have $v_{cg}<1$, then
\begin{align*}
\|\mathcal{X}_{l+1}-\mathcal{X}\|_{F}\leq
v_{cg}^{l+1}\|\mathcal{X}_{0}-\mathcal{X}\|_{F},
\end{align*}
which completes the proof.
\end{proof}

We remark that
 the manifold $\ten{M}_{\textbf{r}}$ is not closed, and the closure of $\ten{M}_{\textbf{r}}$ are bounded by transformed  multi-rank $\textbf{r}$ (point bounded) tensors. Then a sequence of tensors in $\ten{M}_{\textbf{r}}$
 may approach a tensor which the $i$th transformed tubal rank less than $\textbf{r}_{i},$ which is saying that the limit of Algorithm 2
 may not be
 in the manifold $\ten{M}_{\textbf{r}}$. In order to avoid this statements happens, we need the following results.

 \begin{proposition}[Proposition 4.1 in \cite{bart}]
 Let $\{\mathcal{X}_{i}\}$ be an infinite sequence of iterates
 generated by Algorithm \ref{alg1}. Then, every accumulation point $\mathcal{X_{\ast}}$ of $\{\mathcal{X}_{i}\}$ satisfies
 $P_{T_{\mathcal{X}_{\ast}}}R_{\Omega}(\mathcal{X}_{\ast})=P_{T_{\mathcal{X}_{\ast}}}R_{\Omega}(\mathcal{A})$.
 \end{proposition}

\section{The Initialization}

In Theorem \ref{themm}, there are three conditions \eqref{e1}-\eqref{e3} to guarantee the
convergence of Alg. \ref{alg1}. The requirement for $R_{\Omega}$
 to be bounded in (\ref{e1})  is just an requirement of the sampling model,
and by Lemma \ref{le2} it can be satisfied with probability at least $1-n^{3-3\beta}$.
The second condition (\ref{e2}) plays a key role in
nuclear norm minimization for tensor completion which have been proved in  \cite{gross2010,Recht2011,wei} under different sampling models.
For tensor case, it also can be seen as a local restricted isometry
property and has been established in \cite{jiangm}  for the Bernoulli model.
In our setting, we consider the sampling with replacement model instead of without
replacement model, then
by Lemma \ref{le1} we have \eqref{e2} satisfied with
probability at least $1-2(n_{(1)}n_{3})^{1-\frac{5}{4}\beta}$, as
long as $m\geq \frac{20}{3}\beta \mu r n_{(1)}n_{3} \log(n_{(1)}n_3)$.
Thus the only issue that
remains to be addressed is how to produce an initial guess that is
sufficiently close to the original tensor. In this section, we will consider two
initialization strategies.

\subsection{Hard Thresholding}

Based on the transformed tensor incoherence
conditions given in (\ref{eq17}), we can prove the
following lemma which will be used many times in the proofs of Lemmas \ref{le1} and \ref{lemma2}.

\begin{lemma} \label{lem6.3}
Let $\mathcal{A} \in \R^{n_1 \times n_2 \times n_3}$ be an arbitrary
tensor, and $\ten{T}$ be given as (\ref{xin1}). Suppose that the transformed
tensor incoherence conditions (\ref{eq17}) are
satisfied,
then
$$\|P_{\ten{T}}(\mathcal{E}_{abc})\|^{2}_{F}\leq \frac{2\mu_{0} r}{n_{(2)}}.$$
\end{lemma}

\begin{proof}
By the definition of $\PT$, it suffices to show $\|\mathcal{U}^T\ast \Eabc\|_F^2\leq \mu_{0} r/n_1$
and $\| \Eabc\ast\mathcal{V}\|_F^2\leq \mu_{0} r/n_2$. We only prove the first one and the second one can be proved similarly. Simple calculation shows that
\begin{align*}
\|\mathcal{U}^T\ast \Eabc\|_F^2 & = \|\overline{\mathcal{U}}^T\overline{\Eabc}\|_F^2 = \sum_{k=1}^{n_3}\|\mu_c^k\cdot\overline{\mathcal{U}}_k^T\be_a\be_b^T\|_2^2=\sum_{k=1}^{n_3}\|\mu_c^k\cdot\overline{\mathcal{U}}_k^T\be_a\|_2^2,
\end{align*}
where $[\mu_c^1,\cdots,\mu_c^{n_3}]^T$ is the $c$-th column of the DCT matrix. Noting that the magnitude of each entry of the DCT matrix is bounded by $2/\sqrt{n_3}$, $\|\mathcal{U}^T\ast \Eabc\|_F^2\leq \mu_{0} r/n_1$  follows immediately from the tensor incoherence condition.
\end{proof}

Denote $n_{(1)}=\max\{n_{1},n_{2}\},n_{(2)}=\min\{n_{1},n_{2}\}$.
Jiang \cite{jiangm} derived some lemmas based on Bernoulli model sampling without replacement model. For sampling with replacement model, we can get the following
results.

\begin{lemma} \label{le1}
Suppose $\mathcal{X} \in \R^{n_1 \times n_2 \times n_3}$ is a fixed
tensor, and $\Omega$ with $|\Omega|=m$ is a set of indices sampled
independently and uniformly with replacement. Let
$\mathcal{X}=\mathcal{U} \ast \mathcal{S} \ast \mathcal{V}^{T} $  be a tubal rank $r$ tensor which satisfy the incoherence conditions given in \eqref{eq17}.  Then for all $\beta>1$,
$$
\big\|\frac{n_{1}n_{2}n_{3}}{m}\mathcal{P}_{\ten{T}}R_{\Omega}\mathcal{P}_{\ten{T}}-\mathcal{P}_{\ten{T}}\big\|_{op}\lesssim \sqrt{\frac{\mu_{0} r n_{(1)}n_{3}\beta \log(n_{(1)}n_{3})}{m}},
$$
holds with high probability
with the condition that
$m\gtrsim \beta \mu_{0} r n_{(1)}n_{3} \log(n_{(1)}n_3)$.
\end{lemma}

\begin{lemma}\label{lemma2}
Suppose $\mathcal{Z} \in \R^{n_1 \times n_2 \times n_3}$ is a fixed
tensor, and $\Omega$ with $|\Omega|=m$ is a set of indices sampled
independently and uniformly with replacement. Then for all $\beta>1$
\begin{equation*}
\|(\OpId - \frac{n_1 n_2 n_3}{m}R_{\Omega})\mathcal{Z}\| \lesssim
\sqrt{\frac{\beta
n^2_{(1)}n_{(2)}n_{3}\log(n_{(1)}n_3)}{m}}\|\mathcal{Z}\|_{\infty},
\label{eq25}
\end{equation*}
holds with high probability
with the condition that
$m \gtrsim  \beta n_{(1)}n_{3}\log(n_{(1)}n_3)$.
\end{lemma}

\begin{lemma}\label{lemm5}
Suppose that $|\Omega|=m$ is a set of indices sampled independently
and uniformly with replacement. Let
$\mathcal{X}_{0}=\ten{H}_{r}(p^{-1}R_{\Omega}(\mathcal{X})).$ Then
\begin{align}\label{e7}
\|\mathcal{X}_{0}-\mathcal{X}\|_{F}\lesssim
\sqrt{\frac{\mu_{1}^2r^2\beta
n_{(1)}n_3\log(n_{(1)}n_3)}{m}}\|\mathcal{X}\|
\end{align}
holds with high probability
with the condition that $m\gtrsim  \beta n_{(1)} \log n_{(1)}n_{3}$.
\end{lemma}

Then we can establish the following theorem.

\begin{theorem}\label{them2}
Let $\mathcal{X}=\mathcal{U} \ast \mathcal{S} \ast
\mathcal{V}^{T} \in \R^{n_{1}\times n_{2}\times n_{3}}$ with $\rank_{t}(\mathcal{X})= \textbf{r} ~\text{and}~ \rank_{ct}(\mathcal{X})= r$. Suppose that
$|\Omega|=m$ is a set of indices sampled independently and uniformly
with replacement. Let
$\mathcal{X}_{0}=\ten{H}_{r}(\rho^{-1}R_{\Omega}(\mathcal{X}))$.  Then the iterates generated by Algorithm 2
are guaranteed to converge to $\mathcal{X}$ with
high probability provided
\begin{equation}\label{new3}
m\gtrsim  \max
\left\{\frac{\mu_{0}r}{\epsilon_{0}^{2}}\log^{\frac{1}{2}}
n_{(1)}n_{3},\frac{\mu_{1}r(1+\epsilon_{0})\kappa\beta^{\frac{1}{2}}}{\epsilon_{0}}n_{(2)}^{\frac{1}{2}}\log
n \right\}\beta n_{(1)}n_{3} \log^{\frac{1}{2}}
n_{(1)}n_{3}.
\end{equation}
where $\kappa$ is the condition number of $\mathcal{X}$.
\end{theorem}

\begin{proof}
It is readily seen that when \eqref{new3} is satisfied, the assumptions in Lemma \ref{lem1} hold. Thus this theorem follows immediately from Theorem \ref{themm}.
\end{proof}

\subsection{Resampling and Trimming}

In Theorem \ref{them2}, the sampling numbers depends on
$n^{\frac{5}{2}}$, however, in \cite{jiangm} the sampling numbers
can be reduced to $n^{2}$.
Then we need to find some new
initialization scheme to reduce the sampling numbers theoretically.
In this subsection, we generalize the trimming procedure which used in matrix
case \cite{wei} to tensor case, and the specific details can be found in Algorithm 3.
Moreover, when the resampling scheme
breaks the dependence between the past iterate and the new sampling set, we apply the tensor trimming method (Algorithm 4) to
to project the estimate onto the set of $\mu$-incoherent tensors. After that, we need to prove the output of Algorithm 3
reaches a neighborhood
of the original tensor where Theorem \ref{themm} is activated. First of all, we need the following lemmas.

\begin{algorithm}[h]
\caption{Initialization via Tensor Resampled Riemannian Gradient
Descent and Tensor Trimming} \label{alg2}
\textbf{Partition} $\Omega$ \text{into} $L+1$ equal groups: $\Omega_{0},\cdots, \Omega_{L};$ and the size of every group is denoted by $\hat{m}$. \\
\textbf{Set} $\mathcal{Z}_{0}=\ten{H}_{r}(\frac{\hat{m}}{n^3}R_{\Omega_{0}}(\mathcal{X}))$ \\
\textbf{for} $l=0,...,L-1$ do \\
~~1: $\hat{\mathcal{Z}}_{l}=\text{trim}(\mathcal{Z}_{l})$;\\
~~2: $\mathcal{Z}_{l+1}=\ten{H}_{r}(\hat{\mathcal{Z}}_{l}+\frac{n^{3}}{\hat{m}}P_{\hat{\ten{T}}_{l}}R_{\Omega_{l+1}}(\mathcal{X}-\hat{\mathcal{Z}}_{l}))$;\\
  \textbf{end for}\\
  \textbf{Output:} $\mathcal{X}_{0}=\mathcal{Z}_{l}$
\end{algorithm}

\begin{algorithm}[h]
\caption{ Tensor Trimming} \label{alg3}
\textbf{Input} $\mathcal{Z}_{l}=\mathcal{U}_{l}\ast\mathcal{S}_{\texttt{r}}\ast\mathcal{V}_{l}^{T}$. \\
\textbf{Output:} $\hat{\mathcal{Z}}_{l}=\mathcal{A}_{l}\ast\mathcal{S}_{\texttt{r}}\ast\mathcal{B}_{l}^{T},$ where, \\
~~~~~~~~~~~~~~~
$\mathcal{A}_{l}^{[i]}=\frac{\mathcal{U}_{l}^{[i]}}{\|\mathcal{U}_{l}^{[i]}\|_{F}}\min\{\|\mathcal{U}_{l}^{[i]}\|_{F},\sqrt{\frac{\mu r}{n_{1}}}\}$;
 $\mathcal{B}_{l}^{[i]}=\frac{\mathcal{V}_{l}^{[i]}}{\|\mathcal{V}_{l}^{[i]}\|_{F}}\min\{\|\mathcal{V}_{l}^{[i]}\|_{F},\sqrt{\frac{\mu r}{n_{2}}}\}.$\\
\end{algorithm}

\begin{lemma}\label{lemm4}
Let $\mathcal{X},\mathcal{X}_{l}\in \R^{n_{1}\times n_{2}\times n_{3}}$ be two tensors with $\rank_{t}(\mathcal{X})=\rank_{t}(\mathcal{X}_{l}) = \textbf{r} ~\text{and}~ \rank_{ct}(\mathcal{X})=\rank_{ct}(\mathcal{X}_{l}) = r$. Given their skinny t$_{c}$-SVD  $\mathcal{X}_{l}=\mathcal{U}_{l} \ast \mathcal{S}_{l}  \ast
\mathcal{V}^{T}_{l}$ and $\mathcal{X}=\mathcal{U} \ast \mathcal{S} \ast
\mathcal{V}^{T},$ let  $\ten{T}$ and $\ten{T}_{l}$  be the
tangent spaces of the fixed transformed  multi-rank  manifold at
$\mathcal{X},\mathcal{X}_{l}$, respectively. Assume $\Omega$ with $|\Omega|=m$
is a set of indices sampled independently and uniformly with
replacement. Suppose that
\begin{align*}
&\|P_{\mathcal{U}_{l}}\ast\tc{e}_i\|_{F}\leq \sqrt{\frac{\mu_{0} r}{n_1}},~ \|P_{\mathcal{V}_{l}}\ast\tc{e}_j\|_{F}\leq \sqrt{\frac{\mu_{0} r}{n_2}},~
 \|P_{\mathcal{U}}\ast\tc{e}_i\|_{F}\leq \sqrt{\frac{\mu_{0} r}{n_1}},~\text{and}~
\|P_{\mathcal{V}}\ast\tc{e}_j\|_{F}\leq \sqrt{\frac{\mu_{0} r}{n_2}}
\end{align*}
hold for all $1\leq i,j\leq n$.  Then for any $\beta>1$,
\begin{align*}
\Big\|\frac{n_{1}n_{2}n_{3}}{m}P_{\ten{T}_{l}}R_{\Omega}(P_{\mathcal{U}}-P_{\mathcal{U}_{l}})-P_{\ten{T}_{l}}(P_{\mathcal{U}}-P_{\mathcal{U}_{l}})\Big\|
\lesssim \sqrt{\frac{\beta \mu_{0} n_{(1)}n_{3} r \log (n_{(1)}n_{3})}{m}}
\end{align*}
holds with high probability provided $m\gtrsim \beta n_{(1)}n_{3} r \log n_{(1)}n_{3}$.
\end{lemma}

\begin{lemma}\label{lemm2}
Let $\mathcal{U}_{l},~\mathcal{U}\in \R^{n_{1}\times r\times n_{3}}$
be two orthogonal tensors. Then there exist an
orthogonal tensor $\mathcal{Q}\in \R^{r\times r \times n_{3}}$ such that
\begin{align*}
\|\mathcal{U}_{l}-\mathcal{U}\ast \mathcal{Q}\|_{F}\leq
\|\mathcal{U}_{l}\ast\mathcal{U}_{l}^{T}-\mathcal{U}\ast
\mathcal{U}^{T}\|_{F}.
\end{align*}
\end{lemma}

\begin{lemma} \label{lemm3}
Suppose that $\mathcal{Z}_{l}\in \R^{n_{1}\times n_{2}\times n_{3}}$ satisfies $\mathcal{Z}_{l}=\mathcal{U}_{l}\ast \mathcal{S}_{l}\ast
\mathcal{V}_{l}^{T}$ and
$$\|\mathcal{Z}_{l}-\mathcal{X}\|_{F}\leq \frac{\sigma_{min}(\mathcal{X})}{10\sqrt{2}},$$
where $rank_{t}(\mathcal{U}_{l})=rank_{t}(\mathcal{V}_{l})=rank_{t}(\mathcal{Z}_{l})= \textbf{r}$.
Then for $1\leq i\leq n_1, 1\leq j\leq n_2$ the tensor $\hat{\mathcal{Z}}_{l}=\hat{\mathcal{U}}_{l}\ast\hat{\mathcal{S}}_{\textbf{r}}\ast\hat{\mathcal{V}}_{l}^{T}$ returned by Algorithm 3
satisfies
\begin{align*}
\|\hat{\mathcal{U}}_{l}\ast
\tc{e}_i\|_{F}\leq\frac{10}{9}\sqrt{\frac{\mu_{0} r}{n_{1}}}~\text{and}~\|\hat{\mathcal{V}}_{l}\ast
\tc{e}_j\|_{F}\leq\frac{10}{9}\sqrt{\frac{\mu_{0} r}{n_{2}}}.
\end{align*}
Moreover, lettiing
$\kappa=\frac{\sigma_{max}(\mathcal{X})}{\sigma_{min}(\mathcal{X})}$,
then
\begin{align*}
\|\hat{\mathcal{Z}}_{l}-\mathcal{X}\|_{F}\leq 8\kappa
\|\mathcal{Z}_{l}-\mathcal{X}\|_{F}.
\end{align*}
\end{lemma}

With the tools in hand, we can prove the following lemma which plays a key role in deciding the bound on the number of sample entries required for tensor completion.

\begin{lemma}\label{lemm6}
Suppose $\mathcal{X}\in \R^{n_{1}\times n_{2}\times n_{3}}$ with $\rank_{t}(\mathcal{X}) = \textbf{r} ~\text{and}~ \rank_{ct}(\mathcal{X}) = r$, $\kappa$ is the condition number of $\mathcal{X}$ and $L$ is defined as in Algorithm 3.  Then the output of Algorithm 3 satisfies
\begin{align*}
\|\mathcal{X}_{0}-\mathcal{X}\|_{F}\leq
\left(\frac{5}{6}\right)^{L}\frac{\sigma_{min}(\mathcal{X})}{256\kappa^{2}}
\end{align*}
with high probability
provided
\begin{align*}
\hat{m} \gtrsim
\max\{\mu_0\beta r,\mu_{1}^{2}r^2\kappa^{6}\}n_{(1)}n_{3}\log(n_{(1)}n_{3}).
\end{align*}
\end{lemma}

\begin{proof}
Assume that
\begin{align} \label{eq3}
\|\mathcal{Z}_{l}-\mathcal{X}\|_{F}\leq
\frac{\sigma_{min}(\mathcal{X})}{256\kappa^{2}}.
\end{align}
It follows from Lemma \ref{lemm3} that $\mathcal{\hat{Z}}_{l}$ is an
incoherent tensor with incoherence parameter $\frac{100}{81}\mu$ and
\begin{align*}
\|\mathcal{\hat{Z}}_{l}-\mathcal{X}\|_{F}\leq 8\kappa
\|\mathcal{Z}_{l}-\mathcal{X}\|_{F}.
\end{align*}
The approximation error at the $(l+1)th$ iteration can be
decomposed as
\begin{align*}
\|\mathcal{Z}_{l+1}-\mathcal{X}\|_{F} &\leq
2\|(P_{\mathcal{\hat{T}}_{l}}-\frac{n^{2}}{\hat{m}}P_{\mathcal{\hat{T}}_{l}}P_{\Omega_{l+1}}P_{\mathcal{\hat{T}}_{l}})(\mathcal{\hat{Z}}_{l}-\mathcal{X})\|_{F}\\
&+2\|(\mathcal{I}-P_{\mathcal{\hat{T}}_{l}})(\mathcal{\hat{Z}}_{l}-\mathcal{X})\|_{F}
+2\|\frac{n^{2}}{\hat{m}}P_{\mathcal{\hat{T}}_{l}}P_{\Omega_{l+1}}(\mathcal{I}-P_{\mathcal{\hat{T}}_{l}})(\mathcal{\hat{Z}}_{l}-\mathcal{X})\|_{F}\\
&:=I_{5}+I_{6}+I_{7}.
\end{align*}
It follows Lemma \ref{le1} that
\begin{align*}
\|P_{\mathcal{\hat{T}}_{l}}-\frac{n^{2}}{\hat{m}}P_{\mathcal{\hat{T}}_{l}}P_{\Omega_{l+1}}P_{\mathcal{\hat{T}}_{l}}\|\lesssim
\sqrt{\frac{\mu_{0}  r n_{(1)}n_{3}\beta
\log(n_{(1)}n_{3})}{\hat{m}}}
\end{align*}
with high probability. Thus
\begin{align*}
I_{5}\lesssim
\kappa\sqrt{\frac{\mu_{0}  r n_{(1)}n_{3}\beta
\log(n_{(1)}n_{3})}{\hat{m}}}\|\mathcal{Z}_{l}-\mathcal{X}\|_{F}.
\end{align*}
By Lemma \ref{lem5}, we have
\begin{align*}
I_{6}\leq
\frac{2\|\mathcal{\hat{Z}}_{l}-\mathcal{X}\|^2_{F}}{\sigma_{min}(\mathcal{X})}\leq
\frac{128\kappa^2\|\mathcal{Z}_{l}-\mathcal{X}\|^{2}_{F}}{\sigma_{min}(\mathcal{X})}\leq
\frac{1}{2}\|\mathcal{Z}_{l}-\mathcal{X}\|_{F}.
\end{align*}
Note that $\hat{\mathcal{Z}}_{l}$ is independent of
$\Omega_{l+1}$ with the incoherence parameter $\frac{100}{81}\mu_{0},$
then it follows Lemma \ref{lemm4} that
\begin{align*}
\Big\|\frac{n_{1}n_{2}n_{3}}{\hat{m}}P_{\mathcal{T}_{l}}R_{\Omega_{l+1}}(P_{\mathcal{U}}-P_{\mathcal{U}_{l}})
-P_{\mathcal{T}_{l}}(P_{\mathcal{U}}-P_{\mathcal{U}_{l}})\Big\|
\lesssim \sqrt{\frac{\beta \mu_{0} n_{(1)}n_{3} r \log
(n_{(1)}n_{3})}{\hat{m}}}
\end{align*}
with high probability.
Moreover, due to $\mathcal{X}=\mathcal{U}\ast\mathcal{U}^{T}\ast
\mathcal{X}$ and
$P_{\mathcal{\hat{T}}_{l}}(\mathcal{\hat{Z}}_{l})=\mathcal{\hat{Z}}_{l},$ we have
\begin{align*}
&(\mathcal{I}-P_{\mathcal{\hat{T}}_{l}})(\mathcal{\hat{Z}}_{l}-\mathcal{X})=-(\mathcal{I}-P_{\mathcal{\hat{T}}_{l}})(\mathcal{X})\\
=&~-\mathcal{U}\ast \mathcal{U}^{T}\ast
\mathcal{X}+\hat{\mathcal{U}}_{l}\ast \hat{\mathcal{U}}_{l}^{T}\ast
\mathcal{X}+\mathcal{U}\ast \mathcal{U}^{T}\ast
\mathcal{X}\ast\hat{\mathcal{V}}_{l}\ast
\hat{\mathcal{V}}_{l}^{T}-\hat{\mathcal{U}}_{l}\ast
\hat{\mathcal{U}}_{l}^{T}\ast
\mathcal{X}\ast\hat{\mathcal{V}}_{l}\ast \hat{\mathcal{V}}_{l}^{T}\\
=&~-(\mathcal{U}\ast \mathcal{U}^{T}-\hat{\mathcal{U}}_{l}\ast
\hat{\mathcal{U}}_{l}^{T})\ast
\mathcal{X}\ast(\mathcal{I}-\hat{\mathcal{V}}_{l}\ast
\hat{\mathcal{V}}_{l}^{T})\\
=&~(P_{\mathcal{U}}-P_{\hat{\mathcal{U}}_{l}})\ast(\mathcal{\hat{Z}}_{l}-\mathcal{X})\ast(\mathcal{I}-P_{\hat{\mathcal{V}}_{l}}).
\end{align*}
Together with
\begin{align*}
P_{\hat{\mathcal{T}}_{l}}((P_{\mathcal{U}}-P_{\hat{\mathcal{U}}_{l}})\ast(\mathcal{\hat{Z}}_{l}-\mathcal{X})\ast(\mathcal{I}-P_{\hat{\mathcal{V}}_{l}}))
=P_{\hat{\mathcal{T}}_{l}}((\mathcal{I}-P_{\mathcal{\hat{T}}_{l}})(\mathcal{\hat{Z}}_{l}-\mathcal{X}))=0,
\end{align*}
 we can bound $I_{7}$ as follows,
\begin{align*}
I_{7}
&=2\left\|\frac{n_{1}n_{2}n_{3}}{\hat{m}}P_{\mathcal{\hat{T}}_{l}}P_{\Omega_{l+1}}(\mathcal{I}-P_{\mathcal{\hat{T}}_{l}})(\mathcal{\hat{Z}}_{l}-\mathcal{X})\right\|_{F}\\
&=2\left\|\frac{n_{1}n_{2}n_{3}}{\hat{m}}P_{\mathcal{\hat{T}}_{l}}P_{\Omega_{l+1}}(\mathcal{I}-P_{\mathcal{\hat{T}}_{l}})(\mathcal{\hat{Z}}_{l}-\mathcal{X})-
P_{\hat{\mathcal{T}}_{l}}(\mathcal{I}-P_{\mathcal{\hat{T}}_{l}})(\mathcal{\hat{Z}}_{l}-\mathcal{X})\right\|_{F}\\
&\lesssim
\left\|\frac{n_{1}n_{2}n_{3}}{\hat{m}}(P_{\mathcal{\hat{T}}_{l}}P_{\Omega_{l+1}}-P_{\hat{\mathcal{T}}})(P_{\mathcal{U}}-P_{\hat{\mathcal{U}}_{l}})\right\|\|\mathcal{\hat{Z}}_{l}-\mathcal{X}\|_{F}\\
&\lesssim \sqrt{\frac{\beta \mu_{0}  n_{(1)}n_{3} r \log (n_{(1)}n_{3})}{\hat{m}}}\|\mathcal{\hat{Z}}_{l}-\mathcal{X}\|_{F}\\
&\lesssim \kappa\sqrt{\frac{\beta \mu_{0}  n_{(1)}n_{3} r \log (n_{(1)}n_{3})}{\hat{m}}}\|\mathcal{Z}_{l}-\mathcal{X}\|_{F}.\\
\end{align*}
Combining them together gives
\begin{align*}
\|\mathcal{Z}_{L+1}-\mathcal{X}\|_{F}
\leq\left(\frac{1}{2}+96\kappa\sqrt{\frac{\beta\mu_{0} n_{(1)}n_{3}r\log(n_{(1)}n_{3})}{\hat{m}}}\right)\|\mathcal{Z}_{l}-\mathcal{X}\|_{F}
\leq \frac{5}{6}\|\mathcal{Z}_{l}-\mathcal{X}\|_{F}
\end{align*}
holds with high probability
provided
\begin{equation} \label{equ3}
\hat{m}\gtrsim \beta\mu_{0}\kappa^{2}n_{(1)}n_{3}r\log(n_{(1)}n_{3}).
\end{equation}
 Noting that
$\mathcal{Z}_{0}=\ten{H}_{r}(\frac{\hat{m}}{n_{1}n_{2}n_{3}}P_{\Omega_{0}}(\mathcal{X}))$
and
\begin{align*}
\frac{\sigma_{min}(\mathcal{X})}{256\kappa^{2}}
=\frac{\|\mathcal{X}\|}{256\kappa^{3}},~~\|\mathcal{X}\|_{\infty}\leq
\mu_{1}\sqrt{\frac{ r}{n_1 n_2 n_3}}\|\mathcal{X}\|,
\end{align*}
by Lemma \ref{lemma2}, we can obtain
\begin{align*}
 \PP\Big[\|\rho^{-1}\PO(\mathcal{X}) - \mathcal{X}\|>
\frac{1}{256\mu_{1}\kappa^{3}}\sqrt{\frac{n_{1}n_{2}}{r^{2}}}\|\mathcal{X}\|_{\infty}\Big]
 &\leq  2n_{(1)}n_3 \exp\Bigg(\frac{\frac{-cn_{1}n_{2}}{\mu^2_{1}\kappa^{6}r^2}\|\mathcal{X}\|^2_{\infty}}{\frac{n_{(1)}\|\mathcal{X}\|^2_{\infty}}{\rho}}\Bigg)
  = 2(n_{(1)}n_3)^{1-c},
\end{align*}
with the proviso that
\begin{equation} \label{t1}
\hat{m}\gtrsim \beta\mu_{1}^2\kappa^6 r^2 n_{(1)}n_{3}\log(n_{(1)}n_{3}).
\end{equation}
It follows that
\begin{align*}
\|\mathcal{Z}_{0}-\mathcal{X}\|_{F}\leq&\sqrt{n_{3}r}\|\mathcal{Z}_{0}-\mathcal{X}\|
\leq \sqrt{n_{3}r}\frac{1}{256\mu_{1}\kappa^{3}}\sqrt{\frac{n_{1}n_{2}}{r^2}}\|\mathcal{X}\|_{\infty}
\leq \frac{\sigma_{min}(\mathcal{X})}{256\kappa^{2}}
\end{align*}
Therefore taking a maximum of the right hand sides of \eqref{equ3} and \eqref{t1} gives
\begin{align*}
\|\mathcal{Z}_{L}-\mathcal{X}\|_{F}\leq\left(\frac{5}{6}\right)^{L}
\frac{\sigma_{min}(\mathcal{X})}{256\kappa^{2}}
\end{align*}
with high probability provided
$\hat{m}\gtrsim
\max\{\mu_0 r\kappa^{2},\mu_{1}^{2}r^2\kappa^{6}\}\beta n_{(1)}n_{3}\log(n_{(1)}n_{3})$.
\end{proof}

In Algorithm 3, we use fixed stepsize $\frac{n_{1}n_{2}n_{3}}{\hat{m}}$ for ease of
exposition, which can be replaced by the adaptive stepsize. Lemma
\ref{lemm6} implies that if we take $L\geq 6\log\left(\frac{\beta n
\log n}{24 \varepsilon_{0}}\right)$,  the condition (\ref{e3})
in Theorem \ref{themm} can be satisfied with probability at least
$$1-2(n_{(1)}n_3)^{1-c}
-12\log\left(\frac{\beta n
\log n}{24 \varepsilon_{0}}\right)((n_{(1)}n_{3})^{1-\frac{5}{4}\beta}+(n_{(1)}n_{3})^{1-\frac{3}{2}\beta}).$$ Combining the above results, we can
obtain the followings.

\begin{theorem}\label{them3}
Suppose $\mathcal{X}_{l}\in \R^{n_{1}\times n_{2}\times n_{3}}$
with $\rank_{t}(\mathcal{X}_{l}) = \textbf{r} ~\text{and}~ \rank_{ct}(\mathcal{X}_{l}) = r,$ $\kappa$ is the condition number of $\mathcal{X}.$ Let $\Omega$ $(|\Omega|=m)$ is a set of indices sampled independently and uniformly
with replacement. Let $\mathcal{X}_{0}$ be the output of Algorithm 2. Then the iterates generated by Alg. \ref{alg2} is  guaranteed to converge to $\mathcal{X}$
with high probability
 provided
\begin{equation*}
m\gtrsim
\max\{\frac{\mu_{0} r\kappa^{2}}{\epsilon_{0}^2},\mu_{1}^2 r^2\kappa^{6}\}\beta n_{(1)}n_{3}\log(n_{(1)}n_{3})\log\Big(\frac{\beta
n \log n}{24\epsilon_{0}}\Big).
\end{equation*}
\end{theorem}
\begin{proof}
This theorem follows from Lemma~\ref{lem1}, Theorem~\ref{themm}, and Lemma~\ref{lemm6}.
\end{proof}

\section{Numerical Results}

In this section, numerical results are presented to show the effectiveness of the proposed methods
(Algorithms 2 and 3)
for tensor completion. We also compare our methods with the t-svd method introduced in \cite{Zhang2017}.  All
the experiments are performed under Windows 7 and MATLAB R2018a
running on a desktop (Intel Core i7, @ 3.40GHz, 8.00G RAM)

The relative error (Res)  is defined by
\[\text{Res}=\frac{\|\mathcal{X}-\mathcal{X}_{0}\|_F}{\|\mathcal{X}_{0}\|_{F}},\]
where $\mathcal{X}$ is the recovered solution and $\mathcal{X}_{0}$ is the ground-truth tensor. Moreover, in order to evaluate the performance for real-world tensors, the peak signal-to-noise ratio (PSNR) is used to measure the equality of the estimated tensors, which is defined as follows:
\[\text{PSNR}=10\log_{10}\frac{n_1n_2n_3(\mathcal{X}_{\max}-\mathcal{X}_{\min})^2}{\|\mathcal{X}-\mathcal{X}_0\|_F},\]
where $\mathcal{X}_{\max}$ and $\mathcal{X}_{\min}$ are maximal and minimal entries of
$\mathcal{X}_{0}$, respectively.  The stopping criterion of the
algorithm is set to \[\frac{\|\mathcal{X}_{l+1}-\mathcal{X}_l\|_F}{\|\mathcal{X}_l\|_F}\leq 10^{-4}.\]

\subsection{Synthetic Data}

In this subsection, we first verify the correct recovery phenomenon of Algorithms 2 and 3 for synthetic datasets.
For simplicity, we consider the tensors of size $n\times n \times n$ with dimension $n\leq 100.$ We generate the clean tensor $\mathcal{X}_{0}=\mathcal{S}*\mathcal{W}$ with tubal rank ${\rm{rank}}_{ct}(\mathcal{X}_{0})=r$, where the entries of $\mathcal{S}\in \mathbb{R}^{n\times r\times n}$ and $\mathcal{W}\in \mathbb{R}^{r\times n\times n}$ are independently sampled from a standard Gaussian distribution $\mathcal{N}(0,1)$.
We check the recovery abilities of our Algorithms 2 and 3 as a function of tensor dimension $n,$ the tubal rank $r$, and the sampling size $m$.
We set $r$ to be different specified values, and vary $n$ and $m$ to empirically investigate the probability of recovery success. For each pair $n$ and $m$, we simulate 10 test instances and declare a trial to be successful if the recovered tensor $X_l$ satisfies $\frac{\|\mathcal{X}_{0}-\mathcal{X}_l\|_F}{\|\mathcal{X}_{0}\|_F}<10^{-3}$. Figure \ref{fig111} reports the fraction of perfect
recovery for each pair (black = 0\% and white = 100\%).

\begin{figure}[!h]
\centering
\subfigure[Alg. \ref{alg1}]{\includegraphics[width=3.0in, height=2.5in]{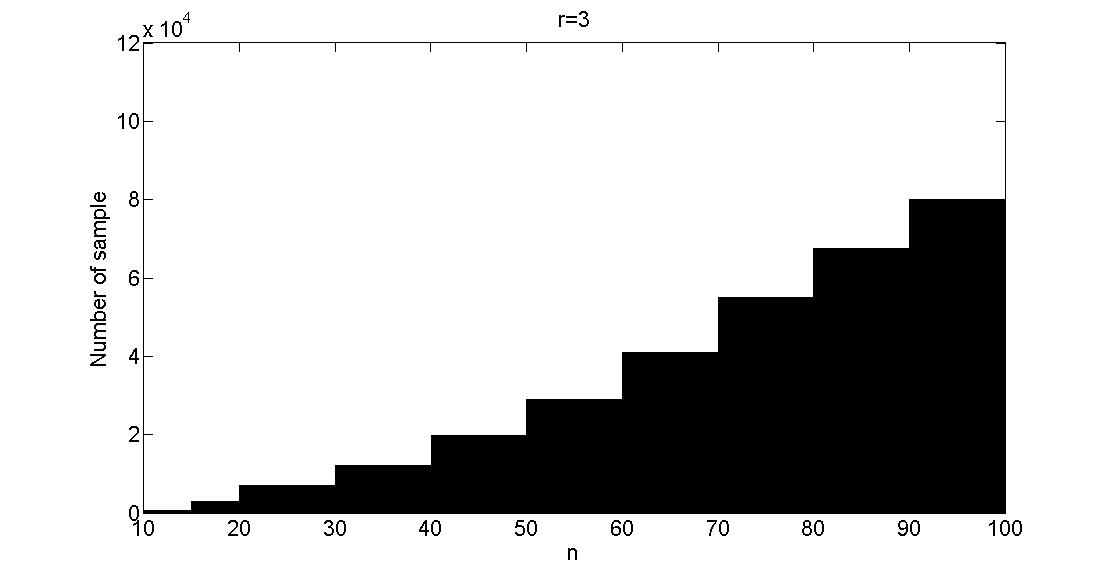}}
\subfigure[Alg. \ref{alg1}]{\includegraphics[width=3.0in, height=2.5in]{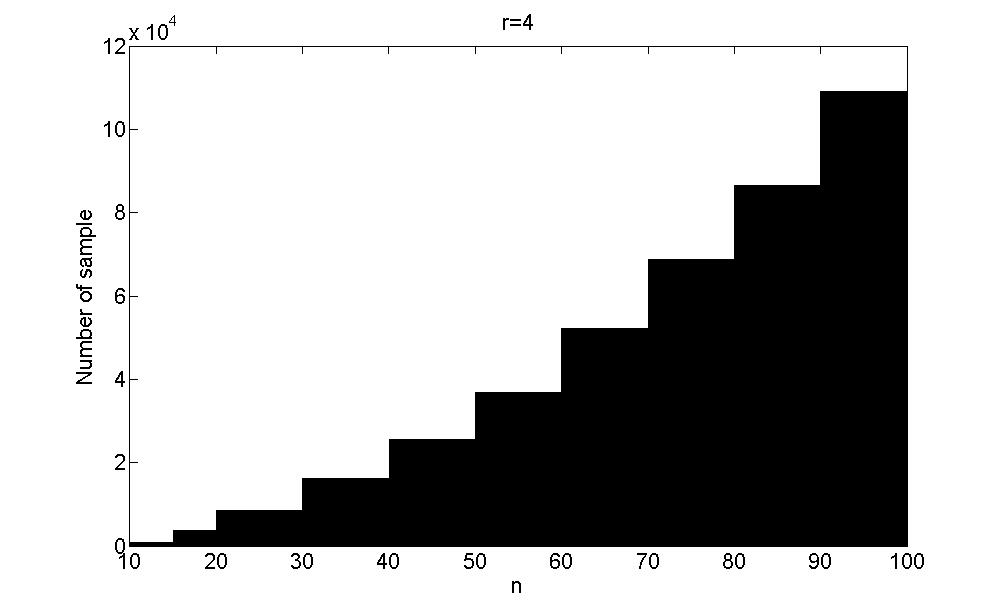}}\\
\subfigure[Alg. \ref{alg2}]{\includegraphics[width=3.0in, height=2.5in]{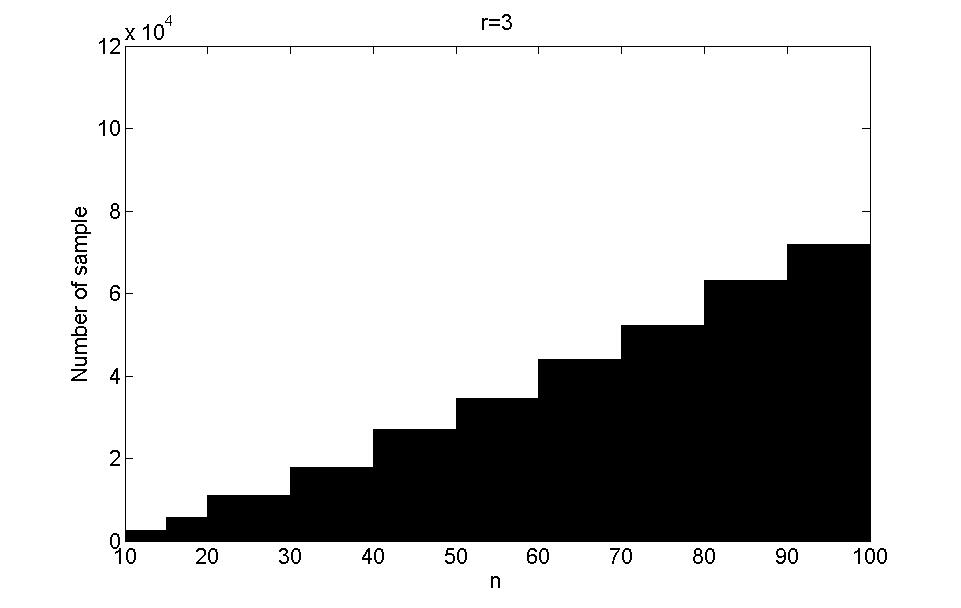}}
\subfigure[Alg. \ref{alg2}]{\includegraphics[width=3.0in, height=2.5in]{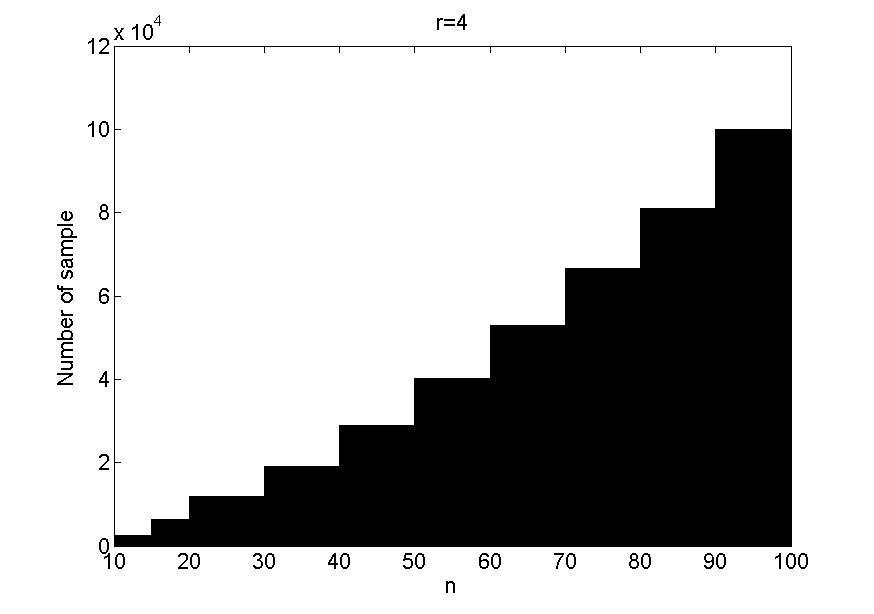}}
\caption{Recovery for different tensor sizes and sampling numbers.}\label{fig111}
\end{figure}

We see clearly that the recovery is correct under the sampling sizes given in Theorems \ref{them2} and \ref{them3} for all the cases. Moreover, compared with Algorithm 2, the sampling sizes needed of Algorithm 3 are improved after applying the trimming method.
In the following Tables and Figures, `Res', `Ite', `Time' and `Sr'  denotes the relative error, the iteration steps, the CPU time and the sampling ratio, respectively. To further corroborate our theoretical results, we check the effects of different values $r$ in the hard thresholding operator proposed in Algorithm 2 on `Res', `Ite' and `Time', respectively.  For a given tensor $\mathcal{X}$ with $n=50$ and $\rank_{ct}(\mathcal{X})=4$,
we test the effects under different sampling ratios, and show the results  in Figure \ref{fig1}.
It can be seen that when $\rank_{ct}(\mathcal{X})=4,$ the `Res', `Ite' and `Time' are better than the other settings under different sampling rates, respectively.
In addition, we compare our Algorithm 2 with one well known convex method, namely, t-svd method introduced in \cite{Zhang2017}.
Once again, we fix $n=50,$ and generate random tensor $\mathcal{X}$ as the prior experiments. For simplicity,
we test $\rank_{ct}(\mathcal{X})=2$ and $\rank_{ct}(\mathcal{X})=4$ under different sampling ratios. As shown in Table \ref{table1}, the `Res', `Ite' and `Time' of our Algorithm 2 are much better than the t-svd method.

\begin{figure}[h]
\centering
\subfigure[]{\includegraphics[width=2in, height=2.5in]{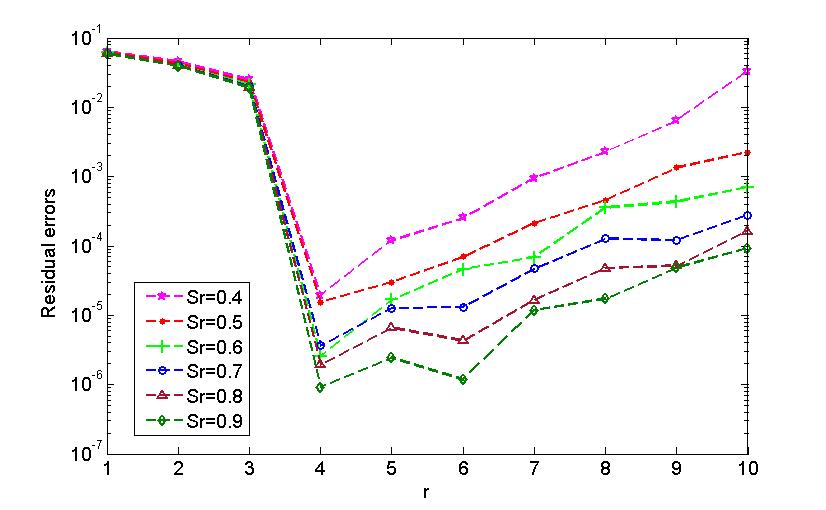}}
\subfigure[]{\includegraphics[width=2in, height=2.5in]{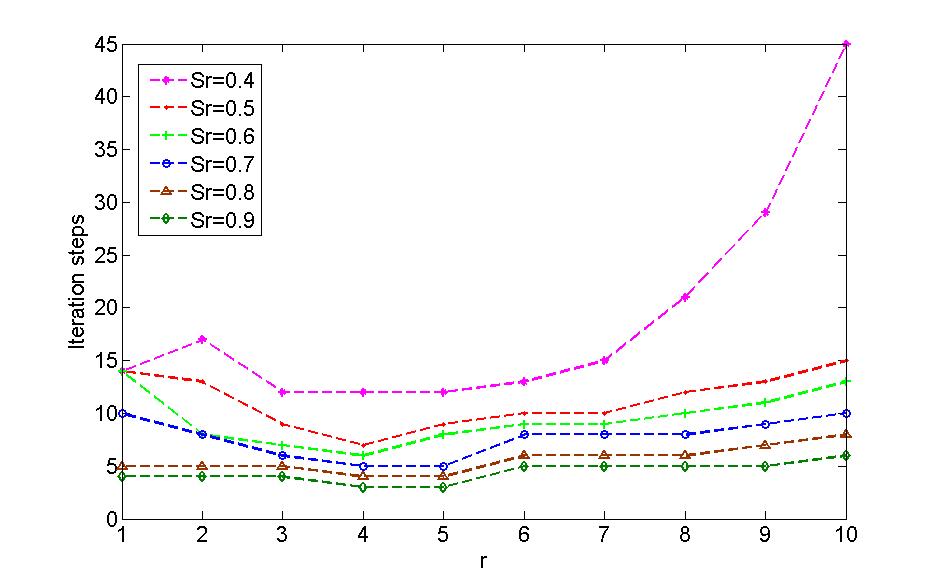}}
\subfigure[]{\includegraphics[width=2in, height=2.5in]{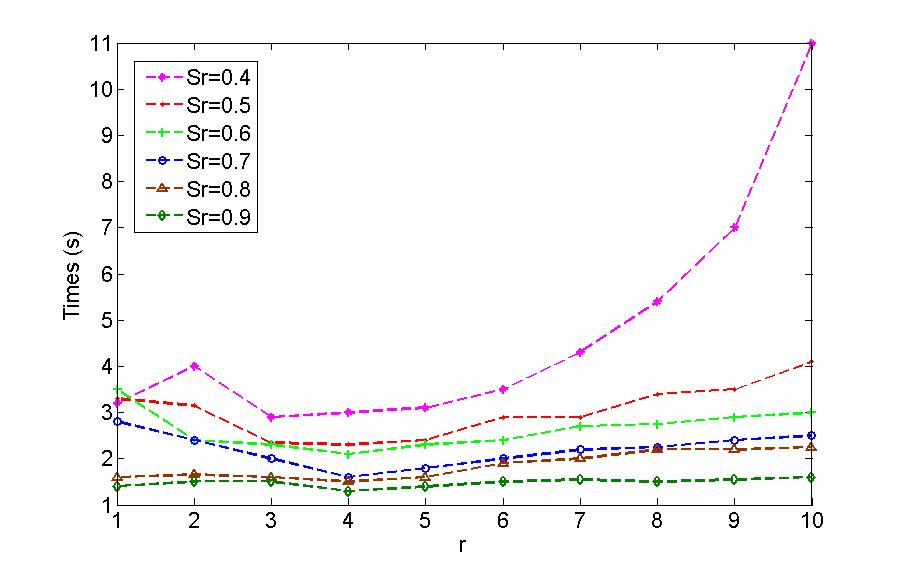}}\\
\caption{The comparison of the effects on `Res', `Ite' and `Time', relate to fixed $rank_{t}=4$.}\label{fig1}
\end{figure}

\begin{table}
  \caption{The comparison of Algorithm 2 and t-svd with different tubal ranks.} \label{table1}
  \begin{center}
{\begin{tabular}{cc|ccc|ccc}
  \hline
 && Alg. \ref{alg1}&&&&\;t-svd \cite{Zhang2017}\\\cline{3-8}
  $\rm{Tubal\;rank}$ &Sr&Time&Ite&Res&Time&Ite&Res\\\hline
  &0.4 &1.5701  &6& 9.5524e-6&7.1069&117&1.6783e-5 \\
  &0.5 &1.4471  &5& 8.4042e-6&4.8378&80&1.1604e-5\\
 2&0.6 &1.2834  &4& 4.7529e-7&3.5759&58&7.5460e-5 \\
  &0.7 & 1.3337 &4& 4.3528e-7&2.8086&43&6.1080e-5 \\
  &0.8 & 1.1767 &3&2.3870e-6 &2.2422&32&4.7257e-5 \\\hline
   &0.4 & 2.0595 & 8 & 3.4762e-5&13.8184&234&3.8036e-5 \\
   &0.5 & 1.6810 &6  &1.4283e-5 &8.5930&141 &2.3471e-5\\
 4 &0.6 & 1.5210 &5  & 2.6949e-6&5.7455&95 &4.3793e-5\\
   &0.7 &1.3695  & 4 &5.2723e-6 &4.0027&66 &4.5901-5\\
   &0.8 & 1.4065 & 4 &1.3991e-6 &2.8752&47 & 6.7191e-5\\\hline
\end{tabular}}
\end{center}
\end{table}

\subsection{Color Image Recovery}

\begin{figure}
\centering
\begin{tabular}{cccccc}
\includegraphics[width=0.8in]{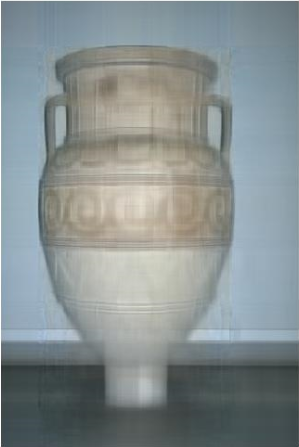}&\includegraphics[width=0.8in]{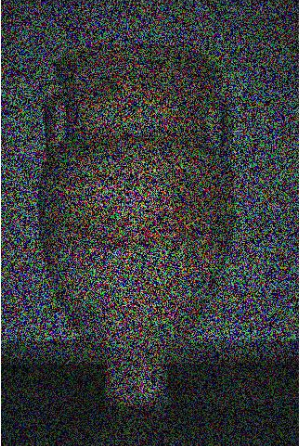} &
\includegraphics[width=0.8in]{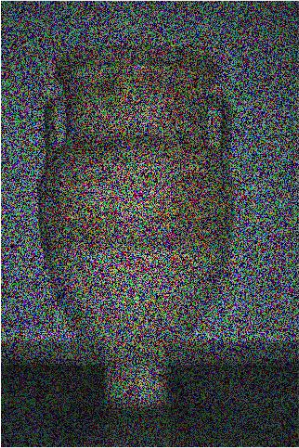} &
\includegraphics[width=0.8in]{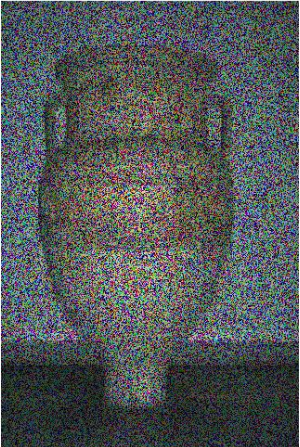}&
\includegraphics[width=0.8in]{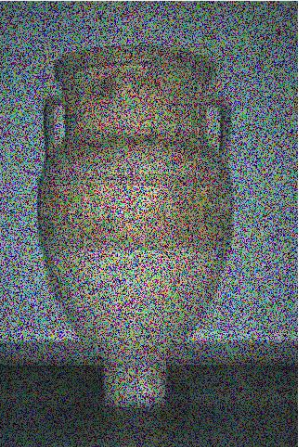}&
\includegraphics[width=0.8in]{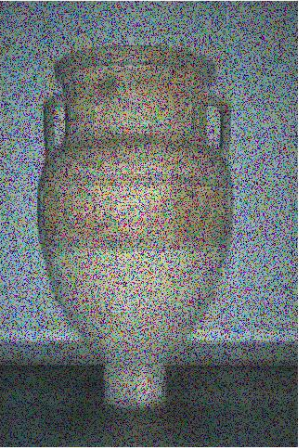}\\
&\includegraphics[width=0.8in]{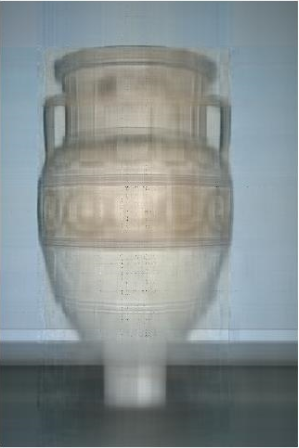} &
\includegraphics[width=0.8in]{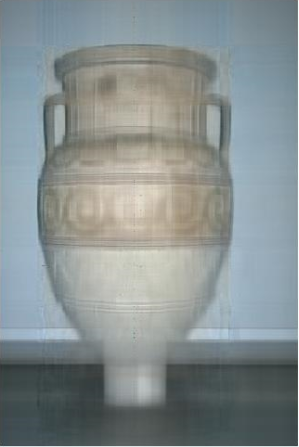}&
\includegraphics[width=0.8in]{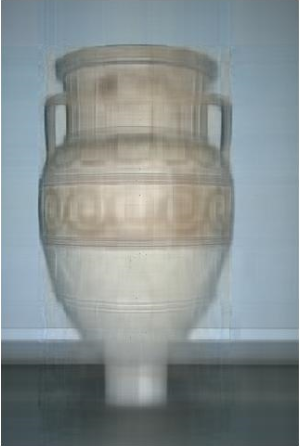}&
\includegraphics[width=0.8in]{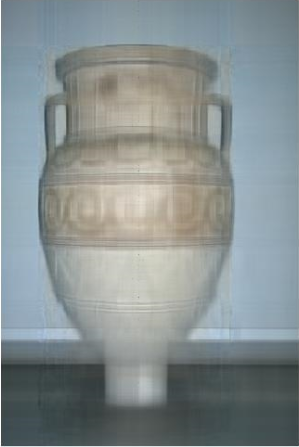}&
\includegraphics[width=0.8in]{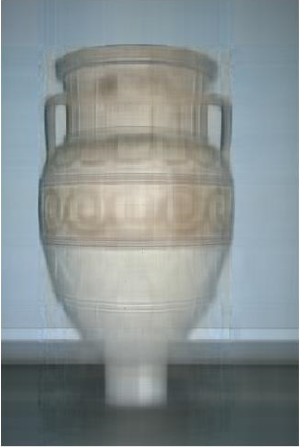}\\
&$Sr=0.4$ & $Sr=0.5$ & $Sr=0.6$ & $Sr=0.7$ & $Sr=0.8$
\end{tabular}
\caption{Recovery results under different sampling ratios by Algorithm 2 with $r=(29,5,1)$.}\label{fig3}
\end{figure}

It is well known that a $n_1\times n_2$ color image with red, blue and green channels can be naturally
regarded as a third-order tensor $\mathcal{A}\in \mathbb{R}^{n_1\times n_2\times 3}$. Each frontal slice of $\mathcal{A}$ corresponds to a channel
of the color image. Actually, each channel of a color image may not be low-rank, but their top
singular values dominate the main information \cite{liu2013,lu2016tensor}. Hence, the image can be reconstructed into a low-tubal-rank tensor by truncated t-svd.
We use the proposed Algorithm 2
to recover the sampled images under different sampling ratios and
test the restoration performances of the proposed Algorithm by computing the PSNR on the sampled images in
The original color image shown in the Figure \ref{fig3} is a $481\times 321\times 3$ tensor with transformed tubal multi-rank $\textbf{r}=(29,5,1)$.  The sampled images and recovered images  are listed in the first and second line of Figure \ref{fig3}, respectively.
Similar to synthetic data case, we compare our proposed  Algorithm 2
with the t-SVD method  in \cite{Zhang2017}.  Both `PSNR' and `Time' are listed for the comparison of different methods which are shown in Table \ref{table4}.  We find that the corresponding performance of PSNR  by the proposed Algorithm
1 is better than that by t-svd, and the running time by Algorithm 1
is less than that by t-svd when the sampling ratios is small ($sr\leq 0.6$).

\begin{table}[!h]
\centering
\caption{PSNR and Time by Algorithm 1 and t-svd \cite{Zhang2017} for a color image recovery with $\textbf{r}=(29,5,1)$.} \label{table4}
\begin{tabular}{c|c|ccccc}
  \hline
  &Sr  & 0.4 & 0.5 & 0.6 & 0.7 & 0.8 \\\hline
 Alg. 1 &PSNR & 41.42&	41.64&	41.69&	41.73&	41.76	\\
  &Time &12.62&   11.75&   10.95&   10.11&   9.51   \\\hline
t-SVD &PSNR &30.01&31.12&31.73&32.12&32.96\\
\cite{Zhang2017}&Time &13.64&13.20&11.28&9.86&8.88\\
  \hline
\end{tabular}
\end{table}

Moreover, we also test the effects of the values of $\textbf{r}$ in the hard thresholding operator
for tensor completion by the proposed
Algorithm 1. For the original color image with transformed tubal multi-rank $\textbf{r}=(29,5,1)$,
we test the effects on `PSNR' and `Time' of the values
$\textbf{r}+h\textbf{b}$ with $h=-2,-1,\ldots,5,~\text{and},~\textbf{b}=[1,1,1]^{T}$, by applying Algorithm 1
under different sampling ratios, respectively. The testing results are shown in Figure \ref{fig5}. It can be seen that, when $h=0$, the `PSNR' and `Time' are better than the other settings.

\begin{figure}[!h]
\centering
\subfigure[]{\includegraphics[width=2.5in, height=2in]{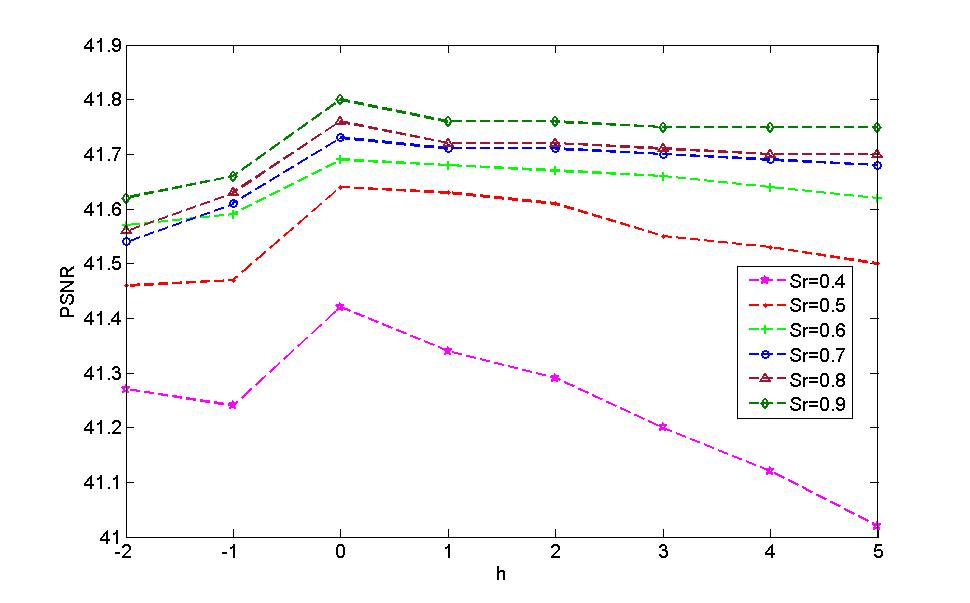}}
\subfigure[]{\includegraphics[width=2.5in, height=2in]{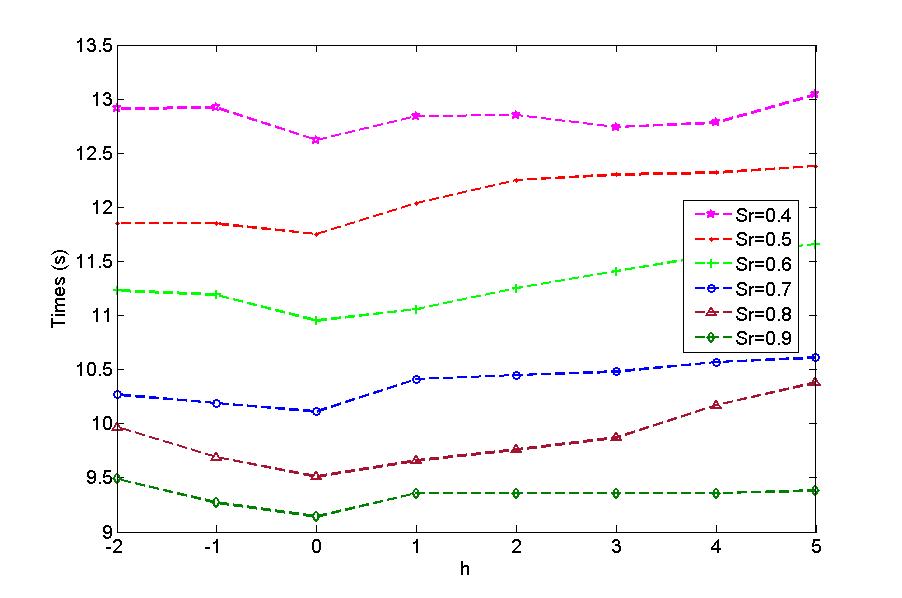}}\\
\caption{The effect of different setting of $\textbf{r}$ by  Algorithm 1
for given tensor with $\textbf{r}=(29,5,1)$.} \label{fig5}
\end{figure}

\subsection{Video Data}

We consider
the given video data Carphone $(144\times 176\times 180)$\footnote{\footnotesize{https://media.xiph.org/video/derf/}}
to test Algorithm 2.
We display the visual comparisons of the testing data in tensor completion with sampling ratios $0.4, 0.5,0.6,0.7,0.8,$ by Algorithm 1
and t-svd in Figure \ref{fig6}.  We can see that the images recovered by Algorithm 2
are much better than t-svd when the sampling ratios are small. The Algorithm 2
can keep more details than t-svd for the testing video. We also show the `PSNR' and `Time' by Algorithm 1 and t-svd for Carphone data with sampling ratios $0.4, 0.5,0.6,0.7,0.8$ in Table \ref{table5}. It can be seen that the PSNR values obtained by Algorithm 2
are higher than those by t-svd, and the CPU time of Algorithm 1 are less than those by t-svd.

\begin{table}[!h]
\centering
\caption{PSNR and Time by Algorithm 2 and t-svd  \cite{Zhang2017} for video data recovery.} \label{table5}
\begin{tabular}{c|c|ccccc}
  \hline
  &Sr  & 0.4 & 0.5 & 0.6 & 0.7 & 0.8 \\\hline
  Alg. \ref{alg1} &PSNR & 48.65&	48.94&	49.11&	49.23&	49.46\\
  &Time &126.88&   105.11&   91.84&   76.88&   70.65\\\hline
t-SVD &PSNR &46.89&47.82&48.24&48.64&48.69\\
\cite{Zhang2017}&Time &145.21&139.50&131.58&127.90&121.50\\
  \hline
\end{tabular}
\end{table}

\begin{figure}[!h]
\centering
\begin{tabular}{cccccc}
\includegraphics[width=0.8in]{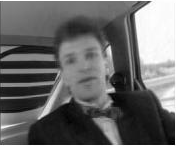}&
\includegraphics[width=0.8in]{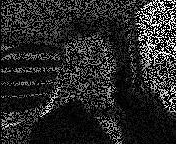} &
\includegraphics[width=0.8in]{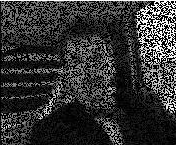} &
\includegraphics[width=0.8in]{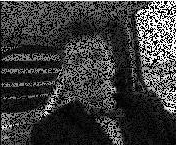}&
\includegraphics[width=0.8in]{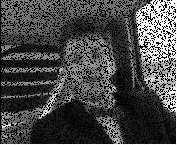}&
\includegraphics[width=0.8in]{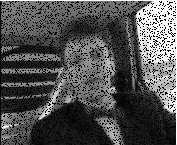}\\

&\includegraphics[width=0.8in]{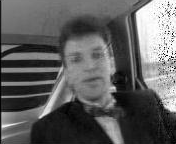} &
\includegraphics[width=0.8in]{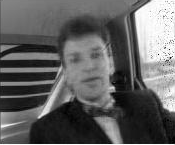}&
\includegraphics[width=0.8in]{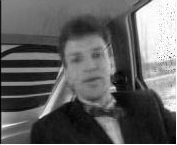}&
\includegraphics[width=0.8in]{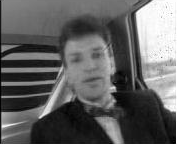}&
\includegraphics[width=0.8in]{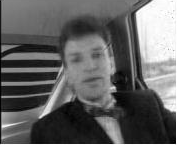}\\
&\includegraphics[width=0.8in]{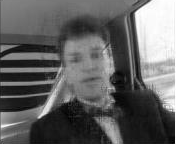} &
\includegraphics[width=0.8in]{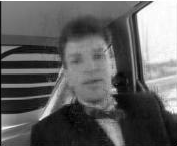}&
\includegraphics[width=0.8in]{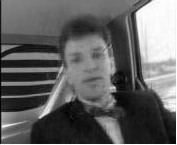}&
\includegraphics[width=0.8in]{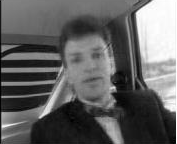}&
\includegraphics[width=0.8in]{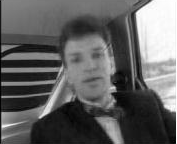}\\
&$Sr=0.4$ & $Sr=0.5$ & $Sr=0.6$ & $Sr=0.7$ & $Sr=0.8$
\end{tabular}
\caption{Recovered images (the tenth frame) by t-svd \cite{Zhang2017} and Algorithm 2
in tensor completion for video data with different sampling ratios.
The observed images, recovered images by t-svd \cite{Zhang2017} and Algorithm 2
are respectively listed from the first row, the second row, and the third row respectively.}\label{fig6}
\end{figure}

\section{Conclusion}

In this paper, we mainly consider the low rank tensor
completion problem by Riemannian optimization methods on the
manifold, where the tensor rank is relate to the transformed  tubal multi-rank which is
based on tensor singular value decomposition.   We discuss the
convergence properties of Riemannian conjugate gradient methods
under different initialized methods, respectively. The minimum
sampling sizes needed to recover the low transformed multi-rank tensor are also
derived. Numerical simulation shows that the algorithms are able to
recover the low transformed multi-rank tensor under different initialized
methods.

\section{Appendix} \label{seca}

\subsection{Proof of  Proposition \ref{prop1}, \ref{prop2} and \ref{prop3}}

In order to prove Proposition \ref{prop1}, we need  the following results which given in \cite{lee2013smooth}.
\begin{lemma}[Lemma 5.5 in \cite{lee2013smooth}]
Let $\mathcal{M}$ be a smooth manifold and $\mathcal{N}$ is a subset of $\mathcal{M}$.
Suppose every point $p\in \mathcal{N}$ has a neighborhood $\mathcal{U}\subset \mathcal{M}$  such that $\mathcal{U}\cap \mathcal{N} $is an embedded submanifold of $\mathcal{U}$. Then $\mathcal{N}$ is an embedded submanifold of $\mathcal{M}$.
\end{lemma}
\begin{proof}[{\bf Proof of Proposition \ref{prop1}}]
Suppose that $\mathcal{E}_{0}\in \mathbb{R}^{n_{1}\times n_{2}\times n_{3}}$ with $\rank_{ct}(\mathcal{E}_{0})=r$ and $\rank_{t}(\mathcal{E})=\textbf{r}=(r_{1},...,r_{n_{3}}).$  By the element transformations it can be expressed as
\begin{align*}
\mathcal{E}_{0}=\left(
    \begin{array}{cc}
      \mathcal{ A}_{0} & \mathcal{ B}_{0} \\
      \mathcal{ C}_{0} & \mathcal{ D}_{0} \\
    \end{array}
  \right)
\end{align*}
where $ \mathcal{A}\in \mathbb{R}^{r\times r \times n_{3}}$ with $\rank_{t}(\mathcal{A})=\textbf{r},$ $ \mathcal{B}\in \mathbb{R}^{r\times (n_2-r) \times n_{3}}$, $ \mathcal{C}\in \mathbb{R}^{(n_1-r)\times r \times n_{3}}$, and $ \mathcal{D}\in \mathbb{R}^{(n_1-r)\times (n_2-r) \times n_{3}}$. Set $\phi : \mathbb{R}^{n_{1}\times n_{2}\times n_{3}}\rightarrow \mathbb{R}^{n_{1}n_{3}\times n_{2}n_{3}}$ as:
\begin{equation*}
\phi(\mathcal{A})  = \bdiag(\widehat{\mathcal{A}}) = \left[
\begin{array}{llll}
\widehat{\mathcal{A}}^{(1)} & & & \\
 & \widehat{\mathcal{A}}^{(2)} & & \\
 & & \ddots &\\
 & & & \widehat{\mathcal{A}}^{(n_3)}
\end{array}
\right],
\end{equation*} with  $\widehat{\mathcal{A}} = \dct(\mathcal{A}, [], 3).$
It is easy to see that $\phi$ is a bijective.
Let $\mathcal{U}$ be the open set
\begin{align*}
\mathcal{U}=\left\{\left(
    \begin{array}{cc}
      \mathcal{ A} & \mathcal{ B} \\
      \mathcal{ C} & \mathcal{ D} \\
    \end{array}
  \right)\in \mathbb{R}^{n_{1}\times n_{2}\times n_{3}}:~\rank_{t}(\mathcal{A})=\textbf{r}
\right\}
\end{align*}
which contains $\mathcal{E}_{0}$.  Define
$F: \mathcal{U}\rightarrow \mathbb{R}^{\sum_{i=1}^{n_3}(n_{1}-r_{i})\times \sum_{i=1}^{n_3}(n_{2}-r_{i})}$ as
\begin{align*}
 F\left(
    \begin{array}{cc}
      \mathcal{ A} & \mathcal{ B} \\
      \mathcal{ C} & \mathcal{ D} \\
    \end{array}
  \right)=(\widehat{\mathcal{D}}^{(1)}-\widehat{\mathcal{B}}^{(1)}(\widehat{\mathcal{A}}_{r_1}^{(1)})^{-1}\widehat{\mathcal{C}}^{(1)}) \oplus \cdots \oplus (\widehat{\mathcal{D}}^{(n_{3})}-\widehat{\mathcal{B}}^{(n_{3})}(\widehat{\mathcal{A}}_{r_{n_3}}^{(n_{3})})^{-1}\widehat{\mathcal{C}}^{(n_{3})}).
\end{align*}
Clearly, $F$ is smooth. In order to show it is a submersion, we need to show $F'(p)$ is surjective for each $p \in \mathcal{U}.$  Note that $\mathbb{R}^{\sum_{i=1}^{n_3}(n_{1}-r_{i})\times \sum_{i=1}^{n_3}(n_{2}-r_{i})}$ is a vector space, the tangent vector at $F(p)$ can be defined as $\sum_{i=1}^{n_3}(n_{1}-r_{i})\times \sum_{i=1}^{n_3}(n_{2}-r_{i})$ matrices. Given
$$\mathcal{E}=\idct\left(\mbox{\tt fold} \left(
                          \begin{array}{cccccc}
                            \hat{\mathcal{A}}_{r_{1}}^{(1)} & \hat{\mathcal{B}}^{(1)}&  &  &  &  \\
                              \hat{\mathcal{C}}^{(1)}& \hat{\mathcal{D}}^{(1)}  &  &  &  &  \\
                              &   & \ddots  &  &   &   \\
                            &   &  &   &  \hat{\mathcal{A}}_{r_{n_{3}}}^{(n_{3})} & \hat{\mathcal{B}}^{(n_{3})}  \\
                             &   &  &   &   \hat{\mathcal{C}}^{(n_{3})}&\hat{\mathcal{D}}^{(n_{3})}   \\
                          \end{array}
                        \right)\right)\in \ten{U},$$
and any tensor
$$\mathcal{X}=\idct\left(\mbox{\tt fold}\left(
                               \begin{array}{cccc}
                                 \hat{\mathcal{X}}^{(1)} &   &  &   \\
                                   &  \hat{\mathcal{X}}^{(2)} &   &   \\
                                  &   & \ddots &   \\
                                   &  &   & \hat{\mathcal{X}}^{(n_3)} \\
                               \end{array}
                             \right)\right)\in \mathbb{R}^{n_1 \times n_2 \times n_3}
$$
with $$\hat{\mathcal{X}}^{(i)}=\left(
                              \begin{array}{cc}
                                \textbf{0}_{r_{i}\times r_{i}} &\textbf{0}_{r_{i}\times (n_{2}-r_{i})} \\
                                \textbf{0}_{(n_1-r_{i})\times r_{i}} & X^{(i)}_{(n_1-r_{i})\times (n_{2}-r_{i}) } \\
                              \end{array}
                            \right),~i=1,...,n_3.
$$
 Define a curve $\tau :(-\xi,\xi)\rightarrow \mathcal{U}$ by
\begin{align*}
\tau(t)=\idct\left(\mbox{\tt fold} \left(
                          \begin{array}{cccccc}
                            \hat{\mathcal{A}}_{r_{1}}^{(1)} & \hat{\mathcal{B}}^{(1)}&  &  &  &  \\
                              \hat{\mathcal{C}}^{(1)}& \hat{\mathcal{Y}}^{(1)}  &  &  &  &  \\
                              &   & \ddots  &  &   &   \\
                            &   &  &   &  \hat{\mathcal{A}}_{r_{n_{3}}}^{(n_{3})} & \hat{\mathcal{B}}^{(n_{3})}  \\
                             &   &  &   &   \hat{\mathcal{C}}^{(n_{3})}&\hat{\mathcal{Y}}^{(n_{3})}   \\
                          \end{array}
                        \right)\right),
\end{align*}
where $\hat{\mathcal{Y}}^{(i)}=\hat{\mathcal{D}_i}+tX^{(i)}_{(n_1-r_{i})\times (n_{2}-r_{i}) } -\hat{\mathcal{C}}^{(i)}(\hat{\mathcal{A}}_{r_{i}}^{(i)})^{-1}\hat{\mathcal{B}}^{(i)},~i=1,...,n_3.$ We remark that $\hat{\mathcal{Y}}^{(i)},~i=1,...,n_3,$ does not need to possess the same sizes, respectively.
Then we can get
\begin{align*}
F_{*}\tau^{'}(0)=(F\circ\tau )^{'}(t)=&\left(
                                        \begin{array}{ccc}
                                          \frac{d}{dt}|_{t=0}(\hat{\mathcal{Y}}^{(1)}) &  &  \\
                                          & \ddots &  \\
                                           &  &  \frac{d}{dt}|_{t=0}(\hat{\mathcal{Y}}^{(n_3)}) \\
                                        \end{array}
                                      \right)\\
=&\left(
                                                 \begin{array}{ccc}
                                                   X^{(1)}_{(n_1-r_{1})\times (n_{2}-r_{1}) } &   &     \\
                                                        &  \ddots &   \\
                                                       &   &  X^{(n_3)}_{(n_1-r_{n_3})\times (n_{2}-r_{n_3}) } \\
                                                 \end{array}
                                               \right),
\end{align*} where $F_{*}$ is the push-forward projection relate $F.$
Then $F$ is a submersion and so  $\mathcal{M}_{\textbf{r}}\cap \mathcal{U}$  is an embedded submanifold of $\mathcal{U}$.
Next if $\mathcal{E}'$ is an arbitrary tensor of $\rank_{t}(\mathcal{E}')=\textbf{r}.$ Just note that it can be transformed to one in $\mathcal{U}$ by a rearrangement along the first and second directions. Such a rearrangement $R$ is a linear isomorphism that preserves the tensor transformed multi-rank, so $\mathcal{U}_0 = R^{-1}(\mathcal{U})$ is a neighborhood of $\mathcal{E}'$ and $F\circ R: \mathcal{U}_0\rightarrow \mathbb{R}^{n_{1}n_{3}\times n_{2}n_{3}}$ is a submersion whose zero level set is $\mathcal{M}_{\textbf{r}} \cap \mathcal{U}$. Thus every point in $\mathcal{M}_{\textbf{r}}$ has a neighborhood $\mathcal{U}$ such that $\mathcal{M}_{\textbf{r}}\cap \mathcal{U}_0$ is an embedded submanifold of $\mathcal{U}_0$, so $\mathcal{M}_{\textbf{r}}$ is an embedded submanifold. Moreover, note that $\rank (F_{*}\tau^{'}(0))=\sum_{i=1}^{n_{3}}((n_{1}+n_{2})r_{i}-r_{i}^2)$ which is saying that $\mathcal{M}_{\textbf{r}}$ possess dimension $\sum_{i=1}^{n_{3}}((n_{1}+n_{2})r_{i}-r_{i}^2).$
\end{proof}

\begin{proof}[{\bf Proof of Proposition \ref{prop2}}]
Suppose that $\ten{X}_l \in \ten{M}_{\textbf{r}},$ it follows from Definition \ref{def1.4} that $\ten{X}_{l}$ can be expressed as
$$\ten{X}_{l}=\ten{U}_l\ast \ten{S}\ast \ten{V}_l^{T},$$
where $\rank_{t}(\mathcal{U}_{l})=\rank_{t}(\mathcal{V}_{l})=\rank_{t}(\mathcal{S})=\textbf{r},~\mathcal{U}^T \ast \mathcal{U}=\mathcal{I}_{\textbf{r}}$,
 $\mathcal{V}_{l}^{T}\ast \mathcal{V}_{l}=\mathcal{I}_{\textbf{r}},$ and $\mathcal{S}$ is a diagonal tensor.
Then define a curve $\gamma: (-\zeta,\zeta)\rightarrow \mathcal{M}_{\textbf{r}}$ by setting $\gamma(t)=(\ten{U}_l+t\ten{X})\ast \ten{S}\ast ((\ten{V}_l+t\ten{Y}))^T$ where $\ten{X}$ and $\ten{Y}$ are arbitrary tensors with proper sizes. It is easy to check that $\gamma$ is smooth and $\gamma(0)=\ten{X}_l.$ Then, $\gamma'(0)=\ten{X}\ast\ten{S}\ast\ten{V}_{1}^T+\ten{U}_{l}\ast\ten{S}\ast\ten{Y}^T,$
and Proposition \ref{prop2} can be derived by setting $\ten{S}\ast\ten{Y}^T$ as $\ten{Z}^T_{1}$ and $\ten{X}\ast\ten{S}$ as $\ten{Z}_{2}.$

\end{proof}

A retraction is any smooth map from the tangent bundle  $T\ten{M}$ into $\ten{M}$ that approximates the exponential map to the first order. In order to prove the t$_c$-SVD truncation is a retraction we need introduce the following definition.
\begin{definition}[Definition 1 in \cite{absil2012}]\label{add1}
Let $\ten{M}$ be a smooth submanifold of $\mathbb{R}^{n_1\times \cdots\times n_{d}},$  $0_{x}$ denote the zero element of $T_{x}\ten{M}$. A mapping $R$ from the tangent bundle $T\ten{M}$ into $\ten{M}$ is said to
be a retraction on $\ten{M}$ around $x\in \ten{M}$ if there exists a neighborhood $U$ of $(x,0_{x})$ in $T\ten{M}$ such that the following properties hold:
\item (a) We have $U\subseteq dom(R)$ and the restriction $R: U \rightarrow \ten{M}$ is smooth.
\item (b) $R(y,0_y)$ for all $(y,0_y)\in U.$
\item (c) With the canonical identification $T_{0_x}T_x\ten{M}\simeq T_x\ten{M}$, $R$ satisfies the local rigidity condition:
$$DR(x,\cdot)(0_x)=id_{T_x\ten{M}}~\text{for all,}~(x,0_x)\in U, $$ where $id_{T_x\ten{M}}$ denotes the identity mapping on $T_{x}\ten{M}.$
\end{definition}
\begin{proof}[{\bf Proof of Proposition \ref{prop3}}]
 Let $H_{r_{i}}$ denote the SVD truncation of the $i$-th front slice of a tensor in the DCT transform domain. Then $P_{\ten{M}_{\textbf{r}}}(\ten{X}+\mathcal{\xi})$ given in \eqref{truncation} can be expressed as $\idct\circ(H_{r_{1}}\circ\cdots\circ H_{r_{n_3}})\circ \dct (\ten{X}+\mathcal{\xi}),$ where $H_{r_{i}}$ are independent with each other.
 Suppose that $\ten{N}_{i}$ denotes the open set of tensors that the $i$-th front slice if $\dct(\ten{N}_{i})$ has a nonzero gap between the $r_{i}$-th  and the $r_{i+1}$-th singular values. Then we can get the truncation operator $H_{r_{i}}$ is smooth and well-defined on $\ten{N}_{i}$.  Note that $\ten{X}\in \ten{M}_{\textbf{r}}$ means that $\ten{X}$ contains in any open set $\ten{N}_{i}$ and $\hat{\ten{X}}$ is a fixed point of every $H_{r_{i}}.$ Therefore,  it is possible to construct an open neighborhood $\ten{N}\in \mathbb{R}^{n_1\times n_2 \times n_2}$ of $\ten{X}$ such that $\idct\circ(H_{r_{1}}\circ\cdots\circ H_{r_{n_3}})\circ \dct (\ten{N}) \subseteq \ten{N}_{i}$ for all $i=1,...,n_3.$ Hence, the smoothness of $P_{\ten{M}_{\textbf{r}}}(\ten{X}+\mathcal{\xi})$ can be derived by the chain rule.  Hence, $P_{\ten{M}_{\textbf{r}}}$ in \eqref{truncation} defines a  locally smooth map in a neighborhood $\ten{N} \in T\ten{M}_{\textbf{r}}$ around $(\ten{X},0_{\ten{X}})$, i.e.,  Definition \ref{add1} (a) is satisfied. Definition \ref{add1} (b) can be derived by the fact that the application of the t$_{c}$-SVD truncation to elements in $\mathcal{M}_{\textbf{r}}$ leaves them unchanged. For Definition \ref{add1} (c).  Note that the tangent space $T_{\ten{X}}\ten{M}_{\textbf{r}}$ is a first order approximation of $\ten{M}_{\textbf{r}}$ around $\ten{X}$, we have $\|(\ten{X}+t\xi)-P_{\ten{M}_{\textbf{r}}}(\ten{X}+t\mathcal{\xi})\|_{F}=O(t)^{2}$ for $t\rightarrow 0.$ Hence, $P_{\ten{M}_{\textbf{r}}}(\ten{X}+\mathcal{\xi})=(\ten{X}+t\xi)+O(s)^{2}$, which gives $\frac{d}{dt}P_{\ten{M}_{\textbf{r}}}(\ten{X}+t\mathcal{\xi})|_{t=0}=\xi$. In other words,  $DP_{\ten{M}_{\textbf{r}}}(\ten{X},\cdot)(0_{\ten{X}})=id_{T_{\ten{X}}\ten{M}_{\textbf{r}}},$ which completes the proof.
 \end{proof}

\subsection{Proof of Lemma~\ref{lem5}}
\begin{proof}[Proof of Lemma~\ref{lem5}]
The proof of (i) can be proceeded as follows:
\begin{align*}\|\mathcal{U}_{l}\ast \mathcal{U}_{l}^{T}- \mathcal{U}\ast
\mathcal{U}^{T}\|&=\|\overline{\mathcal{U}_{l}\ast \mathcal{U}_{l}^{T}- \mathcal{U}\ast
\mathcal{U}^{T}}\|\\
&=\|\overline{\mathcal{U}}_l\overline{\mathcal{U}}_l^T-\overline{\mathcal{U}}\hspace{.1cm}\overline{\mathcal{U}}^T\|\\
&=\max_{i=1,\cdots,n_3}\|\widehat{\ten{U}}^{(i)}_l(\widehat{\ten{U}}^{(i)}_l)^T-\tenhat{U}^{(i)}(\tenhat{U}^{(i)})^T\|\\
&\leq \max_{i=1,\cdots,n_3} \frac{\|\tenhat{X}_l^{(i)}-\tenhat{X}^{(i)}\|_F}{\sigma_{\min}(\tenhat{X}^{(i)})}\\
&\leq \frac{\max_{i=1,\cdots,n_3}\|\tenhat{X}_l^{(i)}-\tenhat{X}^{(i)}\|_F}{\min_{i=1,\cdots,n_3}\sigma_{\min}(\tenhat{X}^{(i)})}\\
&\leq\frac{\|\ten{X}_l-\ten{X}\|_F}{\sigma_{\min}(\ten{X})},
\end{align*}
where the fourth line follows from Lemma 4.1 in \cite{wei2}.  We can similarly prove (ii).

Noting that $\dct(\cdot)$ is a unitary transform, we have
\begin{align*}
\|\mathcal{U}_{l}\ast \mathcal{U}_{l}^{T}- \mathcal{U}\ast
\mathcal{U}^{T}\|_{F}^2& =\|\overline{\mathcal{U}_{l}\ast \mathcal{U}_{l}^{T}- \mathcal{U}\ast
\mathcal{U}^{T}}\|_F^2\\
&=\|\overline{\mathcal{U}}_l\overline{\mathcal{U}}_l^T-\overline{\mathcal{U}}\hspace{.1cm}\overline{\mathcal{U}}^T\|_F^2\\
&=\sum_{i=1}^{n_3}\|\widehat{\ten{U}}^{(i)}_l(\widehat{\ten{U}}^{(i)}_l)^T-\tenhat{U}^{(i)}(\tenhat{U}^{(i)})^T\|_F^2\\
&\leq \sum_{i=1}^{n_3} \frac{2\|\tenhat{X}_l^{(i)}-\tenhat{X}^{(i)}\|_F^2}{\sigma^2_{\min}(\tenhat{X}^{(i)})}\\
&\leq \frac{2}{\sigma^2_{\min}(\ten{X})}\sum_{i=1}^{n_3}\|\tenhat{X}_l^{(i)}-\tenhat{X}^{(i)}\|_F^2\\
&=\frac{2\|\ten{X}_l-\ten{X}\|_F^2}{\sigma^2_{\min}(\ten{X})},
\end{align*}
where the fourth line follows from Lemma 4.1 in \cite{wei2}. Taking square roots on both sides yields (iii) and  (iv) can be proved similarly.

Since $\Pj_\ten{T}(\ten{X})=\ten{X}$, $(\I-\Pj_{\ten{T}_l})\ten{X}=\Pj_{\ten{T}}(\mathcal{X})-\Pj_{\ten{T}_{l}}(\mathcal{X})$. Thus,
\begin{align*}
\Pj_{\ten{T}}(\mathcal{X})-\Pj_{T_{l}}(\mathcal{X})=&~\mathcal{U}\ast
\mathcal{U}^{T}\ast \mathcal{X} +\mathcal{X}\ast \mathcal{V}\ast
\mathcal{V}^{T} - \mathcal{U}\ast \mathcal{U}^{T}\ast
\mathcal{X}\ast \mathcal{V}\ast \mathcal{V}^{T}\\
&-\mathcal{U}_{l}\ast \mathcal{U}_{l}^{T}\ast \mathcal{X}
-\mathcal{X}\ast \mathcal{V}_{l}\ast \mathcal{V}_{l}^{T} +
\mathcal{U}_{l}\ast \mathcal{U}_{l}^{T}\ast \mathcal{X}\ast
\mathcal{V}\ast\mathcal{V}^{T}\\
=&~(\mathcal{U}\ast \mathcal{U}^{T} -\mathcal{U}_{l}\ast
\mathcal{U}_{l}^{T})\ast\mathcal{X}+\mathcal{X}\ast(\mathcal{V}\ast
\mathcal{V}^{T}-\mathcal{V}_{l}\ast \mathcal{V}_{l}^{T})\\
&~- \mathcal{U}\ast\mathcal{U}^{T}\ast \mathcal{X}\ast
\mathcal{V}\ast \mathcal{V}^{T}+ \mathcal{U}\ast \mathcal{U}^{T}\ast
\mathcal{X}\ast \mathcal{V}_{l}\ast\mathcal{V}_{l}^{T}\\
&~- \mathcal{U}\ast \mathcal{U}^{T}\ast \mathcal{X}\ast
\mathcal{V}_{l}\ast \mathcal{V}_{l}^{T}+ \mathcal{U}_{l}\ast
\mathcal{U}_{l}^{T}\ast \mathcal{X}\ast \mathcal{V}_{l}\ast
\mathcal{V}_{l}^{T}\\
=&~(\mathcal{U}\ast \mathcal{U}^{T}-\mathcal{U}_{l}\ast
\mathcal{U}_{l}^{T})\ast(\mathcal{X}-\mathcal{X}_{l})\ast(\mathcal{I}_{\bm{r}}-\mathcal{V}_{l}\ast
\mathcal{V}_{l}^{T}).
\end{align*}
Therefore, (v) can be established as follows:
\begin{align*}
\|(\I-\Pj_{\ten{T}_l})\ten{X}\|_{F} =&\|(\mathcal{U}\ast
\mathcal{U}^{T}-\mathcal{U}_{l}\ast
\mathcal{U}_{l}^{T})\ast(\mathcal{X}-\mathcal{X}_{l})\ast(\mathcal{I}_{\bm{r}}-\mathcal{V}_{l}\ast
\mathcal{V}_{l}^{T})\|_{F}\\
\leq&\|(\mathcal{U}\ast \mathcal{U}^{T}-\mathcal{U}_{l}\ast
\mathcal{U}_{l}^{T})\|\|\mathcal{X}-\mathcal{X}_{l}\|_{F}\|\mathcal{I}_{\bm{r}}-\mathcal{V}_{l}\ast
\mathcal{V}_{l}^{T}\|\\
\leq&\frac{\|\mathcal{X}_{l}-\mathcal{X}\|\|\mathcal{X}-\mathcal{X}_{l}\|_{F}}{\sigma_{\min}(\mathcal{X})}\\
\leq&\frac{\|\mathcal{X}-\mathcal{X}_{l}\|_{F}^2}{\sigma_{\min}(\mathcal{X})},
\end{align*}
where the third line follows from (i).

For any $\ten{Z}$, we have
\begin{align*}
(\Pj_{\ten{T}_{l}}-\Pj_{\ten{T}})(\mathcal{Z})=~&\mathcal{U}_{l}\ast \mathcal{U}_{l}^{T}\ast \mathcal{Z}
+\mathcal{Z}\ast \mathcal{V}_{l}\ast \mathcal{V}_{l}^{T} -
\mathcal{U}_{l}\ast \mathcal{U}_{l}^{T}\ast \mathcal{Z}\ast
\mathcal{V}\ast\mathcal{V}^{T} -\mathcal{U}\ast \mathcal{U}^{T}\ast
\mathcal{Z} -\mathcal{Z}\ast \mathcal{V}\ast \mathcal{V}^{T} \\
&+ \mathcal{U}\ast \mathcal{U}^{T}\ast
\mathcal{Z}\ast \mathcal{V}\ast \mathcal{V}^{T}\\
=~&(\mathcal{U}_{l}\ast \mathcal{U}_{l}^{T}-\mathcal{U}\ast
\mathcal{U}^{T})\ast\ten{Z}\ast(\mathcal{I}_{\bm{r}}-\mathcal{V}\ast
\mathcal{V}^{T})+(\mathcal{I}_{\bm{r}}-\mathcal{U}_{l}\ast
\mathcal{U}_{l}^{T})\ast\mathcal{Z}\ast(\mathcal{V}_{l}\ast
\mathcal{V}_{l}^{T}-\mathcal{V}\ast \mathcal{V}^{T}).
\end{align*}
Consequently, taking the Frobenius norm on both sides of the above equality and utilizing (i) and (ii) yields (vi).
\end{proof}

\subsection{Proof of Lemma \ref{lem1}}
\begin{proof}[Proof of Lemma \ref{lem1}]
For any tensor $\mathcal{Z}$, we have
\begin{align*}
\|\RR_{\Omega}\Pj_{\ten{T}}(\mathcal{Z})\|_{F}^2&=\langle
\RR_{\Omega}\Pj_{\ten{T}}(\mathcal{Z}),\RR_{\Omega}\Pj_{\ten{T}}(\mathcal{Z})\rangle
\leq \frac{10}{3}\beta \log(n)\langle
\Pj_{\ten{T}}(\mathcal{Z}),\RR_{\Omega}\Pj_{\ten{T}}(\mathcal{Z})\rangle\\
&=\frac{10}{3}\beta \log(n)\langle
\Pj_{T}(\mathcal{Z}),\Pj_{\ten{T}}\RR_{\Omega}\Pj_{\ten{T}}(\mathcal{Z})\rangle\\
&\leq \frac{10}{3}\beta
\log(n)(1+\epsilon_{0})p\|\Pj_{\ten{T}}(\mathcal{Z})\|_{F}^{2}.
\end{align*}
It follows that $\|\RR_{\Omega}\Pj_{\ten{T}}\|\leq \sqrt{\frac{10}{3}\beta
\log(n)(1+\epsilon_{0})p}$ and
\begin{align*}
\|\RR_{\Omega}\Pj_{\ten{T}_{l}}\|\leq
\|\RR_{\Omega}(\Pj_{\ten{T}_{l}}-\Pj_{\ten{T}})\|+\|\RR_{\Omega}\Pj_{\ten{T}}\| \leq
\frac{10}{3}\beta \log(n)(1+\epsilon_{0})p^{\frac{1}{2}}.
\end{align*}
Moreover, the application of the triangle inequality yields that
\begin{align*}
&\|\Pj_{\ten{T}_{l}}-p^{-1}\Pj_{\ten{T}_{l}}\RR_{\Omega}\Pj_{\ten{T}_{l}}\|\\
=&\|\Pj_{\ten{T}_{l}}-\Pj_{\ten{T}}+\Pj_{\ten{T}}-p^{-1}\Pj_{\ten{T}}\RR_{\Omega}\Pj_{\ten{T}}+p^{-1}\Pj_{T}\Pj_{\Omega}\Pj_{\ten{T}}
-p^{-1}\Pj_{\ten{T}_{l}}\RR_{\Omega}\Pj_{\ten{T}}+p^{-1}\Pj_{\ten{T}_{l}}\RR_{\Omega}\Pj_{\ten{T}}-
p^{-1}\Pj_{\ten{T}_{l}}\RR_{\Omega}\Pj_{\ten{T}_{l}}\|\\
\leq &
\|\Pj_{\ten{T}_{l}}-\Pj_{\ten{T}}\|+p^{-1}\|\Pj_{\ten{T}}\RR_{\Omega}\Pj_{\ten{T}}-\Pj_{\ten{T}_{l}}\RR_{\Omega}\Pj_{\ten{T}}\|
+p^{-1}\|\Pj_{\ten{T}_{l}}\RR_{\Omega}\Pj_{\ten{T}_{l}}-\Pj_{\ten{T}_{l}}\RR_{\Omega}\Pj_{\ten{T}}\|+\|\Pj_{\ten{T}}-p^{-1}\Pj_{T}\RR_{\Omega}\Pj_{\ten{T}}\|\\
\leq & 4\epsilon_{0},
\end{align*}
which completes the proof. 
\end{proof}

\subsection{ Proof of Lemmas \ref{lem2} and \ref{lemn1} }
\begin{proof}[Proof of Lemma \ref{lem2}]
When $\beta_{l}=0$ one has $P_{l}=\Pj_{\ten{T}_{l}}(\mathcal{G}_{l})$. Thus
$\alpha_l$ can be expressed as
$$\alpha_{l}=\frac{\|\Pj_{\ten{T}_{l}}(\mathcal{G}_{l})\|_{F}^{2}}{\langle
\Pj_{\ten{T}_{l}}(\mathcal{G}_{l}),\RR_{\Omega}\Pj_{\ten{T}_{l}}(\mathcal{G}_{l})\rangle }.$$ Note that if
$\|\Pj_{\ten{T}_{l}}-p^{-1}\Pj_{\ten{T}_{l}}\RR_{\Omega}\Pj_{\ten{T}_{l}}\|\leq 4\epsilon_{0}$
is satisfied, then
\begin{equation*}
\|\Pj_{\ten{T}_{l}}\RR_{\Omega}\Pj_{\ten{T}_{l}}\|\leq
p\|\Pj_{\ten{T}_{l}}-p^{-1}\Pj_{\ten{T}_{l}}\RR_{\Omega}\Pj_{\ten{T}_{l}}\|+p\|\Pj_{\ten{T}_{l}}\|\leq
(1+4\epsilon_{0})p.
\end{equation*}
Consequently,
\begin{align*}
\langle \Pj_{\ten{T}_{l}}(\mathcal{G}_{l}),\RR_{\Omega}\Pj_{\ten{T}_{l}}(\mathcal{G}_{l})\rangle=\langle
\Pj_{T_{l}}(\mathcal{G}_{l}),\Pj_{\ten{T}_{l}}\RR_{\Omega}\Pj_{\ten{T}_{l}}(\mathcal{G}_{l})
\rangle\leq(1+4\epsilon_{0})p\|\Pj_{\ten{T}_{l}}(\mathcal{G}_{l})\|_{F}^2.
\end{align*}
On the other hand,
\begin{align*}
\|\Pj_{\ten{T}_{l}}(\mathcal{G}_{l})\|_{F}^2&=\langle
\Pj_{\ten{T}_{l}}(\mathcal{G}_{l}),(\Pj_{\ten{T}_{l}}-p^{-1}\Pj_{\ten{T}_{l}}\RR_{\Omega}\Pj_{\ten{T}_{l}})(\mathcal{G}_{l})\rangle+\langle
\Pj_{\ten{T}_{l}}(\mathcal{G}_{l}),p^{-1}\Pj_{\ten{T}_{l}}\RR_{\Omega}\Pj_{\ten{T}_{l}}(\mathcal{G}_{l})\rangle\\
&\leq 4\epsilon_{0}\|\Pj_{\ten{T}_{l}}(\mathcal{G}_{l})\|_{F}^2+p^{-1}\langle
\Pj_{\ten{T}_{l}},\RR_{\Omega}\Pj_{\ten{T}_{l}}(\mathcal{G}_{l})\rangle.
\end{align*}
Combining the above two inequalities together yields that
$$\frac{1}{(1+4\epsilon_{0})p}\leq \alpha_{l}\leq
\frac{1}{(1-4\epsilon_{0})p}.$$
It follows that
\begin{align*}
\|\Pj_{\ten{T}_{l}}-\alpha_{l}\Pj_{\ten{T}_{l}}\RR_{\Omega}\Pj_{\ten{T}_{l}}\|
&\leq \|\Pj_{\ten{T}_{l}}-p^{-1}\Pj_{\ten{T}_{l}}\RR_{\Omega}\Pj_{\ten{T}_{l}}\|+(\alpha_{l}-p^{-1})\|\Pj_{\ten{T}_{l}}\RR_{\Omega}\Pj_{\ten{T}_{l}}\|\\
& \leq
4\epsilon_{0}+\frac{4\epsilon_{0}(1+4\epsilon_{0})p}{(1-4\epsilon_{0})p}=\frac{8\epsilon_{0}}{1-4\epsilon_{0}}.
\end{align*}
 This completes the proof.
\end{proof}

\begin{proof}[Proof of Lemma \ref{lemn1}]
When $\beta_{l}\neq0,$ one has
$P_{l}=\Pj_{\ten{T}_{l}}(\mathcal{G}_{l})+\beta_{l}\Pj_{\ten{T}_{l}}(\mathcal{Q}_{l-1}).$ Then the  orthogonalization weight $\beta_{l}$ can be
bounded as follows
\begin{align*}
 |\beta_{l}|= &~ \Big|\frac{\langle \Pj_{\ten{T}_{l}}(\mathcal{G}_{l}),
\RR_{\Omega}\Pj_{\ten{T}_{l}}(\mathcal{Q}_{l-1})\rangle}{\langle
\Pj_{\ten{T}_{l}}(\mathcal{Q}_{l-1}),\RR_{\Omega}\Pj_{\ten{T}_{l}}(\mathcal{Q}_{l-1})\rangle}\Big|
\\
 \leq &~ \Big|\frac{\langle \Pj_{\ten{T}_{l}}(\mathcal{G}_{l}),
(\Pj_{\ten{T}_{l}}\RR_{\Omega}\Pj_{\ten{T}_{l}}-p\Pj_{\ten{T}_{l}})(\mathcal{Q}_{l-1})\rangle}{\langle
\Pj_{\ten{T}_{l}}(\mathcal{Q}_{l-1}),\RR_{\Omega}\Pj_{\ten{T}_{l}}(\mathcal{Q}_{l-1})\rangle}\Big|+\Big|\frac{\langle
p\Pj_{\ten{T}_{l}}(\mathcal{G}_{l}), \RR_{\ten{T}_{l}}(\mathcal{Q}_{l-1})\rangle}{\langle
\Pj_{\ten{T}_{l}}(\mathcal{Q}_{l-1}),\RR_{\Omega}\Pj_{\ten{T}_{l}}(\mathcal{Q}_{l-1})\rangle}\Big|\\
\leq &~
\frac{4\epsilon_{0}p}{(1-4\epsilon_{0})p}\frac{\|\Pj_{\ten{T}_{l}}(\mathcal{G}_{l})\|_{F}}{\|\Pj_{\ten{T}_{l}}(\mathcal{Q}_{l-1})\|_{F}}+
\frac{p}{(1-4\epsilon_{0})p}\frac{|\langle
p\Pj_{\ten{T}_{l}}(\mathcal{G}_{l}),
\Pj_{\ten{T}_{l}}(\mathcal{Q}_{l-1})\rangle|}{\|\Pj_{\ten{T}_{l}}(\mathcal{Q}_{l-1})\|_{F}}\\
\leq& ~ \frac{4k_{2}\epsilon_{0}}{(1-4\epsilon_{0})}+
\frac{k_{1}k_{2}}{(1-4\epsilon_{0})}.
\end{align*}
In order to bound $\alpha_{l},$ we need to bound
$\|\Pj_{\ten{T}_{l}}(\mathcal{G}_{l})\|_{F}$ in terms of
$\|\Pj_{\ten{T}_{l}}(\mathcal{Q}_{l-1})\|_{F}.$ First note that
\begin{align*}
|\beta_{l}\langle
\Pj_{\ten{T}_{l}}(\mathcal{G}_{l}),\Pj_{\ten{T}_{l}}(\mathcal{Q}_{l-1}) \rangle|
&=\Big|\frac{\langle \Pj_{\ten{T}_{l}}(\mathcal{G}_{l}),
P_{\Omega}\Pj_{\ten{T}_{l}}(P_{l-1})\rangle}{\langle
\Pj_{\ten{T}_{l}}(\mathcal{Q}_{l-1}),P_{\Omega}\Pj_{\ten{T}_{l}(\mathcal{Q}_{l-1})}\rangle}\langle
\Pj_{\ten{T}_{l}}(\mathcal{G}_{l}),\Pj_{\ten{T}_{l}}(\mathcal{Q}_{l-1}) \rangle\Big|\\
&\leq\frac{(1+4\epsilon_{0})p\|\Pj_{\ten{T}_{l}}(\mathcal{G}_{l})\|_{F}}{(1-4\epsilon_{0})p\|\Pj_{\ten{T}_{l}}(\mathcal{Q}_{l-1})\|_{F}}\Big|\langle
\Pj_{\ten{T}_{l}}(\mathcal{G}_{l}),\Pj_{\ten{T}_{l}}(\mathcal{Q}_{l-1}) \rangle\Big|\\
&\leq\frac{1+4\epsilon_{0}}{1-4\epsilon_{0}}\|\Pj_{T_{l}}(\mathcal{G}_{l})\|^2_{F}\frac{|\langle
\Pj_{\ten{T}_{l}}(\mathcal{G}_{l}),\Pj_{\ten{T}_{l}}(\mathcal{Q}_{l-1})
\rangle|}{\|\Pj_{\ten{T}_{l}}(\mathcal{G}_{l})\|_{F}\|\Pj_{\ten{T}_{l}}(\mathcal{Q}_{l-1})\|_{F}}\\
&\leq
\frac{k_{1}(1+4\epsilon_{0})}{1-4\epsilon_{0}}\|\Pj_{\ten{T}_{l}}(\mathcal{G}_{l})\|^{2}_{F}.
\end{align*}
Thus there holds
\begin{align*}
|\langle
\Pj_{\ten{T}_{l}}(\mathcal{Q}_{l}),\Pj_{\ten{T}_{l}}(\mathcal{G}_{l})\rangle|&=|\langle
\Pj_{\ten{T}_{l}}(\mathcal{G}_{l})+\beta_{l}\Pj_{\ten{T}_{l}}(\mathcal{Q}_{l-1}),\Pj_{\ten{T}_{l}}(\mathcal{G}_{l})\rangle|\\
&\geq
\|\Pj_{\ten{T}_{l}}(\mathcal{G}_{l})\|_{F}^{2}-|\langle\beta_{l}\Pj_{\ten{T}_{l}}(\mathcal{Q}_{l-1}),\Pj_{\ten{T}_{l}}(\mathcal{G}_{l})\rangle|
\geq
\Big(1-\frac{k_{1}(1+4\epsilon_{0})}{1-4\epsilon_{0}}\Big)\|\Pj_{\ten{T}_{l}}(\mathcal{G}_{l})\|^{2}_{F}.
\end{align*}
Moreover, by the Cauchy-Schwarz inequality we have
\begin{align*}
\Big(1-\frac{k_{1}(1+4\epsilon_{0})}{1-4\epsilon_{0}}\Big)\|\Pj_{\ten{T}_{l}}(\mathcal{G}_{l})\|^{2}_{F}\leq
|\langle \Pj_{\ten{T}_{l}}(\mathcal{Q}_{l}),\Pj_{\ten{T}_{l}}(\mathcal{G}_{l})\rangle| \leq
\|\Pj_{\ten{T}_{l}}(\mathcal{Q}_{l})\|_{F}\|\Pj_{\ten{T}_{l}}(\mathcal{G}_{l})\|_{F}.
\end{align*}
Therefore,
\begin{align*}
\|\Pj_{\ten{T}_{l}}(\mathcal{G}_{l})\|_{F} \leq
\frac{1}{1-\frac{k_{1}(1+4\epsilon_{0})}{1-4\epsilon_{0}}}\|\Pj_{\ten{T}_{l}}(\mathcal{G}_{l})\|_{F}
\end{align*}
Noting that
\begin{align*}
\alpha_{l} =\frac{\langle
\Pj_{\ten{T}_{l}}(\mathcal{G}_{l}),\Pj_{\ten{T}_{l}}(\mathcal{Q}_{l})\rangle}{\langle
\Pj_{\ten{T}_{l}}(\mathcal{Q}_{l}),P_{\Omega}\Pj_{\ten{T}_{l}}(\mathcal{Q}_{l})\rangle}
=p^{-1}+\frac{\langle \Pj_{\ten{T}_{l}}(\mathcal{G}_{l}),
(\Pj_{\ten{T}_{l}}-p^{-1}\Pj_{\ten{T}_{l}}P_{\Omega}\Pj_{\ten{T}_{l}})(\mathcal{Q}_{l})\rangle}{\langle
\Pj_{\ten{T}_{l}}(\mathcal{Q}_{l}),P_{\Omega}\Pj_{\ten{T}_{l}}(\mathcal{Q}_{l})\rangle},
\end{align*}
we have
\begin{align*}
|\alpha_{l}\cdot p-1|& \leq p\Big|\frac{\langle
\Pj_{\ten{T}_{l}}(\mathcal{G}_{l}),(\Pj_{\ten{T}_{l}}-p^{-1}\Pj_{\ten{T}_{1}}P_{\Omega}\Pj_{\ten{T}_{l}})(\mathcal{Q}_{l})\rangle}{\langle
\Pj_{\ten{T}_{l}}(\mathcal{Q}_{l}),\Pj_{\ten{T}_{l}}P_{\Omega}\Pj_{\ten{T}_{l}}(\mathcal{Q}_{l}) \rangle}\Big|
\leq \frac{4\epsilon_{0}}{(1-4\epsilon_{0})-k_{1}(1+4\epsilon_{0})}.
\end{align*}
 Thus the spectral norm of $\|\Pj_{\ten{T}_{l}}-\alpha_{l}
\Pj_{\ten{T}_{l}}\RR_{\Omega}\Pj_{\ten{T}_{l}}\|$ can be bounded as

\begin{align*}
\|\Pj_{\ten{T}_{l}}-\alpha_{l}\Pj_{\ten{T}_{l}}P_{\Omega}\Pj_{\ten{T}_{l}}\|
&\leq \|\Pj_{\ten{T}_{l}}-p^{-1}\Pj_{\ten{T}_{l}}P_{\Omega}\Pj_{\ten{T}_{l}}\|+(\alpha_{l}-p^{-1})\|\Pj_{\ten{T}_{l}}P_{\Omega}\Pj_{\ten{T}_{l}}\|\\
& \leq
4\epsilon_{0}+(\alpha_{l}p-1)\|p^{-1}\Pj_{\ten{T}_{l}}P_{\Omega}\Pj_{\ten{T}_{l}}\|\leq
4\epsilon_{0}+4\epsilon_{\alpha}(1+4\epsilon_{0}).
\end{align*} This completes the proof.
\end{proof}

\subsection{Proof of Lemma \ref{le1} }
\begin{lemma}[Bernstein's Inequality \cite{Recht2011}]\label{new6}
Let $X_{1},...,X_{L}\in \R^{n\times n}$ be independent zero mean random matrices of dimension $d_{1}\times d_{2}$. Suppose
\begin{equation*}
\rho_{k}^2=\max\left\{\|E [X_{k}X_{k}^{T}]\|,\|E [X_{k}^{T}X_{k}]\|\right\}
\end{equation*}
and $\|X_{k}\|\leq B$ almost surely for all $k.$ Then for any $\tau>0,$
\begin{equation*}
\mathbb{P}[\|\sum_{k=1}^{L}X_{k}\|>\tau]\leq(d_{1}+d_{2})\exp\left(\frac{-\tau^2/2}{\sum_{k=1}^{L}\rho_{k}^2+B\tau/3}\right).
\end{equation*}
\end{lemma}

\begin{proof}[Proof of Lemma \ref{le1}]
We begin the proof with the following decomposition:
\begin{equation}
\Pj_\ten{T}(\mathcal{Z}) =  \sum_{abc} \<\Pj_\ten{T}(\mathcal{Z}), \Eabc\> \Eabc =
\sum_{abc} \<\mathcal{Z}, \Pj_\ten{T}(\Eabc)\> \Eabc. \nonumber
\end{equation}
It follows that
\begin{equation}
\PT\PO\PT(\mathcal{Z})  = \sum_{k=1}^{m}
\<\mathcal{Z}, \PT(\Eabck)\>\PT(\Eabck), \nonumber
\end{equation}
where $(a_{k},b_{k},c_{k})$ are indices sampled from $\{1,...,n_{1} \}\times\{1,...,n_{2} \}\times\{1,...,n_{3} \}$ independently and
uniform with replacement.

Let $\TT_{a_kb_kc_k}$ be the linear operator which maps $\mathcal{Z}$ to
$\<\mathcal{Z}, \PT(\Eabck)\>\PT(\Eabck)$. It is not hard to see that $\TT_{a_kb_kc_k}$ is rank-1 linear operator with
$$\|\TT_{a_kb_kc_k}\| =\|\PT(\Eabck)\|^2_F\leq  \frac{2\mu r}{n_{(2)}},$$
where the inequality follows from Lemma \ref{lem6.3}. Since $\|\PT\| \leq 1$, it follows that
\begin{equation}
\left\|\TT_{a_{k}b_{k}c_{k}} -
\frac{1}{n_1n_2n_3}\PT\right\| \leq
\max\left\{\|\PT(\Eabck)\|^2_F,
\frac{1}{n_1n_2n_3}\right\} \leq  \frac{2\mu r}{n_{(2)}},
\nonumber
\end{equation}
where the first inequality uses the fact that if $\mtx{A}$ and
$\mtx{B}$ are positive semidefinite matrices, then $\|\mtx{A} -
\mtx{B}\| \leq \max\{\|\mtx{A}\|, \|\mtx{B}\|\}$. In addition, we have
\begin{align}
&\Big\|\E\Big[\Big(\TT_{a_{k}b_{k}c_{k}} - \frac{1}{n_1 n_2 n_3}\PT\Big)^2\Big]\Big\|\nonumber \\
=~&  \left\|\E\left[\|\PT(\Eabck)\|^2_F{\TT_{a_{k}b_{k}c_{k}}}\right] - \frac{2}{n_1 n_2 n_3}{\PT}\E({\TT_{a_{k}b_{k}c_{k}}}) + \frac{1}{n^2_1n^2_2n^2_3}{\PT}\right\| \nonumber\\
=~&   \left\|\E\left[\|\PT(\Eabck)\|^2_F{\TT_{a_{k}b_{k}c_{k}}}\right] - \frac{1}{n^2_1 n^2_2 n^2_3}{\PT}\right\| \nonumber\\
\leq~&\max\left\{\frac{2\mu r}{n_{(2)}}\left\|\E[\TT_{a_{k}b_{k}c_{k}}]\right\|,\frac{1}{n^2_1 n^2_2 n^2_3}\right\}\\
\leq~&    \frac{2\mu r}{n_{(1)}n^2_{(2)}n_3}. \nonumber
\end{align}

Noticing that $\Pj_{\ten{T}}\RR_{\Omega}\Pj_{\ten{T}}=\sum_{k=1}^m\TT_{a_{k}b_{k}c_{k}}$, taking $\tau
=\frac{m}{n_1n_2n_3}\sqrt{\frac{c_1\mu r
n_{(1)}n_{3}\beta \log(n_{(1)}n_{3})}{m}}$ in Lemma~\ref{new6} yields the result.
\end{proof}


\subsection{Proof of Lemma \ref{lemma2} }
\begin{proof}[Proof of Lemma \ref{lemma2}]
First note that
\begin{equation}
\frac{n_{1}n_{2}n_{3}}{m}\PO(\mathcal{Z}) - \mathcal{Z} = \frac{1}{m}\sum_{k=1}^{m}n_{1}n_{2}n_{3}\mathcal{Z}_{a_{k}b_{k}c_{k}}\Eabck-\mathcal{Z}.
\nonumber
\end{equation}
By Definition \ref{def5},
\begin{align}
&\left\| \frac{n_{1}n_{2}n_{3}}{m}\PO(\mathcal{Z}) - \mathcal{Z}\right\|\nonumber\\
=~& \left\|\overline{\frac{n_{1}n_{2}n_{3}}{m}\PO(\mathcal{Z}) - \mathcal{Z}}\right\|\nonumber\\
=~& \left\|
 \frac{1}{m}\sum_{k=1}^{m}n_{1}n_{2}n_{3}\mathcal{Z}_{a_{k}b_{k}c_{k}}\overline{\Eabck}-\overline{\mathcal{Z}}\right\|.\label{keq01}
\end{align}
Let
\begin{equation}
\overline{\mathcal{C}_{a_{k}b_{k}c_{k}}} = n_{1}n_{2}n_{3}\mathcal{Z}_{a_{k}b_{k}c_{k}}\overline{\Eabck}-\overline{\mathcal{Z}}.\nonumber
\end{equation}
It is easy to see that
$\E[\overline{\mathcal{C}_{a_{k}b_{k}c_{k}}}]= 0$.
By the definition of $\Eabck$ and the fact DCT is a unitary transform, a simple calculation yields that
\begin{align*}
\|\overline{\Eabck}\|\leq 1\quad\mbox{and}\quad \|\overline{\mathcal{Z}}\|\leq\sqrt{n_1n_2n_3}\|\mathcal{Z}\|_\infty.
\end{align*}
Hence, $\|\overline{\mathcal{C}_{a_{k}b_{k}c_{k}}}\|\leq\frac{3n_1n_2n_3}{2}\|\mathcal{Z}\|_\infty$ for $n_3\geq 2$.

To bound $\E[\overline{\mathcal{C}_{a_{k}b_{k}c_{k}}}^T\overline{\mathcal{C}_{a_{k}b_{k}c_{k}}}]$, we need to first bound $\sum_{abc}\mathcal{Z}_{abc}^2\overline{\Eabc}^T\overline{\Eabc}$. To this end, let $[\mu_1^{abc},\cdots,\mu_{n_3}^{abc}]^T$ be the DFT of the $(a,b)$-th tube of $\Eabc$. Then,
\begin{align*}
\overline{\Eabc}^T\overline{\Eabc} =
\begin{bmatrix}
(\mu_1^{abc})^2\be_b\be_b^T\\
&\ddots\\
&&(\mu_{n_3}^{abc})^2\be_b\be_b^T
\end{bmatrix}.
\end{align*}
Thus we have
\begin{align*}
\sum_{abc}\mathcal{Z}_{abc}^2\overline{\Eabc}^T\overline{\Eabc}=
\sum_b\begin{bmatrix}
\sum_{ac}\Zabc^2(\mu_1^{abc})^2\be_b\be_b^T\\
&\ddots\\
&&\sum_{ac}\Zabc^2(\mu_{n_3}^{abc})^2\be_b\be_b^T
\end{bmatrix},
\end{align*}
which implies
\begin{align*}
\|\sum_{abc}\mathcal{Z}_{abc}^2\overline{\Eabc}^T\overline{\Eabc}\|\leq &
\max_{b,i}\sum_{ac}\Zabc^2(\mu_{i}^{abc})^2\leq\|\mathcal{Z}\|_\infty^2\sum_{ac}(\mu_{i}^{abc})^2 \leq n_1\|\mathcal{Z}\|_\infty^2,
\end{align*}
where the last equality follows from
$
\sum_{c}(\mu_{i}^{abc})^2=1
$
due to that DCT is a unitary transform. It follows that
\begin{align*}
\|\E[\overline{\mathcal{C}_{a_{k}b_{k}c_{k}}}^T\overline{\mathcal{C}_{a_{k}b_{k}c_{k}}}]\| =&\|n_1^2n_2^2n_3^2\E[\Zabck^2\overline{\Eabc}^T\overline{\Eabc}]-\Z^T\Z\|\\
\leq & \max\lcb n_1^2n_2^2n_3^2\|\E[\Zabck^2\overline{\Eabc}^T\overline{\Eabc}]\|,\|\Z^T\Z\|\rcb\\
\leq & n_{(1)}^2n_{(2)}n_3\|\Z\|_\infty^2.
\end{align*}
Moreover, $\|\E[\overline{\mathcal{C}_{a_{k}b_{k}c_{k}}}^T\overline{\mathcal{C}_{a_{k}b_{k}c_{k}}}]\|$ can be bounded similarly.

Applying the Bernstein's inequality to \eqref{keq01} concludes the proof.
\end{proof}

\subsection{Proof of Lemma \ref{lemm5}}
\begin{proof}[Proof of Lemma \ref{lemm5}]
Denote $\mathcal{W}_{0}=\mathrm{H}_{\bm{r}}(p^{-1}R_{\Omega}(\mathcal{X}))$. Lemma \ref{lemma2} implies that
\begin{align*}
\|\mathcal{X}_{0}-\mathcal{X}\|=
\|\mathcal{X}_{0}-\mathcal{W}_{0}+\mathcal{W}_{0}-\mathcal{X}\|\leq
2 \|\mathcal{W}_{0}-\mathcal{X}\|\lesssim \sqrt{\frac{\beta n_{(1)}^2n_{(2)}n_3
\log(n_{(1)}n_3)}{m}}\|\mathcal{X}\|_{\infty}
\end{align*}
holds with high probability.
Consequently,
\begin{align*}
\|\mathcal{X}_{0}-\mathcal{X}\|_{F}=\|\overline{\mathcal{X}_{0}}-\overline{\mathcal{X}}\|_{F}\leq
\sqrt{n_{3}r}\|\mathcal{X}_{0}-\mathcal{X}\|\lesssim \sqrt{\frac{\beta rn_{(1)}^2n_{(2)}n_3^2
\log(n_{(1)}n_3)}{m}}\|\mathcal{X}\|_{\infty}.
\end{align*}
Therefore, plugging the joint incoherence condition
$\|\mathcal{X}\|_{\infty}\leq
\mu_{1}\sqrt{\frac{r}{n_{1}n_{2}n_{3}}}\|\mathcal{X}\|$  into the above inequality yields that
\begin{align*}
\|\mathcal{X}_{0}-\mathcal{X}\|_{F}&\lesssim\sqrt{\frac{\mu_{1}^2r^2\beta
n_{(1)}n_3\log(n_{(1)}n_3)}{m}}\|\mathcal{X}\|,
\end{align*}   which completes the proof.
\end{proof}

\subsection{Proof of Lemma \ref{lemm4}}
\begin{proof}[Proof of Lemma \ref{lemm4}]
Noting that
\begin{equation}
(\Pj_{\mathcal{U}}-\Pj_{\mathcal{U}_l})(\mathcal{Z}) =  \sum_{abc} \<(\Pj_{\mathcal{U}}-\Pj_{\mathcal{U}_l})(\mathcal{Z}), \Eabc\> \Eabc =
\sum_{abc} \<\mathcal{Z}, (\Pj_{\mathcal{U}}-\Pj_{\mathcal{U}_l})(\Eabc)\> \Eabc, \nonumber
\end{equation}
we have
\begin{align*}
\Pj_{\ten{T}_{l}}\RR_{\Omega}(\Pj_{\mathcal{U}}-\Pj_{\mathcal{U}_l})(\mathcal{Z}) =\sum_{k=1}^{m}\<\mathcal{Z}, (\Pj_{\mathcal{U}}-\Pj_{\mathcal{U}_l})(\Eabck)\> \Pj_{\ten{T}_l}(\Eabck).
\end{align*}
Let $\TT_{a_kb_kc_k}$ be the linear operator which maps $\mathcal{Z}$ to
$\<\mathcal{Z}, (\Pj_{\mathcal{U}}-\Pj_{\mathcal{U}_l})(\Eabck)\> \Pj_{\ten{T}_l}(\Eabck)$.
It follows that
\begin{align*}
\|\TT_{a_kb_kc_k}\|\leq \|\Pj_{\ten{T}_l}(\Eabck)\|_F(\|\Pj_{\mathcal{U}}(\Eabck)\|_F+\|\Pj_{\mathcal{U}_l}(\Eabck)\|_F)\leq \frac{4\mu r}{n_{(2)}}.
\end{align*}
Moreover, there holds
\begin{align*}
&\|\E[(\TT_{a_kb_kc_k}-\E(\TT_{a_kb_kc_k}))^{T}(\TT_{a_kb_kc_k}-\E(\TT_{a_kb_kc_k}))]\|
=\|\E[(\TT^{T}_{a_kb_kc_k}\TT_{a_kb_kc_k})]-\E[\TT^{T}_{a_kb_kc_k}]\E[{\TT}_{a_kb_kc_k}]\|\\
\leq
&~\|\E[(\TT^{T}_{a_kb_kc_k}\TT_{a_kb_kc_k})]\|+\|\frac{1}{n_{1}^{2}n_{2}^{2}n_{3}^{2}}(\Pj_{\mathcal{U}}-\Pj_{\mathcal{U}_{l}})\Pj_{T_{l}}(\Pj_{\mathcal{U}}-\Pj_{\mathcal{U}_{l}})\|
\lesssim \frac{\mu r}{n_{(1)}n^{2}_{(2)}n_{3}}.
\end{align*}
Similarly, we have
\begin{align*}
\|\E[(\TT_{a_kb_kc_k}-\E(\TT_{a_kb_kc_k}))(\TT_{a_kb_kc_k}-\E(\TT_{a_kb_kc_k})^T)]\|\lesssim \frac{\mu r}{n_{(1)}n^{2}_{(2)}n_{3}}.
\end{align*}
Noting that $\Pj_{\ten{T}_{l}}\RR_{\Omega}(\Pj_{\mathcal{U}}-\Pj_{\mathcal{U}_l})=\sum_{k=1}^m\TT_{a_kb_kc_k}$, the lemma follows immediately from the Bernstein's inequality.
\end{proof}

\subsection{Proof of Lemma \ref{lemm2}}
\begin{proof}[Proof of Lemma \ref{lemm2}]
Similar to the proof for the matrix case  in \cite{wei2}, we have
\begin{align*}
&\|\mathcal{U}_{l}-\mathcal{U}\ast\mathcal{Q}\|_{F}^{2}=\langle\mathcal{U}_{l}-\mathcal{U}\ast\mathcal{Q},\mathcal{U}_{l}-\mathcal{U}\ast\mathcal{Q}\rangle=2r-2\langle\mathcal{U}_{l},\mathcal{U}\ast\mathcal{Q}\rangle,\\
&\|\mathcal{U}_{l}\ast\mathcal{U}^{T}_{l}-\mathcal{U}\ast\mathcal{U}^{T}\|^{2}_{F}=2r-2\langle\mathcal{U}_{l}\ast\mathcal{U}^{T}_{l},\mathcal{U}\ast\mathcal{U}^{T}\rangle.
\end{align*}
Then it suffices to show there exists an orthogonal tensor
$\mathcal{Q}\in\R^{r \times r \times n_3} $ such that
\begin{equation*}
\langle\mathcal{U}_{l},\mathcal{U}\ast\mathcal{Q}\rangle \geq
\langle\mathcal{U}_{l}\ast\mathcal{U}^{T}_{l},\mathcal{U}\ast\mathcal{U}^{T}\rangle,
\end{equation*}
which is equivalent to
$\langle\mathcal{U}\ast\mathcal{U}_{l},\mathcal{Q}\rangle \geq
\langle\mathcal{U}^{T}\ast\mathcal{U}_{l},\mathcal{U}^{T}\ast\mathcal{U}_{l}\rangle.$
Noting that
 $\mathcal{U}^{T}\ast\mathcal{U}_{l}=\mathcal{Q}_{1}\ast\Sigma
\ast\mathcal{Q}_{2}^T,$
we can choose
$\mathcal{Q}=\mathcal{Q}_{1}\ast\mathcal{Q}^T_{2},$  which completes the proof.
\end{proof}

\subsection{Proof of Lemma \ref{lemm3}}
\begin{proof}[Proof of Lemma \ref{lemm3}]
Let $d=\|\mathcal{X}_{l}-\mathcal{X}\|_{F}$. By Lemma \ref{lem5} we
have
\begin{align*}
\|\mathcal{U}_{l}\ast \mathcal{U}_{l}^{T}-\mathcal{U}\ast
\mathcal{U}^{T}\|_{F}\leq
\frac{\sqrt{2}d}{\sigma_{\min}(\mathcal{X})}
~~\text{and}~~\|\mathcal{V}_{l}\ast
\mathcal{V}_{l}^{T}-\mathcal{V}\ast \mathcal{V}^{T}\|_{F}\leq
\frac{\sqrt{2}d}{\sigma_{\min}(\mathcal{X})}.
\end{align*}
Moreover, by Lemma \ref{lemm2} there exists two unitary tensors
$\mathcal{Q}_{u}$ and $\mathcal{Q}_{v}$ such that
\begin{align*}
\|\mathcal{U}_{l}-\mathcal{U}\ast \mathcal{Q}_{u}\|_{F}\leq
\frac{\sqrt{2}d}{\sigma_{\min}(\mathcal{X})},~~\|\mathcal{V}_{l}-\mathcal{V}\ast
\mathcal{Q}_{v}\|_{F}\leq
\frac{\sqrt{2}d}{\sigma_{\min}(\mathcal{X})}.
\end{align*}
Noting that $\|\mathcal{X}_{l}-\mathcal{X}\|_{F}\leq
\frac{\sigma_{\min}(\mathcal{X})}{10\sqrt{2}}\leq
\frac{\sigma_{\max}(\mathcal{X})}{10\sqrt{2}}$ and
$\|\mathcal{X}_{l}\|\leq
\|\mathcal{X}+\mathcal{Z}_{l}-\mathcal{X}\|\leq
\sigma_{\max}(\mathcal{X})+d,$  we have
\begin{align*}
 &\|\mathcal{S}_{l}-\mathcal{Q}_{u}^{T}\ast \mathcal{S} \ast
\mathcal{Q}_{v}\|_{F}=\|\mathcal{U}_{l}^{T}\ast \mathcal{X}_{l}\ast
\mathcal{V}_{l}-(\mathcal{U}\ast \mathcal{Q}_{u})^{T}\ast
\mathcal{X}\ast(\mathcal{V}\ast \mathcal{Q}_{v})\|_{F}\\
\leq~&\|\mathcal{U}_{l}^{T}\ast \mathcal{X}_{l}\ast
\mathcal{V}_{l}-(\mathcal{U}\ast \mathcal{Q}_{u})^{T}\ast
\mathcal{X}_{l}\ast\mathcal{V}_{l}\|_{F}+\|(\mathcal{U}\ast
\mathcal{Q}_{u})^{T}\ast
\mathcal{X}_{l}\ast\mathcal{V}_{l}-(\mathcal{U}\ast
\mathcal{Q}_{u})^{T}\ast \mathcal{X}_{l}\ast \mathcal{V}_{l}\|_{F}\\
&+\|(\mathcal{U}\ast \mathcal{Q}_{u})^{T}\ast \mathcal{X}_{l}\ast
\mathcal{V}_{l}-(\mathcal{U}\ast \mathcal{Q}_{u})^{T}\ast
\mathcal{X}\ast(\mathcal{V}\ast \mathcal{Q}_{v})\|_{F}\\
\leq~&\|\mathcal{U}_{l}-\mathcal{U}\ast \mathcal{Q}_{u}\|_{F}\|\mathcal{X}_{l}\|+\|\mathcal{X}_{l}-\mathcal{X}\|_{F}+\|\mathcal{X}\|\|\mathcal{V}-\mathcal{V}\ast\mathcal{Q}_{v}\|_{F}\\
 \leq~& 4\kappa d,
\end{align*}
where $\kappa$ is the condition number of $\mathcal{X}.$  Recall
that $\mathcal{A}_{l}$ and $\mathcal{B}_{l}$ are defined as
\begin{align*}
\mathcal{A}_{l}^{[i]}=\frac{\mathcal{U}_{l}^{[i]}}{\|\mathcal{U}_{l}^{[i]}\|_{F}}\min\left(\|\mathcal{U}_{l}^{[i]}\|_{F},\sqrt{\frac{\mu_{0}r}{n_{1}}}\right),~~
\mathcal{B}_{l}^{[i]}=\frac{\mathcal{V}_{l}^{[i]}}{\|\mathcal{V}_{l}^{[i]}\|_{F}}\min\left(\|\mathcal{V}_{l}^{[i]}\|_{F},\sqrt{\frac{\mu_{0}r}{n_{2}}}\right).
\end{align*}
Together with
\begin{align*}
\|(\mathcal{U}\ast \mathcal{Q}_{u})^{[i]}\|_{F}\leq
\sqrt{\frac{\mu_{0}r}{n_{1}}},~\|(\mathcal{V}\ast
\mathcal{Q}_{v})^{[i]}\|_{F}\leq \sqrt{\frac{\mu_{0}r}{n_{2}}},
\end{align*}
we have
\begin{align*}
\|\mathcal{A}_{l}^{[i]}-(\mathcal{U}\ast
\mathcal{Q}_{u})^{[i]}\|_{F}\leq\|\mathcal{U}_{l}^{[i]}-(\mathcal{U}\ast
\mathcal{Q}_{u})^{[i]}\|_{F},~~\|\mathcal{B}_{l}^{[i]}-(\mathcal{V}\ast
\mathcal{Q}_{v})^{[i]}\|_{F}\leq\|\mathcal{V}_{l}^{[i]}-(\mathcal{V}\ast
\mathcal{Q}_{v})^{[i]}\|_{F}.
\end{align*}
It follows that
\begin{align*}
&\|\mathcal{A}_{l}-\mathcal{U}\ast\mathcal{Q}_{u}\|_{F}\leq
\|\mathcal{U}_{l}-\mathcal{U}\ast\mathcal{Q}_{u}\|_{F}\leq
\frac{\sqrt{2}d}{\sigma_{\min}(\mathcal{X})},~
\|\mathcal{B}_{l}-\mathcal{U}\ast\mathcal{Q}_{v}\|_{F}\leq
\|\mathcal{V}_{l}-\mathcal{U}\ast\mathcal{Q}_{v}\|_{F}\leq
\frac{\sqrt{2}d}{\sigma_{\min}(\mathcal{X})}.
\end{align*}
Thus together with $\hat{\mathcal{X}}=\mathcal{A}_{l}\ast \mathcal{S}_{l} \ast
\mathcal{B}^{\ast}_{l},$ we have
\begin{align*}
\|\hat{\mathcal{X}}-\mathcal{X}\|_{F}=&
\|\mathcal{A}_{l}\ast
\mathcal{S}_{l}\ast \mathcal{B}^{T}_{l}-(\mathcal{U}\ast
\mathcal{Q}_{u})\ast(\mathcal{Q}_{u}^{T}\ast \mathcal{S} \ast
\mathcal{Q}_{v})\ast (\mathcal{V}\ast \mathcal{Q}_{v})^{T}\|_{F}\\
\leq &~\|\mathcal{A}_{l}^{T}\ast \mathcal{S}_{l}\ast
\mathcal{B}^{T}_{l}-(\mathcal{U}\ast \mathcal{Q}_{u})\ast
\mathcal{S}_{l}\ast\mathcal{B}^{T}_{l}\|_{F}+\|(\mathcal{U}\ast
\mathcal{Q}_{u})\ast
\mathcal{S}_{l}\ast\mathcal{B}^{T}_{l}-(\mathcal{U}\ast
\mathcal{Q}_{u})\mathcal{Q}^{T}_{u}\ast \mathcal{S}\ast\mathcal{Q}_{v} \ast\mathcal{B}^{T}_{l}\|_{F}\\
&~+\|(\mathcal{U}\ast \mathcal{Q}_{u})\ast\mathcal{Q}^{T}_{u}\ast
\mathcal{S}\ast \mathcal{Q}_{v}\ast\mathcal{B}_{l}-(\mathcal{U}\ast
\mathcal{Q}_{u})\ast\mathcal{Q}^{T}_{u}\ast
\mathcal{S} \ast \mathcal{Q}_{v}\ast(\mathcal{V}\ast \mathcal{Q}_{v})\|_{F}\\
\leq &~
\|\mathcal{A}_{l}-\mathcal{U}\ast\mathcal{Q}_{u}\|_{F}\|\mathcal{S}_{l}\|\|\mathcal{B}_{l}\|+
\|\mathcal{S}_{l}-\mathcal{Q}^{T}_{u}\ast
\mathcal{S} \ast \mathcal{Q}_{v}\|_{F}\|\mathcal{B}_{l}\|+\|\mathcal{S}\|\|\mathcal{B}_{l}-\mathcal{V}\ast \mathcal{Q}_{v}\|_{F}\\
\leq &~ 8\kappa d.
\end{align*}
We also need to estimate the incoherence of $\hat{\mathcal{Z}}_{l}.$
Since $\mathcal{A}_{l}$ and $\mathcal{B}_{l}$ are not necessarily
unitary, we consider their QR factorizations:
\begin{equation*}
\mathcal{A}_{l}=\widetilde{\mathcal{U}}_{l}\ast \mathcal{R}_{u},~~
\mathcal{B}_{l}=\widetilde{\mathcal{V}}_{l}\ast \mathcal{R}_{v}.
\end{equation*}\
Noting that
\begin{align*}
&\sigma_{\min}(\mathcal{A}_{l})=\min(\sigma(\overline{\mathcal{A}}_{l}))\geq
1-\|\mathcal{A}_{l}-\mathcal{U}\ast\mathcal{Q}_{u}\|\geq
1-\frac{\sqrt{2}d}{\sigma_{\min}(\mathcal{X})}\geq\frac{9}{10},\\
&\sigma_{\min}(\mathcal{B}_{l})=\min(\sigma(\overline{\mathcal{B}}_{l}))\geq
1-\|\mathcal{B}_{l}-\mathcal{V}\ast\mathcal{Q}_{v}\|\geq
1-\frac{\sqrt{2}d}{\sigma_{\min}(\mathcal{X})}\geq\frac{9}{10}\
\end{align*}
by the assumption $d\leq \sigma_{\min}(\mathcal{X})/10\sqrt{2}$ and
the Weyl inequality, we have
$\|\mathcal{R}_{u}^{-1}\|\leq\frac{10}{9}$ and
$\|\mathcal{R}_{v}^{-1}\|\leq\frac{10}{9}.$ Consequently,
\begin{align*}
\|\hat{\mathcal{U}}_{l}\|_{F}=\|\widetilde{\mathcal{U}}_{l}\|_{F}=\|
\mathcal{A}_{l}\ast \mathcal{R}^{-1}_{u}\|_{F}\leq
\frac{10}{9}\sqrt{\frac{\mu_{0}r}{n_{1}}}~~\text{and}~~
\|\hat{\mathcal{V}}_{l}\|_{F}=\|\widetilde{\mathcal{V}}_{l}\|_{F}=\|
\mathcal{B}_{l}\ast \mathcal{R}^{-1}_{u}\|_{F}\leq
\frac{10}{9}\sqrt{\frac{\mu_{0}r}{n_{2}}}.
\end{align*}  This completes the proof.
\end{proof}
\noindent

\section*{Acknowledgement}
The authors would like to thank Dr. Ke Wei for his careful reading on the manuscript and
his comments and suggestions for improving the presentation of the manuscript.

\end{document}